\documentclass[UTF-8,reqno]{amsart}
\usepackage{enumerate}
\usepackage{mhequ}
\usepackage[margin=1.2in]{geometry}
\usepackage{amssymb,url,color, booktabs,nccmath}
\usepackage{mathrsfs}
\usepackage{enumitem}
\usepackage{graphicx}
\usepackage{tikz}
\usetikzlibrary{shapes,snakes}
\usetikzlibrary{calc}
\usetikzlibrary{decorations.shapes}
\usepackage{color}
\usepackage[colorlinks=true]{hyperref}
\hypersetup{
    linkcolor=blue,          
    citecolor=red,        
    filecolor=blue,      
    urlcolor=cyan
}

\definecolor{darkergreen}{rgb}{0.0, 0.5, 0.0}

\numberwithin{equation}{section}

\newcommand{\be}{\begin{eqnarray}}
\newcommand{\ee}{\end{eqnarray}}
\newcommand{\ce}{\begin{eqnarray*}}
\newcommand{\de}{\end{eqnarray*}}
\newtheorem{theorem}{Theorem}[section]
\newtheorem{lemma}[theorem]{Lemma}
\newtheorem{remark}[theorem]{Remark}
\newtheorem{definition}[theorem]{Definition}
\newtheorem{proposition}[theorem]{Proposition}
\newtheorem{Examples}[theorem]{Example}
\newtheorem{corollary}[theorem]{Corollary}
\newtheorem{assumption}{Assumption}[section]

\newenvironment{nouppercase}{%
  \renewcommand{\uppercasenonmath}[1]{}}{}

\def\[{{\Big[}}
\def\]{{\Big]}}
\def\<{{\langle}}
\def\>{{\rangle}}
\def\({{\Big(}}
\def\){{\Big)}}

\def\bx{{\mathbf{x}}}

\def\={&\!\!=\!\!&}

\def\1{{I}}

\def\geq{\geqslant}
\def\leq{\leqslant}
\def\ge{\geqslant}

\def\[{{\Big[}}
\def\]{{\Big]}}
\def\<{{\langle}}
\def\>{{\rangle}}
\def\({{\Big(}}
\def\){{\Big)}}

\def\bx{{\mathbf{x}}}

\def\={&\!\!=\!\!&}
\def\bt{\begin{theorem}}
\def\et{\end{theorem}}
\def\bl{\begin{lemma}}
\def\el{\end{lemma}}
\def\br{\begin{remark}}
\def\er{\end{remark}}
\def\bx{\begin{Examples}}
\def\ex{\end{Examples}}
\def\bd{\begin{definition}}
\def\ed{\end{definition}}
\def\bp{\begin{proposition}}
\def\ep{\end{proposition}}
\def\bc{\begin{corollary}}
\def\ec{\end{corollary}}

\def\geq{\geqslant}
\def\leq{\leqslant}
\def\ge{\geqslant}

\def\<{\langle} \def\>{\rangle}

\allowdisplaybreaks

\tikzset{
        dot/.style={circle,fill=black,inner sep=0pt, outer sep=0.7pt, minimum size=1mm},
        Phi/.style={white!40!red,thick,snake=coil,segment amplitude=0.6pt, segment length=2pt},
         Z/.style={black!40!green,thick,snake=coil,segment amplitude=0.6pt, segment length=2pt},
        C/.style={thick,black!20!blue},
          Cr/.style={thick,black!20!red},
            Cg/.style={thick,black!20!green},
       }

\usepackage[colorinlistoftodos,prependcaption,color=yellow,textsize=tiny]{todonotes}

\begin{document}

\title[Quantitative Error Estimates for Fluctuation Fields]
{\LARGE Quantitative Error Estimates for Learning Macroscopic Mobilities from Microscopic Fluctuations}

\author[Nicolas Dirr]{\large Nicolas Dirr}
\address[N. Dirr]{Cardiff School of Mathematics, Cardiff University, Cardiff, UK}
\email{dirrnp@cardiff.ac.uk} 

\author[Zhengyan Wu]{\large Zhengyan Wu}
\address[Z. Wu]{Department of Mathematics, Technische Universit\"at M\"unchen, Boltzmannstr. 3, 85748 Garching, Germany}
\email{wuzh@cit.tum.de}

\author[Johannes Zimmer]{\large Johannes Zimmer}
\address[J. Zimmer]{Department of Mathematics, Technische Universit\"at M\"unchen, Boltzmannstr. 3, 85748 Garching, Germany}
\email{jz@tum.de}

\begin{abstract}
We develop quantitative error estimates connecting microscopic fluctuation of interacting particle systems with the mobilities of their hydrodynamic limits. Focusing on the Symmetric Simple Exclusion Process and systems of independent Brownian particles, we provide explicit bounds for the discrepancy between the quadratic variation of fluctuation fields and the corresponding mobilities, in terms of time and spatial discretization parameters. In addition, we establish analogous error estimates for a class of fluctuating hydrodynamic stochastic PDEs with regularized coefficients. For stochastic PDEs with irregular square-root type coefficients, including Dean-Kawasaki type equations, we further identify the asymptotic behavior of the associated fluctuation structures within the framework of renormalized kinetic solutions. Our results provide quantitative insights into the relationship between microscopic fluctuation mechanisms and macroscopic mobilities, and contribute to a structured comparison between discrete particle systems and continuum fluctuating hydrodynamic descriptions.
\end{abstract}


\date{\today}

\begin{nouppercase}
\maketitle
\end{nouppercase}

\setcounter{tocdepth}{1}
\tableofcontents

\section{Introduction}\label{sec-1}
In the macroscopic regime, the evolution of physical systems is often described by partial differential equations (PDEs), while at the microscopic level, one observes the dynamics of interacting particles, such as molecules. This dichotomy naturally leads to the question of how microscopic particle systems can give rise to macroscopic PDE models. While there is a rich body of work on error estimates for approximating PDEs via numerical schemes, much less attention has been devoted to quantifying the error between interacting particle system models and their corresponding macroscopic PDE descriptions. In particular, error estimates concerning certain observables defined from the particle dynamics and their macroscopic counterparts remain relatively scarce in the existing literature. 

In this paper, we investigate a class of PDEs exhibiting the following gradient flow structure:
\begin{equation}\label{PDE-0-intro}
\partial_t\bar{\rho}=-\frac{1}{2}\nabla\cdot\Big(m(\bar{\rho})\nabla\frac{\delta\mathcal{F}(\bar{\rho})}{\delta\bar{\rho}}\Big),\quad \bar{\rho}(0)=\rho_0, 
\end{equation}
where $\mathcal{F}$ denotes an energy functional and $\rho_0$ denotes the initial datum, and the operator $-\nabla\cdot(m(\bar{\rho})\nabla\cdot)$ represents the mobility operator associated with a given mobility function $m(\cdot)$. PDEs of the form \eqref{PDE-0-intro} arise naturally in the macroscopic description of various interacting particle systems. For instance, when 
\begin{align}\label{coefficient-ssep}
m(\bar{\rho})=\bar{\rho}(1-\bar{\rho}) \text{ and }  \mathcal{F}(\bar{\rho})=\int\bar{\rho}\log\bar{\rho}+(1-\bar{\rho})\log(1-\bar{\rho})\,dx, 
\end{align}
equation \eqref{PDE-0-intro} reduces to a diffusion equation, which can be interpreted as the hydrodynamic limit of the symmetric simple exclusion process (SSEP). 

The primary objective of this work is to investigate how the mobility operator $-\nabla\cdot(m(\bar{\rho})\nabla\cdot)$ describes the diffusion of the fluctuation field arising from the underlying microscopic particle models. In particular, we aim to derive quantitative error estimates that relate microscopic observables to their macroscopic counterparts encoded in the PDE structure.

Precisely speaking, for each $N \in \mathbb{N}_+$, let $\eta_t$ denote the SSEP on the discrete torus $\mathbb{T}^d_N$. The fluctuation field around its hydrodynamic limit $\bar{\rho}$ (the solution to \eqref{PDE-0-intro} with coefficients defined by \eqref{coefficient-ssep}) is given by
\begin{align*}
\bar{X}^{1,N}(t, dy) := N^{-d/2} \sum_{x \in \mathbb{T}^d_N} \big(\eta_t(x) - \bar{\rho}(x,t)\big)\, \delta_{x}(dy).
\end{align*}
For all $t > h > 0$ and for every test function $\phi \in C^{\infty}(\mathbb{T}^d)$, we aim to establish an error estimate of the form
\begin{align}\label{error-intro}
\Bigg|\frac{1}{h} \mathbb{E}\big(\bar{X}^{1,N}_1(t+h,\phi)& - \bar{X}^{1,N}_1(t,\phi)\big)^2 - \langle \nabla\phi, m(\bar{\rho}) \nabla\phi \rangle\Bigg| \leq C(\rho_0,\phi)\, \text{Error}(N,h),
\end{align}
where $\text{Error}(N,h)$ denotes the error depending on the grid size $1/N^d$ and time discretization $h$, and satisfies
$$
\limsup_{N \rightarrow \infty}\limsup_{h \rightarrow 0} \text{Error}(N,h) = 0.
$$

These advancements have established a rigorous framework for quantifying the errors in the limiting behavior conjectured in \cite{EDZ18}. This framework enables precise, finite-size-controlled reconstruction of mobility or diffusivity from non-equilibrium particle data, going beyond traditional equilibrium-based methods. Such error estimates are valuable in the context of macroscopic fluctuation theory (MFT) and non-equilibrium statistical mechanics \cite{BDGJL,Derrida}. By providing transport coefficients with rigorously quantified uncertainty, these estimates allow microscopic fluctuations to be systematically incorporated into MFT predictions with controlled error, evaluate model validity, and propagate microscopic noise into macroscopic predictions. As a result, this methodology not only enables the extraction of physically meaningful transport properties from particle systems, such as the SSEP, but also provides a principled approach to integrate finite-size and statistical uncertainties into theoretical and computational analyses of non-equilibrium systems, see \cite{EDZ18} for further details. 

On the other hand, as a continuous mesoscopic analogue of particle systems, the theory of fluctuating hydrodynamics provides a framework for capturing mesoscopic effects by introducing conservative stochastic PDE models. These models are driven by conservative noise terms designed to reproduce the correct fluctuation behavior of the underlying interacting particle systems. The equations of fluctuating hydrodynamics can be viewed as fluctuation corrections to the macroscopic PDEs. In particular, for PDEs with a gradient flow structure of the form \eqref{PDE-0-intro}, the fluctuation corrections are governed by the so-called fluctuation-dissipation relations, leading to the following stochastic PDE:
\begin{align}
\partial_t\rho_{\varepsilon}=-\frac{1}{2}\nabla\cdot\Big(m(\rho_{\varepsilon})\nabla\frac{\delta\mathcal{F}(\rho_{\varepsilon})}{\delta\rho_{\varepsilon}}\Big)-\varepsilon^{1/2}\nabla\cdot\Big(m^{1/2}(\rho_{\varepsilon})\circ\xi_{\delta(\varepsilon)}\Big),
\end{align}
where $\xi_{\delta(\varepsilon)}$ denotes a spatially correlated noise, $\circ$ represents Stratonovich integration, and the parameters $\varepsilon$ and $\delta(\varepsilon)$ depend on $N$, effectively encoding the fluctuation amplitude and the grid scale of the underlying particle system. As a point of comparison with the particle systems, we also investigate error estimates of the form \eqref{error-intro}, concerning the discrepancy between the diffusion quadratic variation of the fluctuation field $\varepsilon^{-1/2}(\rho_{\varepsilon} - \bar{\rho})$ and its corresponding mobility operator. 

\subsection{Main results} 
The main results of this paper consist of four parts. The first three concern quantitative error estimates for Brownian particles, the SSEP, and a regularized fluctuating hydrodynamics SPDE. In the case of the unregularized fluctuating hydrodynamics SPDE, the irregularity of the square-root coefficient prevents us from deriving comparable error estimates. Instead, we employ the framework of renormalized kinetic solutions to obtain a qualitative asymptotic description for SPDEs with such irregular coefficients.

We begin with the error estimate for Brownian particles. For each $N \in \mathbb{N}_+$, let $(B_i)_{i=1}^N$ be a sequence of i.i.d.\ Brownian motions on the torus $\mathbb{T}^d$. Let $\pi_N$ denote the associated empirical measure, and let $\bar{\rho}$ be its mean-field limit, i.e., the solution of the diffusion equation. Define
$$
\bar{\pi}^1_N := \sqrt{N}\,(\pi_N - \bar{\rho}).
$$
Our first result is as follows.
\begin{theorem}\label{thm-1-intro}
Under the above setting, for every test function $\phi \in C^\infty(\mathbb{T}^d)$, let $\bar{\pi}^{1,\phi}_N := \bar{\pi}^1_N(\phi)$. Then there exists a constant $C = C(\phi, \bar{\rho}) > 0$ such that, for all $0 < h < t$,
\begin{align*}
\left| \frac{1}{h} \, \mathbb{E} \left( \bar{\pi}^{1,\phi}_N(t+h) - \bar{\pi}^{1,\phi}_N(t) \right)^2 - \left\langle \nabla \phi, \bar{\rho}(t) \nabla \phi \right\rangle \right| \leq C(\phi, \rho_0) \, h.
\end{align*}
\end{theorem}

For the error estimates of the SSEP, we recall that $\bar{\rho}$ denotes its hydrodynamic limit and $\bar{X}^{1,N}$ its fluctuation field. A more detailed definition together with the basic framework of the SSEP will be presented in Section \ref{sec-2}. The main results concerning the SSEP are stated as follows. 

\begin{theorem}\label{thm-2-intro}
	For every $0 < h < t < T$, we have
	\begin{align*}
	\Big|\frac{1}{h}\mathbb{E}(\bar{X}^{1,N}(t+h,\phi)-\bar{X}^{1,N}(t,\phi))^2& - \langle \nabla\phi, \bar{\rho}(1-\rho_0) \nabla\phi \rangle \Big|\leq  C(\bar{\rho},\phi)\left(h + hN^{d-4}+N^{-1}\right). 
	\end{align*}
\end{theorem}

We explain how these bounds arise. The first bound depends only on the time discretization parameter and is parallel to the error in Theorem~\ref{thm-1-intro}. The critical dimension $d=4$ appears through the second error term $h N^{d-4}$. This term originates from the competition between the renormalization factor $N^{d/2}$ in the fluctuation field $\bar{X}^{1,N}$ and the discretization error between the discrete Laplacian and the classical Laplacian. In particular, when $d>4$, this factor diverges as $N \to \infty$. Consequently, for applications, or from a practical point of view, a scaling relation between the time discretization parameter $h$ and the spatial discretization parameter $N$ is required in order to control the error in dimensions $d>4$. More precisely, one needs $h\,N(h)^{d-4} \to 0,\ \text{as } h \to 0$. The third bound arises from the Riemann sum error. 

In contrast with microscopic particle models, it is instructive to establish the identical results for the corresponding mesoscopic fluctuating hydrodynamics SPDEs. In the following, we consider 
\begin{equation}\label{SPDE-2-intro}
d\rho_{\varepsilon}=\frac{1}{2}\Delta\rho_{\varepsilon}dt-\varepsilon^{1/2}\nabla\cdot(\sigma(\rho_{\varepsilon})\circ dW_{\delta(\varepsilon)}), 	
\end{equation}
where $\sigma(\cdot)$ denotes a possibly irregular coefficient function, for example, $\sigma(\zeta)=\sqrt{\zeta}$ or $\sigma(\zeta)=\sqrt{\zeta(1-\zeta)}$. We begin by regularizing the irregular coefficient $\sigma(\cdot)$. Let $(\sigma_n(\cdot))_{n\geq1}$ be a sequence of smooth approximations of $\sigma(\cdot)$, and let $\rho_{\varepsilon}$ denote the solution of 
\begin{equation}\label{SPDE-3-intro}
d\rho_{\varepsilon}=\frac{1}{2}\Delta\rho_{\varepsilon}\,dt-\varepsilon^{1/2}\nabla\cdot\bigl(\sigma_{n(\varepsilon)}(\rho_{\varepsilon})\circ dW_{\delta(\varepsilon)}\bigr), 	
\end{equation}
Then we obtain the following error estimates.
\begin{theorem}\label{thm-3-intro}
Under appropriate assumptions on the initial datum $\rho_0$ and the coefficients $\sigma_n(\cdot)$, let $\phi \in C^{\infty}(\mathbb{T}^d)$ and define $\rho_{\varepsilon}^{\phi} := \langle \rho_{\varepsilon}, \phi \rangle$. The associated fluctuation field is then given by
\begin{align*}
\bar{\rho}_{\varepsilon}^{1,\phi} := \varepsilon^{-1/2} \big( \rho_{\varepsilon}^{\phi} - \bar{\rho}^{\phi} \big).
\end{align*}
Then, for every $t > h > 0$, the following estimate holds:
\begin{align*}
\left| \frac{1}{h} \mathbb{E} \left( \bar{\rho}_{\varepsilon}^{1,\phi}(t+h) - \bar{\rho}_{\varepsilon}^{1,\phi}(t) \right)^2 
- \left\langle \nabla \phi, \bar{\rho}(t)(1 - \bar{\rho}(t)) \nabla \phi \right\rangle \right|
\leq C(\phi, \rho_0, t) \cdot \mathrm{Error}(h, \varepsilon, \delta(\varepsilon), n(\varepsilon)),
\end{align*}
where the error term is given by
\begin{align*}
\mathrm{Error}(h, \varepsilon, \delta(\varepsilon), n(\varepsilon)) =\ & h 
+ \varepsilon \delta(\varepsilon)^{-2d} \|\sigma_{n(\varepsilon)}'\|_{L^{\infty}([0,1])}^4 h +\varepsilon \delta(\varepsilon)^{-d-2}. 
\end{align*}
\end{theorem}
Finally, we aim to investigate the renormalized kinetic solution to \eqref{SPDE-2-intro} as an analytically irregular object. Although a corresponding error estimate cannot be obtained, we are able to establish its asymptotic behavior. Let $\rho_{\varepsilon}$ denote the renormalized kinetic solution of \eqref{SPDE-2-intro} with coefficient $\sigma(\zeta)=\sqrt{\zeta}$ or $\sigma(\zeta)=\sqrt{\zeta(1-\zeta)}$. Then we obtain the following result.
\begin{theorem}\label{thm-4-intro}
Let $\rho_{\varepsilon}$ be the renormalized kinetic solution of \eqref{SPDE-2-intro}. Then, under the scaling regime $$\lim_{\varepsilon\rightarrow0}\varepsilon\delta(\varepsilon)^{-d-2}=0,$$ for every nonnegative test function $\phi \in C^{\infty}(\mathbb{T}^d)$ and every $t \in [0,T]$, the following asymptotic fluctuation identity holds:
\begin{align*}
	\lim_{\varepsilon \to 0} \liminf_{h \to 0} \frac{1}{h} \, \mathbb{E} \left( \left\langle \varepsilon^{-1/2}(\rho_{\varepsilon} - \bar{\rho})(t + h) - \varepsilon^{-1/2}(\rho_{\varepsilon} - \bar{\rho})(t), \phi \right\rangle \right)^2 = \langle \sigma(\bar{\rho}(t))^2 \nabla \phi, \nabla \phi \rangle.
\end{align*}
\end{theorem}

\subsection{Key ideas and technical comment}
This paper assembles three ideas with different techniques. 

\textbf{Error estimates. }In the case of the SSEP, we begin by exploiting its duality properties in order to derive a Duhamel-type representation for the associated empirical measure. The analysis of error terms relies on precise estimates for the quadratic variation of the fluctuation field as well as for the mobility matrix. These estimates are obtained by combining a characterization of the martingale quadratic variation through the carr\'e-du-champ operator with uniform moment bounds for the SSEP (see Lemma~\ref{error-rate} for details). The resulting bounds quantify the discretization errors both in time and in space. 

While quantitative fluctuation results for interacting particle systems have recently attracted increasing attention, the existing literature, for instance, Gess and Konarovskyi \cite{GK24}, does not directly address the type of estimates pursued in the present work. In particular, although Gess and Konarovskyi provide quantitative estimates for the full fluctuation field, our focus is on specific observables of the field, for which we aim to establish quantitative control of their short-time quadratic variation that explicitly depends on the time-discretization parameter. Their results do not directly reveal the dependence of constants on this time scale. 

Although, in principle, one might be able to adapt their generator-comparison approach to obtain such estimates, we instead adopt a more direct method by characterizing the martingale quadratic variation via the carr\'e-du-champ operator. This approach not only yields precise, time-local error bounds for the observables of interest but also allows natural extensions to analogous estimates for fluctuating hydrodynamics SPDEs.

Moreover, similar estimates for systems of independent Brownian particles and for the regularized fluctuating hydrodynamics equations can be established by applying the same characterization to the martingale terms arising in the corresponding fluctuation fields.

\textbf{Asymptotic behavior for irregular SPDEs. }  
In the setting of fluctuating hydrodynamics equations with irregular coefficients of square-root type, the mild formulation is no longer well-defined due to the lack of integrability of the highly irregular Stratonovich-to-It\^o correction terms. Instead, \cite{FG24} develop a framework for the well-posedness of such equations based on the notion of renormalized kinetic solutions. In this context, we decompose the fluctuation field into two contributions: 
\begin{align*}
&\frac{1}{h} \, \mathbb{E} \left( \left\langle \varepsilon^{-1/2}(\rho_{\varepsilon} - \bar{\rho})(t + h) - \varepsilon^{-1/2}(\rho_{\varepsilon} - \bar{\rho})(t), \phi \right\rangle \right)^2\\ 
\lesssim& \frac{1}{h} \, \mathbb{E} \left( \left\langle \varepsilon^{-1/2} \rho_{\varepsilon}(t + h) - \varepsilon^{-1/2} \rho_{\varepsilon}(t), \phi \right\rangle \right)^2 + \frac{1}{h} \, \mathbb{E} \left( \left\langle \varepsilon^{-1/2} \bar{\rho}(t + h) - \varepsilon^{-1/2} \bar{\rho}(t), \phi \right\rangle \right)^2, 
\end{align*}
where the second term can be estimated directly by exploiting the time-regularity of $\bar{\rho}$.  

Next, fix $\beta \in (0, 1/2)$ (to be specified later), and let $\eta_{\beta}$ denote a smooth truncation function such that $\eta_{\beta}(\zeta) =1$ for $\zeta \geq \beta$, $\eta_{\beta}(\zeta) =0$ for $\zeta \leq \beta/2$, and satisfying the derivative bound $\eta'_{\beta} \lesssim \beta^{-1} I_{\left[\frac{\beta}{2}, \beta\right]}$. Let $\chi_{\varepsilon} = I_{\{0 < \zeta < \rho_{\varepsilon}\}}$ denote the renormalized kinetic function. With this notation, we further decompose the first term as 
\begin{align*}
&\varepsilon^{-1/2} \left\langle \rho_{\varepsilon}(t + h) - \rho_{\varepsilon}(t), \phi \right\rangle\\ 
=& \varepsilon^{-1/2} \left\langle \chi_{\varepsilon}(t + h) - \chi_{\varepsilon}(t), \eta_{\beta} \phi \right\rangle_{L^2(\mathbb{T}^d\times\mathbb{R}_+)}
+ \varepsilon^{-1/2} \left\langle \chi_{\varepsilon}(t + h) - \chi_{\varepsilon}(t), (1 - \eta_{\beta}) \phi \right\rangle_{L^2(\mathbb{T}^d\times\mathbb{R}_+)}.
\end{align*}
Under this decomposition, one may apply the kinetic formulation to expand both terms. In the subsequent estimates, an explicit choice of $\beta$ depending on the time discretization parameter $h$ is required. By carefully analyzing the contributions of the kinetic measure together with the remainder terms in the kinetic formulation, one ultimately obtains the desired asymptotic behavior as $h \to 0$ and $\varepsilon \to 0$.

\subsection{Comment on the literature} \ 

{\bf The SSEP: fluctuations and quantitative estimates. }
The study of fluctuations in the SSEP has a long and rich history. The central limit theorem for the empirical density field was first established in \cite{R92}, where convergence to an infinite-dimensional Ornstein-Uhlenbeck process was proved. Kipnis and Varadhan \cite{KV86} established a central limit theorem for additive functionals of reversible Markov processes. Building on this framework, Yau and collaborators~\cite{Y91, CY92} developed the relative entropy method and the Boltzmann-Gibbs principle to derive hydrodynamic and fluctuation limits for conservative particle systems, including non-gradient models. Subsequent works by Jara and Gonçalves~\cite{J09, GJ14} extended these results to non-equilibrium and weakly asymmetric regimes, establishing connections between exclusion processes and nonlinear stochastic PDEs such as the KPZ equation. More recently, quantitative central limit theorems have been obtained for the SSEP~\cite{GK24}, providing optimal convergence rates toward the Gaussian fluctuation field. In addition, \cite{AZ24} studied fluctuations of the Symmetric Simple Inclusion Process.

{\bf Fluctuating hydrodynamics and conservative SPDEs. }
Macroscopic Fluctuation Theory (MFT) provides a unifying framework for studying non-equilibrium systems, extending the classical linear response theory valid near equilibrium (see Bertini et al. \cite{BDGJL}, Derrida \cite{Derrida}). At its heart lies an ansatz for the large deviation principles describing interacting particle systems. A natural companion to MFT is the theory of Fluctuating Hydrodynamics (FHD), which incorporates microscopic fluctuations in accordance with statistical mechanics and non-equilibrium thermodynamics. In the FHD setting, one postulates conservative stochastic PDEs to encode the essential features of non-equilibrium fluctuations (see \cite{LL87}, \cite{HS}). From this perspective, the MFT ansatz emerges as the zero-noise large deviation principle associated with such equations. A prototypical example is the Dean-Kawasaki equation (see \cite{D96,K98}), which captures the fluctuation behavior of mean-field interacting particle systems.

The mathematical analysis of the Dean-Kawasaki equation and related models has progressed substantially over the past decade. Konarovskyi, Lehmann and von Renesse \cite{KLR19}  established a rigorous martingale framework, thereby providing a mathematically justified formulation. Subsequently, Konarovskyi, Lehmann and von Renesse \cite{KLR20} extended the well-posedness result to the case of nonlocal interactions by employing a Girsanov transform.
Further studies such as Konarovskyi and M\"uller \cite{KF24}, Dello Schiavo and Konarovskyi \cite{DK25} and Dello Schiavo \cite{D25} investigated additional properties of the Dean--Kawasaki equation driven by space-time white noise. More recently, M\"uller, von Renesse and Zimmer \cite{FMJ25} extended the martingale approach to the Vlasov-Fokker-Planck-type Dean--Kawasaki equation. A key methodological advance came from Debussche and Vovelle \cite{DV10}, who introduced the notion of kinetic solutions for stochastic conservation laws. This framework was subsequently generalized to parabolic-hyperbolic SPDEs with conservative noise by Gess and Souganidis \cite{GS17}, Fehrman and Gess \cite{FG19}, and Dareiotis and Gess \cite{DG20}, and later adapted to Dean-Kawasaki type systems. For locally interacting models, Fehrman and Gess \cite{FG24} proved the well-posedness of functional-valued solutions with correlated noise, and in follow-up work \cite{FG23} established small-noise large deviation principles. Building on these results, Clini and Fehrman \cite{CF23} obtained a central limit theorem for the nonlinear Dean-Kawasaki equation with correlated noise, while Gess, the author, and Zhang \cite{GWZ24} derived higher-order fluctuation expansions. Additionally, Hao, the second author and the third author \cite{HWJ25} proved a strong well-posedness result for the Vlasov-Fokker-Planck-Dean-Kawasaki equation with correlated noise. 

For nonlocally interacting systems, Wang, the author, and Zhang \cite{WWZ22}, as well as the author and Zhang \cite{WZ24}, established well-posedness and large deviation principles for Dean-Kawasaki equations with singular interactions, including applications to the fluctuating Ising-Kac-Kawasaki equation. Recently, the second author \cite{W25} established the nonlinear fluctuations near criticality of the fluctuating Ising-Kac-Kawasaki equation. Related developments include the work of Martini and Mayorcas \cite{AA25,AA24}, who proved well-posedness and large deviation principles for an additive-noise approximation of the Keller-Segel Dean-Kawasaki equation. 

Other results for the Dean-Kawasaki equation with correlated noise include the recent works of Fehrman and Gess \cite{fehrman2024conservative} who consider the whole space case, and the work of Fehrman \cite{fehrman2025stochastic} who considers the bounded domains with homogeneous Neumann boundary data case.
Finally, we mention that Djurdjevac, Kremp and Perkowski \cite{djurdjevac2022weak} and Djurdjevac, Ji and Perkowski \cite{djurdjevac2025weak} show quantitative error estimates between the approximating SPDE and the empirical measure of the corresponding particle system. Further results concerning weak error estimates for the Dean-Kawasaki equation can be found in Cornalba and Fisher \cite{CF23-arma}, Cornalba, Fisher, Ingmanns and Raithel \cite{CFIR26}, and Magaletti, Gallo, Perez, Carrillo and Kalliadasis \cite{MGPCK22}. Furthermore, regarding long time behaviors, Fehrman, Gvalani and Gess \cite{fehrman2022ergodicity} studied ergodicity of the Dean-Kawasaki equation with regularized coefficients. Recently, Popat extended the well-posedness and large deviations theory for the Dean-Kawasaki equation to the case of bounded domains with Dirichlet boundary conditions \cite{Shyam25,Shyam25-fluc}. Subsequently, Popat and the second author \cite{PW25} established ergodicity for the Dean--Kawasaki equation with Dirichlet boundary conditions and square-root coefficients.

\subsection{Structure of the paper} 
The paper is organized as follows. In Section \ref{sec-2}, we introduce basic notations, provide a rigorous formulation of the SSEP, and define solution concepts for the fluctuating hydrodynamics equations. Section \ref{sec-3} begins with error estimates for independent Brownian particles and establishes Theorem \ref{thm-1-intro}. In Section \ref{sec-4}, we derive error estimates for the SSEP, thereby proving Theorem \ref{thm-2-intro}. Section \ref{sec-5} presents error estimates for the regularized fluctuating hydrodynamics SPDEs, leading to the proof of Theorem \ref{thm-3-intro}. Finally, Section \ref{sec-6} addresses fluctuating hydrodynamics SPDEs with irregular coefficients and establishes the asymptotic behavior described in Theorem \ref{thm-4-intro}.

\section{Preliminary}\label{sec-2}
\subsection{Notations}
Throughout the paper, we denote by $\mathbb{T}^d$ the $d$-dimensional torus normalized to have volume one. For each $N \in \mathbb{N}_+$, let $\mathbb{T}^d_N\subset\mathbb{T}^d$ denote the discrete $d$-dimensional torus consisting of $N^d$ sites. The operator $\nabla$ stands for the spatial gradient, and $\nabla\cdot$ denotes the divergence with respect to the space variable $x \in \mathbb{T}^d$. For any integer $p \in [1,\infty]$, we denote by $\|\cdot\|_{L^p(\mathbb{T}^d)}$ the norm in the Lebesgue space $L^p(\mathbb{T}^d)$ (or in $L^p(\mathbb{T}^d; \mathbb{R}^d)$ when vector-valued functions are considered). The inner product in $L^2(\mathbb{T}^d)$ is denoted by $\langle\cdot,\cdot\rangle$. 

We write $C^\infty(\mathbb{T}^d \times (0,\infty))$ for the space of infinitely differentiable functions defined on $\mathbb{T}^d \times (0,\infty)$, and $C^\infty_c(\mathbb{T}^d \times (0,\infty))$ for its subspace consisting of functions with compact support. Given a non-negative integer $k$ and $p \in [1,\infty]$, the Sobolev space $W^{k,p}(\mathbb{T}^d)$ denotes the space of functions on $\mathbb{T}^d$ with weak derivatives up to order $k$ in $L^p(\mathbb{T}^d)$. We define $H^a(\mathbb{T}^d) := W^{a,2}(\mathbb{T}^d)$ for $a \in \mathbb{R}$, and let $H^{-a}(\mathbb{T}^d)$ denote the dual space of $H^a(\mathbb{T}^d)$. The space $C^k_{\mathrm{loc}}(\mathbb{R})$ consists of all $k$-times continuously differentiable functions equipped with the topology of local convergence. The bracket $\langle\cdot,\cdot\rangle$ also denotes the dual pairing between $C^\infty(\mathbb{T}^d)$ and the space of distributions over $\mathbb{T}^d$.

Let $X$ be a real Banach space equipped with norm $\|\cdot\|_X$. For every $p \in [1,\infty]$, we denote by $L^p([0,T];X)$ the standard Bochner-Lebesgue space of $X$-valued functions, and by $W^{1,p}([0,T];X)$ the corresponding Sobolev space of functions whose weak derivatives in time also lie in $L^p([0,T];X)$.

In what follows, the notation $a \lesssim b$ for $a, b \in \mathbb{R}$ indicates that there exists a constant $\mathcal{D} > 0$, independent of any relevant parameters, such that $a \leq \mathcal{D}b$. The symbol $C$ is used generically to denote a positive constant which may vary from line to line and can be computed explicitly in terms of known quantities.

Furthermore, we fix $(\Omega,\mathcal{F},\mathbb{P})$ as a probability space.

\subsection{Basic definition of interacting particle systems} 
For each $N \in \mathbb{N}_+$, let $\Omega_N := \{0,1\}^{\mathbb{T}^d_N}$ denote the configuration space of the symmetric simple exclusion process (SSEP) on the discrete torus $\mathbb{T}^d_N$. The generator of the SSEP, acting on functions $f: \Omega_N \to \mathbb{R}$, is given by
\begin{align*}
(\mathcal{L}_N f)(\eta) = N^2 \sum_{x \in \mathbb{T}^d_N} \sum_{z \in \mathbb{T}^d_N} \eta(x)(1 - \eta(x+z)) p(z) \big[f(\eta^{x,x+z}) - f(\eta)\big],
\end{align*}
where $\eta^{x,x+z}$ denotes the configuration obtained from $\eta$ by exchanging the occupation variables at sites $x$ and $x+z$, and $p(z)$ is the symmetric jump probability kernel. If $\{\bar{e}_j\}_{j=1}^d$ denote the canonical basis vectors of $\mathbb{T}^d_N$, $p$ is given by 
\begin{align*}
p(z) = \frac{1}{2d}, \quad \text{if } z = \pm \bar{e}_j,\ j = 1, \ldots, d, \quad \text{and } p(z) = 0 \quad \text{otherwise}.
\end{align*}

For the reader's convenience, we now recall the notion of slowly varying initial data in the context of interacting particle systems.

\begin{definition}\label{def-initialdata}
Let $\rho_0: \mathbb{T}^d \to \mathbb{R}_+$ be a given function. We say that $\nu^N_{\rho_0(\cdot)}$ is a product measure with a slowly varying parameter associated with the macroscopic profile $\rho_0$ if the product measure $\nu^N_{\rho_0(\cdot)}$ on $\{0,1\}^{\mathbb{T}^d_N}$ satisfies
$$
\nu^N_{\rho_0(\cdot)} \bigl\{ \eta : \eta(x) = k \bigr\} = \nu^N_{\rho_0(x)} \bigl\{ \eta : \eta(0) = k \bigr\}, \quad \text{for all } x \in \mathbb{T}^d_N, \ k \in \{0,1\},
$$
where $\nu^N_{\rho_0(x)}$ denotes the Bernoulli distribution with parameter $\rho_0(x)$.

\end{definition}

In this framework, we fix a smooth function $\rho_0 \in C^{\infty}(\mathbb{T}^d;[0,1])$, and assume that the law of the initial configuration of the SSEP is given by the product Bernoulli measure $\nu^N_{\rho_0(\cdot)}$ with slowly varying parameter associated to $\rho_0$. 

We define the empirical measure associated with the configuration $\eta_t$ at time $t\in[0,T]$ as
\begin{align*}
\pi^N(\eta_t, dy) := N^{-d} \sum_{x \in \mathbb{T}^d_N} \eta_t(x)\, \delta_{x}(dy),
\end{align*}
where $\delta_{x}$ denotes the Dirac mass at $x \in \mathbb{T}^d_N$. For every test function $\phi: \mathbb{T}^d \to \mathbb{R}$ of class $C^\infty(\mathbb{T}^d)$, we define the action of the empirical measure on $\phi$ by
\begin{align*}
\pi^N(\eta_t, \phi) := N^{-d} \sum_{x \in \mathbb{T}^d_N} \eta_t(x)\, \phi(x).
\end{align*}

Let $\bar{\rho}$ denote the unique solution to the heat equation with initial condition $\rho_0$:
\begin{align}\label{heat-equation}
\partial_t \bar{\rho} = \frac{1}{2}\Delta \bar{\rho}, \quad \bar{\rho}(0) = \rho_0.
\end{align}

We now define the fluctuation field, which captures the microscopic deviations from the macroscopic hydrodynamic profile. For every $t \in [0,T]$, the fluctuation field is given in distributional form by
\begin{align*}
\bar{X}^{1,N}(t, dy) := N^{-d/2} \sum_{x \in \mathbb{T}^d_N} \big(\eta_t(x) - \bar{\rho}(x,t)\big)\, \delta_{x}(dy). 
\end{align*}
Correspondingly, for any test function $\phi \in C^\infty(\mathbb{T}^d)$, we define the action of the fluctuation field on $\phi$ by
\begin{align*}
\bar{X}^{1,N}(t, \phi) := N^{-d/2} \sum_{x \in \mathbb{T}^d_N} \big(\eta_t(x) - \bar{\rho}(x,t)\big)\, \phi(x).
\end{align*}

By applying Dynkin's formula to the empirical measure, we obtain that almost surely for every $t\in[0,T]$, 
\begin{align}\label{Dynkin-1}
\langle \pi^N_t, \phi \rangle = \langle \pi^N_0, \phi \rangle + \int_0^t \mathcal{L}_N \langle \pi^N_s, \phi \rangle\, ds + M^{N,\phi}_t,
\end{align}
where $M^{N,\phi}_t$ is a martingale with respect to the natural filtration $\{\mathcal{F}_t\}_{t \in [0,T]} := \{\sigma(\eta_s: s \leq t)\}_{t \in [0,T]}$.

From the definition of the generator, one can compute
\begin{align}\label{L-N}
\mathcal{L}_N \eta(x) =& \frac{N^2}{2} \sum_{j=1}^d \big[\eta(x+\bar{e}_j) + \eta(x-\bar{e}_j) - 2\eta(x)\big]. 
\end{align}

Let $\mathcal{P}_N$ denote the fundamental solution of the Chapman-Kolmogorov equation \cite{KL99}:  
\begin{align*}
\partial_t\mathcal{P}_N(t,\eta,\xi)=(\mathcal{L}_N)_{\eta}\mathcal{P}_N(t,\eta,\xi),\quad \mathcal{P}_N(0,\eta,\xi)=I_{\{\eta=\xi\}},\quad \xi,\eta\in\Omega_N. 	
\end{align*}
Define the associated semigroup $(P_N(t))_{t\in[0,T]}$ by
\begin{align}\label{PNt}
P_N(t)f(\eta):=\sum_{\xi\in\Omega_N}\mathcal{P}_N(t,\eta,\xi)f(\xi),\quad\forall t\in[0,T], 	
\end{align}
for all suitable functions $f:\Omega_N\rightarrow\mathbb{R}$.

Furthermore, for each $x\in\mathbb{T}^d_N$, consider the SSEP starting from the single-particle configuration $\delta_x$. If one tracks the position of the particle, denoted by the process $(X_N(t))_{t\in[0,T]}$, let $(G_N(t,x))_{t\in[0,T],\,x\in\mathbb{T}^d_N}$ denote the corresponding transition probabilities. In what follows, we present a lemma that characterizes the dual semigroup of $(P_N(t))_{t\in[0,T]}$.
\begin{lemma}\label{dual-semigroup}
Let $(\eta_t)_{t\in[0,T]}$ denote the SSEP with slowly varying initial data, as stated in Definition \ref{def-initialdata}, and let $(P_N(t))_{t\in[0,T]}$ be the semigroup defined in \eqref{PNt}. Then, for every $\varphi \in C^{\infty}(\mathbb{T}^d)$, we have that almost surely, for all $0 \leq s < t \leq T$,
\begin{align*}
\sum_{x\in\mathbb{T}^d_N} P_N(t-s)\eta_s(x) \, \delta_{x}(\varphi) = \sum_{x\in\mathbb{T}^d_N} \eta_s(x) \, \delta_{x}(S_N(t-s)\varphi),
\end{align*}
where $(S_N(t))_{t\in[0,T]}$ denotes the semigroup defined by
\begin{align}\label{semi-group-SN}
S_N(t)\varphi(y) = \sum_{x\in\mathbb{T}^d_N} G_N(t, y - x) \, \varphi(x), \quad y \in \mathbb{T}^d_N, \ t \in [0,T],
\end{align}
and $(G_N(t,x))_{t\in[0,T],\,x\in\mathbb{T}^d_N}$ is the transition probability of the independent random walk corresponding to the one-particle SSEP introduced above.
\end{lemma}

\begin{remark}
We remark that the discrete heat semigroup $S_N(t)$, as defined in \eqref{semi-group-SN}, can be naturally extended to the domain $\mathbb{T}^d$. More precisely, for any suitable function $\varphi:\mathbb{T}^d \to \mathbb{R}$, by applying a change of variables, we may extend the definition as  
\begin{align*}
S_N(t)\varphi(y) = \sum_{x \in \mathbb{T}^d_N} G_N(t, x)\, \varphi(y-x), 
\quad y \in \mathbb{T}^d,\ t \in [0,T].
\end{align*}
Furthermore, a straightforward computation shows that the generator of $S_N(t)$ is given by $\frac{1}{2}\Delta_N$, where $\Delta_N$ denotes the discrete Laplacian. 
\end{remark}

Before presenting the proof of Lemma~\ref{dual-semigroup}, we first recall some fundamental properties concerning the duality structure of the SSEP.

In the following, we define the SSEP restricted to configurations containing exactly $\kappa$ particles, for some integer $\kappa \in (0, N^d]$. Let
\begin{align*}
	(\mathcal{L}_{N,\kappa} f)(\xi) = N^2 \sum_{x\in\mathbb{T}^d_N} \sum_{z\in\mathbb{T}^d_N} \xi(x)(1 - \xi(x+z))p(z)\big[f(\xi^{x,x+z}) - f(\xi)\big],
\end{align*}
for every $\xi \in \Omega_{N,\kappa} := \{ \xi \in \Omega_N : |\xi| = \kappa \}$, where $|\xi| := \sum_{x \in \mathbb{T}^d_N} \xi(x)$, and $f : \Omega_{N,\kappa} \rightarrow \mathbb{R}$ is a suitable function.

We say that the SSEP is self-dual if there exists a function $D : \Omega_{N,\kappa} \times \Omega_N \to \mathbb{R}$ such that
$$
\mathcal{L}_{N} D(\xi,\eta) = \mathcal{L}_{N,\kappa} D(\xi,\eta), 
$$
for every integer $\kappa \in (0, N^d]$ and every $\xi \in \Omega_{N,\kappa}$, $\eta \in \Omega_N$, where $\mathcal{L}_N$ acts on the configuration $\eta$ and $\mathcal{L}_{N,\kappa}$ acts on the configuration $\xi$.

Let $(\eta_t)_{t\geq 0}$ and $(\xi_t)_{t\geq 0}$ be the Markov processes generated by $\mathcal{L}_N$ and $\mathcal{L}_{N,\kappa}$, respectively. Then the self-duality relation implies that
\begin{align}\label{self-duality}
\mathbb{E}_{\eta} D(\xi, \eta_t) = \mathbb{E}^{\kappa}_{\xi} D(\xi_t, \eta),
\end{align}
for every integer $\kappa \in (0, N^d]$ and every $\xi \in \Omega_{N,\kappa}$, $\eta \in \Omega_N$. Here, $\mathbb{E}_{\eta}$ denotes the expectation with respect to the law of $\eta(t)$ started from $\eta$, and $\mathbb{E}^{\kappa}_{\xi}$ denotes the expectation with respect to the law of $\xi(t)$ started from $\xi$.

\begin{lemma}
The SSEP is self-dual.
\end{lemma}

\begin{proof}
For every $\kappa \in (0, N^d]$, define the function $D : \Omega_{N,\kappa} \times \Omega_N \rightarrow \mathbb{R}$ by
\begin{align}\label{function-D}
D(\xi, \eta) = \prod_{x \in \mathbb{T}^d_N, \, \xi(x) = 1} I_{\{\xi(x) = \eta(x)\}}.
\end{align}
For every $\xi \in \Omega_{N,\kappa}$ and $\eta \in \Omega_N$, by the definition of the generator $\mathcal{L}_N$, we compute:
\begin{align*}
&\mathcal{L}_N D(\xi, \eta)
= \mathcal{L}_N \left( \prod_{x \in \mathbb{T}^d_N,\, \xi(x) = 1} \eta(x) \right) \\
&= N^2 \sum_{x \in \mathbb{T}^d_N} \sum_{z \in \mathbb{T}^d_N} \eta(x)(1 - \eta(x+z))p(z) \left[ \prod_{y \in \mathbb{T}^d_N,\, \xi(y) = 1} \eta^{x,x+z}(y) - \prod_{y \in \mathbb{T}^d_N,\, \xi(y) = 1} \eta(y) \right] \\
&= N^2 \sum_{x \in \mathbb{T}^d_N} \sum_{z \in \mathbb{T}^d_N} \eta(x)(1 - \eta(x+z))p(z) \left[ I_{\{\xi(x) = 0,\, \xi(x+z) = 1\}} - I_{\{\xi(x) = 1,\, \xi(x+z) = 0\}} \right]\prod_{y \in \mathbb{T}^d_N\slash\{x,z\},\, \xi(y) = 1} \eta(y) \\
&= N^2 \frac{1}{2d} \left| \left\{ (x,z) \in \mathbb{T}^d_N : \eta(x) = 1,\, \eta(x+z) = 0,\, \xi(x+z) = 1,\, \xi(x) = 0,\  \xi(y)=\eta(y),\ y\in\mathbb{T}^d_N\slash\{x,z\} \right\} \right| \\
&\quad - N^2 \frac{1}{2d} \left| \left\{ (x,z) \in \mathbb{T}^d_N : \eta(x) = 1,\, \eta(x+z) = 0,\, \xi(x) = 1,\, \xi(x+z) = 0,\ \xi(y)=\eta(y),\ y\in\mathbb{T}^d_N\slash\{x,z\} \right\} \right|.
\end{align*}
By the same reasoning, we obtain
\begin{align*}
&\mathcal{L}_{N,\kappa} D(\xi, \eta)\\
&= N^2 \frac{1}{2d} \left| \left\{ (x,z) \in \mathbb{T}^d_N : \xi(x) = 1,\, \xi(x+z) = 0,\, \eta(x+z) = 1,\, \eta(x) = 0,\ \xi(y)=\eta(y),\ y\in\mathbb{T}^d_N\slash\{x,z\} \right\} \right| \\
&\quad - N^2 \frac{1}{2d} \left| \left\{ (x,z) \in \mathbb{T}^d_N : \xi(x) = 1,\, \xi(x+z) = 0,\, \eta(x) = 1,\, \eta(x+z) = 0,\ \xi(y)=\eta(y),\ y\in\mathbb{T}^d_N\slash\{x,z\} \right\} \right| \\
&= \mathcal{L}_N D(\xi, \eta).
\end{align*}
This completes the proof.
\end{proof}

We are now ready to prove Lemma~\ref{dual-semigroup}.

\begin{proof}[Proof of Lemma~\ref{dual-semigroup}]
Fix $\xi = \delta_x$, where $\delta_x$ denotes the Dirac mass at $x$ viewed as a configuration in $\Omega_N$. Let $D$ be the duality function defined in \eqref{function-D}, and let $(X_N(t))_{t \in [0,T]}$ be the random walk associated with the one-particle SSEP.

To avoid ambiguity, we represent the semigroup $(P_N(t))_{t \in [0,T]}$ via the relation
\begin{align*}
P_N(t)f(\eta) = \mathbb{E}_{\eta} f(\hat{\eta}_t),
\end{align*}
where $(\hat{\eta}_t)_{t \in [0,T]}$ is an independent copy of the SSEP generated by $\mathcal{L}_N$ on another probability space $(\hat{\Omega}, \hat{\mathcal{F}}, \hat{\mathbb{P}})$.

Using the self-duality relation \eqref{self-duality}, we compute
\begin{align*}
P_N(t-s)\eta_s(x)
&= \hat{\mathbb{E}} \left[ \hat{\eta}_{t-s}(x) \,\big|\, \hat{\eta}_0 = \eta_s \right] \\
&= \hat{\mathbb{E}} \left[ D(\delta_x, \hat{\eta}_{t-s}) \,\big|\, \hat{\eta}_0 = \eta_s \right] \\
&= \hat{\mathbb{E}} \left[ D(\delta_{X_N(t-s)}, \eta_s) \,\big|\, X_0 = x \right] \\
&= \hat{\mathbb{E}}_x \, \eta_s(X_N(t-s)) 
= \sum_{y \in \mathbb{T}^d_N} \eta_s(y) G_N(t-s, x - y).
\end{align*}

Consequently, by changing variables, we obtain
\begin{align*}
\sum_{x \in \mathbb{T}^d_N} P_N(t-s)\eta_s(x) \, \delta_{x}(\varphi)
&= \sum_{x \in \mathbb{T}^d_N} \sum_{y \in \mathbb{T}^d_N} \eta_s(y) G_N(t-s, x - y) \, \delta_{x}(\varphi) \\
&= \sum_{x \in \mathbb{T}^d_N} \eta_s(x) \, \delta_{x}(S_N(t-s)\varphi).
\end{align*}
This concludes the proof.
\end{proof}

\begin{lemma}\label{Duhamel-formula}
Under the same assumptions as in Lemma~\ref{dual-semigroup}, the following Duhamel formula holds almost surely for every $t \in [0,T]$:
\begin{align*}
	\langle \pi^N_t, \phi \rangle 
	&= \langle P_N(t)\pi^N_0, \phi \rangle + \int_0^t \langle \phi, P_N(t-s) \, dM^N_s \rangle \\
	&= \langle \pi^N_0, S_N(t)\phi \rangle + \int_0^t \langle S_N(t-s)\phi, dM^N_s \rangle.
\end{align*}
\end{lemma}

\begin{proof}
By Dynkin's formula, the process
\begin{align}\label{martingale}
M_t = \pi^N_t - \pi^N_0 - \int_0^t \mathcal{L}_N \pi^N_s \, ds
\end{align}
defines a measure-valued martingale.

Applying the Duhamel formula to \eqref{martingale}, and testing against a smooth function $\phi$, we obtain
\begin{align*}
\langle \pi^N_t, \phi \rangle 
= \langle P_N(t)\pi^N_0, \phi \rangle + \int_0^t \langle \phi, P_N(t-s) \, dM^N_s \rangle.
\end{align*}

Next, by \eqref{martingale} again, we observe that the martingale $M_t$ depends linearly on the underlying SSEP configuration $\eta_t$. Therefore, we are allowed to apply the duality result stated in Lemma~\ref{dual-semigroup}, which relates the semigroup $P_N(t)$ acting on the configuration to the dual semigroup $S_N(t)$ acting on test functions.

Using Lemma~\ref{dual-semigroup}, we can exchange the semigroup and the test function, which yields
\begin{align*}
\langle \pi^N_t, \phi \rangle 
&= \langle P_N(t)\pi^N_0, \phi \rangle + \int_0^t \langle \phi, P_N(t-s) \, dM^N_s \rangle \\
&= \langle \pi^N_0, S_N(t)\phi \rangle + \int_0^t \langle S_N(t-s)\phi, dM^N_s \rangle.
\end{align*}
This completes the proof.
\end{proof}

\subsection{Basic setting of the SPDEs}
In Sections \ref{sec-5} and \ref{sec-6}, we study the fluctuating hydrodynamics of the SSEP. In preparation, this section introduces basic definitions of solutions to the relevant SPDEs, beginning with the description of the stochastic noise. Let $(e_k)_{k \geq 0}$ denote the standard Fourier basis of $L^2(\mathbb{T}^d)$, and let $m$ be a standard convolution kernel. For each $\delta > 0$, let $m_{\delta}$ be the standard mollifier on $\mathbb{T}^d$. Given $k \geq 0$ and $\delta > 0$, we define the mollified Fourier modes by $f_{\delta,k} := m_{\delta} \ast e_k$. 

We introduce the following Brownian motions:
\begin{align}\label{brownian}
W := \sum_{k \geq 0} e_k B_t^k \quad \text{and} \quad W_{\delta} := \sum_{k \geq 0} f_{\delta,k} B_t^k,
\end{align}
where $(B_t^k)_{k \geq 0}$ is a family of independent standard $d$-dimensional Brownian motions.

Furthermore, we define the associated coefficient functions:
\begin{equation*}
	F_{1,\delta} := \sum_{k \geq 0} |f_{\delta,k}|^2, \quad 
	F_{2,\delta} := \frac{1}{2} \sum_{k \geq 0} \nabla f_{\delta,k}^2, \quad 
	F_{3,\delta} := \sum_{k \geq 0} |\nabla f_{\delta,k}|^2.
\end{equation*}

It is important to note that, under this construction, $F_{1,\delta},F_{3,\delta}$ are constants, and we have the identity $F_{2,\delta} = 0$ for all $\delta > 0$. 

Let $\varepsilon>0$. As a continuous analogue of the SSEP, we investigate its associated fluctuating hydrodynamics equation in Section \ref{sec-5}, which is governed by the following stochastic PDE:
\begin{equation}\label{SPDE-1}
d\rho_{\varepsilon} = \frac{1}{2}\Delta\rho_{\varepsilon}\,dt - \varepsilon^{1/2} \nabla\cdot\Big(\sigma_{n(\varepsilon)}(\rho_{\varepsilon}) \circ dW_{\delta(\varepsilon)}\Big), \quad \rho_{\varepsilon}(0) = \rho_0, 
\end{equation}
where $\sigma_{n(\varepsilon)}(\cdot)$ is a smooth approximation of  $\sigma(\zeta)=\sqrt{\zeta(1-\zeta)},\ \zeta\in[0,1]$. We first introduce assumptions on the initial condition and the regularized mobility coefficient.

\begin{assumption}\label{Assump-initial-1}
We assume that the initial profile satisfies $\rho_0 \in C^{\infty}(\mathbb{T}^d)$ and is uniformly bounded away from the boundary values $0$ and $1$, i.e., there exists $c \in (0,1/2)$ such that $c \leq \rho_0 \leq 1 - c$ almost everywhere.
\end{assumption}

\begin{assumption}\label{Assump-initial-2}
We assume that the initial profile satisfies $\rho_0 \in C^{\infty}(\mathbb{T}^d)$ and is uniformly bounded away from the boundary values $0$, i.e., there exists $c \in (0,1/2)$ such that $\rho_0 \geq c$ almost everywhere.
\end{assumption}

Next, we formulate assumptions on the regularized mobility function, which approximates a singular square-root-type coefficient.

\begin{assumption}\label{Assump-sigman}
Suppose $\sigma(\cdot) \in C^{1}_{c}((0,1))$ and satisfies the following properties:
\begin{enumerate}
    \item[(1)] $\sigma \in C([0,1)) \cap C^{\infty}((0,1))$ with $\sigma(0)=\sigma(1)=0$ and $\sigma' \in C^{\infty}_c([0,1))$;
    \item[(2)] There exists a constant $c \in (0,\infty)$ such that for all $\zeta \in [0,1]$,
    \begin{align}
        |\sigma(\zeta)| \leq c, \quad \text{and} \quad |\sigma(\zeta)\sigma'(\zeta)| \leq c.
    \end{align}
\end{enumerate}
\end{assumption}

\begin{assumption}\label{Assump-sigman-sqrt}
Suppose $\sigma(\cdot) \in C^{1}_{c}((0,\infty))$ and satisfies the following properties:
\begin{enumerate}
    \item[(1)] $\sigma \in C([0,\infty)) \cap C^{\infty}((0,\infty))$ with $\sigma(0)=0$ and $\sigma' \in C^{\infty}_c([0,\infty))$;
    \item[(2)] There exists a constant $c \in (0,\infty)$ such that for all $\zeta \in [0,\infty)$,
    \begin{align}
        |\sigma(\zeta)| \leq c, \quad \text{and} \quad |\sigma(\zeta)\sigma'(\zeta)| \leq c.
    \end{align}
\end{enumerate}
\end{assumption}

The next lemma guarantees the existence of a regular approximation satisfying Assumption~\ref{Assump-sigman} or \ref{Assump-sigman-sqrt} for the canonical square-root coefficient.

\begin{lemma}\label{lem-sigma}
Let $\sigma(\zeta) = \sqrt{\zeta(1 - \zeta)}$ for all $\zeta \in [0,1]$. Then there exists a sequence of smooth functions $(\sigma_n(\cdot))_{n \geq 1}$ satisfying Assumption~\ref{Assump-sigman}, such that
$$
\sigma_n(\cdot) \longrightarrow \sigma(\cdot) \quad \text{in } C^1_{\mathrm{loc}}((0,1)), \quad \text{as } n \to \infty.
$$
Similarly, for the case $\sigma(\zeta) = \sqrt{\zeta}$ for all $\zeta \in [0,\infty)$, there exists a sequence of smooth functions $(\sigma_n(\cdot))_{n \geq 1}$ satisfying Assumption~\ref{Assump-sigman-sqrt}, such that
$$
\sigma_n(\cdot) \longrightarrow \sigma(\cdot) \quad \text{in } C^1_{\mathrm{loc}}((0,\infty)), \quad \text{as } n \to \infty.
$$

 \end{lemma}
Given the structure of the noise, the stochastic equation \eqref{SPDE-1} can be equivalently rewritten in the It\^o form:
\begin{align}\label{SPDE-1-ito}
d\rho_{\varepsilon} = \frac{1}{2}\Delta \rho_{\varepsilon}\,dt - \varepsilon^{1/2} \nabla\cdot\left(\sigma_{n(\varepsilon)}(\rho_{\varepsilon})\, dW_{\delta(\varepsilon)}\right) + \frac{\varepsilon}{2} \nabla\cdot\left(\sigma_{n(\varepsilon)}'(\rho_{\varepsilon})^2\, \nabla \rho_{\varepsilon} \right) F_{1,\delta(\varepsilon)}.
\end{align}

Let $\varepsilon, n(\varepsilon), \delta(\varepsilon) > 0$ be fixed parameters, and let $(S(t))_{t \in [0,T]}$ denote the heat semigroup generated by $\frac{1}{2}\Delta$ on $\mathbb{T}^d$. We now provide the definition of a mild solution to \eqref{SPDE-1-ito}.

\begin{definition}
Assume that $\rho_0$ and $\sigma_{n(\varepsilon)}(\cdot)$ satisfy Assumptions \ref{Assump-initial-1} and \ref{Assump-sigman}, respectively. An $L^2(\mathbb{T}^d)$-valued, $\mathcal{F}_t$-adapted process $\rho_{\varepsilon}$ is called a \emph{mild solution} to \eqref{SPDE-1-ito} with initial data $\rho_0$ if, almost surely, $\rho_{\varepsilon} \in C([0,T]; L^2(\mathbb{T}^d)) \cap L^2([0,T]; H^1(\mathbb{T}^d))$, and
\begin{equation} \label{mild-1}
\rho_{\varepsilon}(t) = S(t)\rho_0 
- \varepsilon^{1/2} \int_0^t S(t-s) \nabla\cdot(\sigma_{n(\varepsilon)}(\rho_{\varepsilon}) \, dW_{\delta(\varepsilon)}) 
+ \frac{\varepsilon}{2} \int_0^t S(t-s) \nabla \cdot \left( \sigma_{n(\varepsilon)}'(\rho_{\varepsilon})^2 \, \nabla \rho_{\varepsilon} \right) F_{1,\delta(\varepsilon)} \, ds,
\end{equation}
for every $t \in [0,T]$.
\end{definition}

\begin{theorem}
Assume that $\rho_0$ and $\sigma_{n(\varepsilon)}(\cdot)$ satisfy Assumptions \ref{Assump-initial-1} and \ref{Assump-sigman}, respectively. Then there exists a unique mild solution to \eqref{SPDE-1-ito} with initial data $\rho_0$.
\end{theorem}

\begin{proof}
The result follows by combining the arguments in \cite{FG24} with Duhamel's formula.
\end{proof}

Furthermore, in Section \ref{sec-6}, we investigate the SPDE 
\begin{equation}\label{SPDE-2}
d\rho_{\varepsilon}=\frac{1}{2}\Delta\rho_{\varepsilon}dt-\varepsilon^{1/2}\nabla\cdot(\sigma(\rho_{\varepsilon})\circ dW_{\delta(\varepsilon)}), 	
\end{equation}
where $\sigma(\cdot)$ denotes a possibly irregular coefficient function, for example, $\sigma(\zeta)=\sqrt{\zeta}$. This equation can be equivalently rewritten in It\^o form as
\begin{equation}\label{SPDE-irregular}
d\rho_{\varepsilon}=\frac{1}{2}\Delta\rho_{\varepsilon}dt-\varepsilon^{1/2}\nabla\cdot(\sigma(\rho_{\varepsilon}) dW_{\delta(\varepsilon)})+\frac{\varepsilon}{2}\nabla\cdot(\sigma'(\rho_{\varepsilon})^2\nabla\rho_{\varepsilon})F_{1,\delta(\varepsilon)}dt. 	
\end{equation}
Let $\chi_{\varepsilon}=I_{\{0<\zeta<\rho_{\varepsilon}\}}$ denote the associated renormalized kinetic function. Applying the chain rule formally, we obtain the following identity:
\begin{align}\label{kinetic-formula}
d\chi_{\varepsilon}=&\frac{1}{2}\nabla\cdot(\delta_0(\zeta-\rho_{\varepsilon})\nabla\rho_{\varepsilon})dt-\varepsilon^{1/2}\delta_0(\rho_{\varepsilon}-\zeta)\nabla\cdot(\sigma(\rho_{\varepsilon})dW_{\delta(\varepsilon)})\notag\\
&+\partial_{\zeta}p_{\varepsilon}dt-\frac{\varepsilon}{2}\partial_{\zeta}(\delta_0(\zeta-\rho_{\varepsilon})\sigma(\zeta)^2F_{3,\delta(\varepsilon)})dt+\frac{\varepsilon}{2}F_{1,\delta(\varepsilon)}\delta_0(\rho_{\varepsilon}-\zeta)\nabla\cdot(\sigma'(\rho_{\varepsilon})^2\nabla\rho_{\varepsilon})dt, 	
\end{align}
where the defect measure $p_{\varepsilon}$ is given by
\begin{align}\label{kineticmeasure-eq}
p_{\varepsilon}=\frac{1}{2}\delta_0(\zeta-\rho_{\varepsilon})|\nabla\rho_{\varepsilon}|^2.  	
\end{align}
We remark that this identity does not always valid. As observed in \cite{FG24}, the compactness of approximating sequences for \eqref{SPDE-2} cannot guarantee convergence of the right-hand side of \eqref{kineticmeasure-eq}. Instead, by invoking the lower semi-continuity of the $L^2(\mathbb{T}^d)$-norm, we are led to consider the distributional inequality
\begin{align}\label{kineticmeasure-ineq}
p_{\varepsilon}\geq\frac{1}{2}\delta_0(\zeta-\rho_{\varepsilon})|\nabla\rho_{\varepsilon}|^2.    	
\end{align}
This motivates a generalized definition of the kinetic measure. For clarity, we illustrate the notion using the concrete example $\sigma(\zeta)=\sqrt{\zeta}$, which corresponds to the Dean-Kawasaki equation. The case $\sigma(\zeta)=\sqrt{\zeta(1-\zeta)}$ associated with the fluctuating SSEP can be treated analogously.

\begin{definition}\label{def-kineticmeasure}
	 A \emph{kinetic measure} is a map $p$ from $\Omega$ into the space of nonnegative, locally finite measures on $\mathbb{T}^{d}\times(0,\infty)\times[0,T]$ such that for every test function $\psi\in C^{\infty}_c(\mathbb{T}^{d}\times (0,\infty))$, the mapping
	\begin{align*}
	(\omega,t)\in \Omega\times[0,T]\mapsto
	\int^t_0\int_{\mathbb{R}}\int_{\mathbb{T}^{d}}\psi(x,\zeta)\,\mathrm{d}p(x,\zeta,r)
	\end{align*}
	is $\mathcal{F}_t$-predictable.
\end{definition}

We now proceed to introduce the concept of renormalized kinetic solutions. 
\begin{definition}\label{def-kineticsolution} 
Let $\sigma(\zeta)=\sqrt{\zeta}$ and assume that $\rho_0$ satisfies Assumption \ref{Assump-initial-1}. For every $\varepsilon>0$, a nonnegative, $L^1(\mathbb{T}^{d})$-valued $\mathcal{F}_t$-progressively measurable process $\rho_{\varepsilon}$ is called a \emph{renormalized kinetic  solution} of \eqref{SPDE-irregular} with initial data $\rho_0$ if the following properties hold.
\begin{enumerate}
		\item Preservation of mass: almost surely for every $t\in[0,T]$,
		\begin{equation}\label{preservation-mass}
		\|\rho_{\varepsilon}(t)\|_{L^{1}(\mathbb{T}^{d})}=\|\rho_0\|_{L^{1}\left(\mathbb{T}^{d}\right)}.
		\end{equation}
		\item Entropy and energy dissipation estimate: there exists a constant $c\in(0,\infty)$, which depends on $T,\|\rho_0\|_{L^1(\mathbb{R}^{2d})}$, such that
		\begin{equation}\label{L2-es}
			\mathbb{E}\int_{0}^{T}\int_{\mathbb{T}^{d}}|\nabla \sqrt{\rho_{\varepsilon}(t)}|^2dzdt\leq c(T,\rho_0).
		\end{equation}
		Furthermore, there exists a nonnegative kinetic measure $p_{\varepsilon}$ satisfying the following properties.
		\item Regularity: in distributional sense, we have 
		\begin{align}\label{control KM}
		4\delta_{0}(\zeta-\rho_{\varepsilon})\zeta|\nabla \sqrt{\rho_{\varepsilon}}|^{2}\leq p_{\varepsilon}\ \text{on}\ \mathbb{T}^{d}\times(0,\infty)\times[0,T], \quad \mathbb{P}-a.s.
		\end{align}
		\item Vanishing at infinity:
		\begin{align}\label{control CKM VI}
		\liminf_{M\to \infty}\mathbb{E}\Big[p_{\varepsilon}(\mathbb{T}^{d}\times [0,T]\times [M,M+1])\Big]=0.
		\end{align}
		\item The equation: almost surely for every $t\in[0,T]$ and $\psi\in\mathrm{C}_{c}^{\infty}\left(\mathbb{T}^{d}\times(0,\infty)\right)$,
	\begin{align}\label{MC kenitic solution}
&\int_{\mathbb{R}}\int_{\mathbb{T}^{d}} \chi_{\varepsilon}(x, \zeta, t) \psi(x, \zeta)=\int_{\mathbb{R}} \int_{\mathbb{T}^{d}} \chi_{\varepsilon}(x, \zeta, 0)  \psi(x, \zeta)-\frac{1}{2}\int_{0}^{t} \int_{\mathbb{T}^{d}}\nabla\rho_{\varepsilon}\cdot(\nabla\psi)(x,\rho_{\varepsilon})\notag\\
&-\varepsilon^{1/2}\int^t_0\int_{\mathbb{T}^{d}}\psi(x,\rho_{\varepsilon})\nabla\cdot(\sigma(\rho_{\varepsilon})dW_{\delta(\varepsilon)})-\frac{\varepsilon}{2}\int^t_0\int_{\mathbb{T}^{d}}F_{1,\delta(\varepsilon)}\sigma'(\rho_{\varepsilon})^2\nabla\rho_{\varepsilon}\cdot(\nabla\psi)(x,\rho_{\varepsilon})\notag\\
&+\frac{\varepsilon}{2}\int^t_0\int_{\mathbb{T}^{d}}F_{3,\delta(\varepsilon)}\sigma^2(\rho_{\varepsilon})(\partial_{\zeta}\psi)(z,\rho_{\varepsilon})-\int^t_0\int_{\mathbb{R}}\int_{\mathbb{T}^{d}}\partial_{\zeta}\psi dp_{\varepsilon}. 
\end{align}
\end{enumerate}			
\end{definition}
The following well-posedness result is a straightforward consequence of \cite{FG24}.  
\begin{theorem}
Assume that $\rho_0$ satisfies Assumption \ref{Assump-initial-1}. Then there exists a unique renormalized kinetic solution of \eqref{SPDE-irregular} with initial data $\rho_0$. 
\end{theorem}

\section{From Brownian particles to the mobility matrix}\label{sec-3}
Before embarking on the analysis of the SSEP, we first examine a related setting involving independent Brownian particles. This preliminary investigation allows us to establish analogous quantitative error estimates and to verify the optimality of the discretization error in time. More precisely, in this section, we derive an estimate for the discrepancy between the quadratic variation of the fluctuation field associated with independent Brownian motions and the corresponding mobility matrix. This serves as a benchmark for later results in the interacting particle case.

\textbf{Initial data.} Let $(X_i)_{i \geq 1}$ be a sequence of i.i.d.\ $\mathbb{T}^d$-valued, $\mathcal{F}$-measurable random variables with common density $\rho_0 \in C^\infty(\mathbb{T}^d)$. We have that, for every test function $\phi \in C_b(\mathbb{T}^d)$, there exists a constant $C>0$, such that
\begin{align}\label{fluc-initial}
\mathbb{E}\Big[N^{1/2}\Big(\frac{1}{N} \sum_{i=1}^N \phi(X_i) - \langle \phi, \rho_0 \rangle\Big)\Big]^2=&NVar\Big(\frac{1}{N}\sum_{i=1}^N\phi(X_i)\Big)\notag\\
=&\frac{1}{N}\sum_{i=1}^NVar(\phi(X_i))\notag\\
=&\frac{1}{N}\sum_{i=1}^N\mathbb{E}(\phi(X_i)-\mathbb{E}\phi(X_i))^2\leq C\|\phi\|_{L^{\infty}(\mathbb{T}^d)}^2, 
\end{align}
uniformly for every $N\geq0$. 

For each $N \in \mathbb{N}_+$, let $(B_i)_{i=1}^N$ be a sequence of i.i.d. Brownian motions on the torus $\mathbb{T}^d$, with initial conditions $(X_i)_{i=1}^N$. Define the empirical measure associated with these particles by
\begin{align*}
\pi_N = \frac{1}{N} \sum_{i=1}^N \delta_{B_i}.
\end{align*}
An application of It\^o's formula yields that $\pi_N$ satisfies, in the sense of distributions,
\begin{align*}
d\pi_N = \frac{1}{2} \Delta \pi_N \, dt - \frac{1}{N} \nabla \cdot \left( \sum_{i=1}^N \delta_{B_i} \, dB_i \right).
\end{align*}
Let $\bar{\rho}$ denote the unique solution to the diffusion equation
\begin{align}\label{diffusion-eq}
\partial_t \bar{\rho} = \Delta \bar{\rho}, \quad \bar{\rho}(0) = \rho_0,
\end{align}
which serves as the macroscopic limit of the empirical measure $\pi_N$ as $N \to \infty$. 

We define the fluctuation field by $\bar{\pi}^1_N := N^{1/2}(\pi_N - \bar{\rho})$, which represents the deviation of the empirical measure from its deterministic mean-field limit, rescaled by the central limit theorem scaling. Then $\bar{\pi}^1_N$ satisfies the SPDE
\begin{align}\label{pi1n}
d\bar{\pi}^1_N = \frac{1}{2} \Delta \bar{\pi}^1_N \, dt - \frac{1}{\sqrt{N}} \nabla \cdot \left( \sum_{i=1}^N \delta_{B_i} \, dB_i \right),
\end{align}
in the weak (distributional) sense. Let $\{S(t)\}_{t \in [0,T]}$ denote the standard heat semigroup on $\mathbb{T}^d$, with associated smooth transition kernel $p_t(x)$ satisfying $S(t)f(x) = \int_{\mathbb{T}^d} p_t(x - y) f(y) \, dy$ for every $f \in C^\infty(\mathbb{T}^d)$. 

Applying the mild formulation of \eqref{pi1n} and using integration by parts along with the properties of the heat kernel, we obtain that, almost surely, for every $t \in [0,T]$,
\begin{align}\label{mild-0}
\bar{\pi}^1_N(t) =&S(t)(N^{1/2}(\pi_N(0)-\rho_0)) -\frac{1}{\sqrt{N}} \int_0^t S(t-s) \nabla \cdot \left( \sum_{i=1}^N \delta_{B_i} \, dB_i \right)\notag\\ 
=&S(t)(N^{1/2}(\pi_N(0)-\rho_0)) -\frac{1}{\sqrt{N}} \sum_{i=1}^N \int_0^t \nabla_x p_{t-s}(x - B_i(s)) \, dB_i(s).
\end{align}

For any test function $\phi \in C^\infty(\mathbb{T}^d)$, we denote by $\bar{\pi}^{1,\phi}_N := \bar{\pi}^1_N(\phi)$ the action of the fluctuation field on $\phi$. Then, from the mild formulation \eqref{mild-0}, we deduce that, almost surely, for every $t \in [0,T]$,
\begin{align*}
\bar{\pi}^{1,\phi}_N(t) 
=&\langle S(t)(N^{1/2}(\pi_N(0)-\rho_0)),\phi\rangle+ \frac{1}{\sqrt{N}} \sum_{i=1}^N \int_0^t \left\langle \nabla \phi, p_{t-s}(\cdot - B_i(s)) \right\rangle \, dB_i(s)\notag\\
=&\langle N^{1/2}(\pi_N(0)-\rho_0),S(t)\phi\rangle+ \frac{1}{\sqrt{N}} \sum_{i=1}^N \int_0^t (S(t-s)\nabla \phi)(B_i(s)) \, dB_i(s).
\end{align*}

We are now in a position to state the main quantitative error estimate concerning the quadratic variation of the fluctuation field $\bar{\pi}^1_N$ and its convergence to the limiting mobility matrix.  
\begin{theorem}
	Let $N \in \mathbb{N}_+$, and let $(B_i)_{i=1}^N$, $(X_i)_{i=1}^N$, $\bar{\rho}$, $\pi_N$, and $\bar{\pi}^1_N$ be defined as above. For every test function $\phi \in C^\infty(\mathbb{T}^d)$, then there exists a constant $C = C(\phi, \rho_0) > 0$ such that for all $0 < h < t$,
\begin{align*}
\left| \frac{1}{h} \, \mathbb{E} \left( \bar{\pi}^{1,\phi}_N(t+h) - \bar{\pi}^{1,\phi}_N(t) \right)^2 - \left\langle \nabla \phi, \bar{\rho}(t) \nabla \phi \right\rangle \right| \leq C(\phi, \rho_0) \, h.
\end{align*}
\end{theorem}
\begin{proof}
For every $0 < h < t$, we analyze the increment of the fluctuation field $\bar{\pi}^{1,\phi}_N$ using its mild representation \eqref{mild-0}. We begin by decomposing the time integral in the mild form of $\bar{\pi}^1_N(t+h)$ into two intervals: from $0$ to $t$, and from $t$ to $t+h$. By expanding the square of the difference and exploiting the independence of the Brownian motions, we obtain 
\begin{align*}
\mathbb{E} \left( \bar{\pi}_N^{1,\phi}(t+h) - \bar{\pi}_N^{1,\phi}(t) \right)^2 
=  \mathbb{E} \Bigg( &\big\langle N^{1/2}(\pi_N(0) - \rho_0), (S(t+h) - S(t))\phi \big\rangle \\
&+ \frac{1}{\sqrt{N}}\sum_{i=1}^N \Bigg( \int_0^t \left( S(t+h-s) \nabla \phi \right)(B_i) \, dB_i 
- \int_0^t \left( S(t-s) \nabla \phi \right)(B_i) \, dB_i \\
&\quad + \int_t^{t+h} \left( S(t+h-s) \nabla \phi \right)(B_i) \, dB_i \Bigg) \Bigg)^2\\
=&\mathbb{E}\Big(\big\langle N^{1/2}(\pi_N(0) - \rho_0), (S(t+h) - S(t))\phi \big\rangle\Big)^2\\
&+\frac{1}{N}\sum_{i=1}^N\mathbb{E}\left(\int^t_0([S(t+h-s)-S(t-s)]\nabla\phi)(B_i)\, dB_i\right)^2\\
&+\frac{1}{N}\sum_{i=1}^N\mathbb{E}\left(\int_t^{t+h} \left( S(t+h-s) \nabla \phi \right)(B_i) \, dB_i\right)^2. 
\end{align*} 
In the final equality above, we have employed the independence of the Brownian motions $(B_i)_{i=1}^N$ and the fact that the covariance between the stochastic integrals over the intervals $[0,t]$ and $[t, t+h]$ vanishes. Furthermore, to analyze the covariance between the initial data and the stochastic integrals, we utilized the properties of conditional expectation with respect to the filtration $\mathcal{F}_0$, from which it follows that these covariances are zero. By applying It\^{o} isometry, the expression reduces to 
\begin{align*}
\mathbb{E} \left( \bar{\pi}^{1,\phi}_N(t+h) - \bar{\pi}^{1,\phi}_N(t) \right)^2 
=& \mathbb{E}\Big(\big\langle N^{1/2}(\pi_N(0) - \rho_0), (S(t+h) - S(t))\phi \big\rangle\Big)^2\\
&+\frac{1}{N} \sum_{i=1}^N \mathbb{E} \int_0^t \left| \left(S(t+h-s)\nabla\phi\right)(B_i) - \left(S(t-s)\nabla\phi\right)(B_i) \right|^2 ds \\
&+ \frac{1}{N} \sum_{i=1}^N \mathbb{E} \int_t^{t+h} \left| \left(S(t+h-s)\nabla\phi\right)(B_i) \right|^2 ds.
\end{align*}
Rewriting this in terms of the empirical measure and taking expectations yields
\begin{align*}
\mathbb{E} \left( \bar{\pi}^{1,\phi}_N(t+h) - \bar{\pi}^{1,\phi}_N(t) \right)^2 
=&\mathbb{E}\Big(\big\langle N^{1/2}(\pi_N(0) - \rho_0), (S(t+h) - S(t))\phi \big\rangle\Big)^2\\
&+ \mathbb{E} \int_0^t \left\langle \left| S(t+h-s)\nabla\phi - S(t-s)\nabla\phi \right|^2, \pi_N(s) \right\rangle ds \\
&+ \mathbb{E} \int_t^{t+h} \left\langle \left| S(t+h-s)\nabla\phi \right|^2, \pi_N(s) \right\rangle ds.
\end{align*}
Using the fact that $\mathbb{E}\pi_N(t, \phi) = \langle \bar{\rho}(t), \phi \rangle$ holds for every test function $\phi \in C^{\infty}(\mathbb{T}^d)$ and all $t \in [0, T]$, we are therefore justified in replacing $\pi_N$ by $\bar{\rho}$ in expectation. Consequently, we obtain 
\begin{align*}
\mathbb{E} \left( \bar{\pi}^{1,\phi}_N(t+h) - \bar{\pi}^{1,\phi}_N(t) \right)^2 
=&\mathbb{E}\Big(\big\langle N^{1/2}(\pi_N(0) - \rho_0), (S(t+h) - S(t))\phi \big\rangle\Big)^2\\
&+\int_0^t \left\langle \left| S(t+h-s)\nabla\phi - S(t-s)\nabla\phi \right|^2, \bar{\rho}(s) \right\rangle ds \\
&+ \int_t^{t+h} \left\langle \left| S(t+h-s)\nabla\phi \right|^2, \bar{\rho}(s) \right\rangle ds.
\end{align*}
Exploiting the regularity and continuity properties of both the semigroup $S(t)$ and the macroscopic density profile $\bar{\rho}$, we estimate the first integral via a Taylor expansion in time and control the second using the temporal continuity of $\bar{\rho}(s)$. This yields the existence of a constant $C = C(\phi, \rho_0)$, depending only on the test function $\phi$ and the initial profile $\rho_0$, such that
\begin{align*}
	\int_0^t \left\langle \left| S(t+h-s)\nabla\phi - S(t-s)\nabla\phi \right|^2, \bar{\rho}(s) \right\rangle ds\leq C(\phi, \rho_0) h^2.
\end{align*}
As a consequence, we obtain the following bound on the quadratic variation:
\begin{align*}
\mathbb{E} \left( \bar{\pi}^{1,\phi}_N(t+h) - \bar{\pi}^{1,\phi}_N(t) \right)^2 
&\leq h \left\langle |\nabla\phi|^2, \bar{\rho}(t) \right\rangle +\mathbb{E}\Big(N^{1/2}\big\langle \pi_N(0) - \rho_0, (S(t+h) - S(t))\phi \big\rangle\Big)^2+ C(\phi, \rho_0) h^2 \\
&\quad + \int_t^{t+h} \left( \left\langle \left| S(t+h-s)\nabla\phi \right|^2, \bar{\rho}(s) \right\rangle 
- \left\langle |\nabla\phi|^2, \bar{\rho}(t) \right\rangle \right) ds.
\end{align*}
The last integral accounts for the local error arising from the time dependence of $\bar{\rho}(s)$ in a neighborhood of $s = t$, and we further decompose it as follows:
\begin{align*}
	\int^{t+h}_t\langle|S(t+h-s)\nabla\phi|^2,\bar{\rho}(s)\rangle-\langle|\nabla\phi|^2,\bar{\rho}(t)\rangle ds
	=&\int^{t+h}_t\langle|S(t+h-s)\nabla\phi|^2-|\nabla\phi|^2,\bar{\rho}(s)\rangle ds \\
	&+\int^{t+h}_t\langle|\nabla\phi|^2,(\bar{\rho}(s)-\bar{\rho}(t))\rangle ds \\
	\leq& C(\phi,\rho_0)h^2.
\end{align*}
Here, the second term on the right-hand side is controlled using the smoothness of $\bar{\rho}$. 

As for the initial data term, with the help of \eqref{fluc-initial}, there exists a constant $C>0$ such that
\begin{align*}
	\mathbb{E}\Big(N^{1/2}\big\langle \pi_N(0) - \rho_0, (S(t+h) - S(t))\phi \big\rangle\Big)^2 
	&\lesssim C\|(S(t+h) - S(t))\phi\|_{L^{\infty}(\mathbb{T}^d)}^2 \\
	&\leq C(\phi) \, h^2.
\end{align*}
This yields the estimate
\begin{align*}
\left| \frac{1}{h} \, \mathbb{E} \left( \bar{\pi}^{1,\phi}_N(t+h) - \bar{\pi}^{1,\phi}_N(t) \right)^2 - \left\langle \nabla \phi, \bar{\rho}(t) \nabla \phi \right\rangle \right| \leq C(\phi, \rho_0) \, h.
\end{align*}

\end{proof}

\section{From the SSEP to the mobility matrix}\label{sec-4}
Recall that $\rho_0 \in C^{\infty}(\mathbb{T}^d;[0,1])$, and let $(\eta_t)_{t\in[0,T]}$ denote the configuration of the SSEP with initial data as specified in Definition \ref{def-initialdata}, associated with $\rho_0$. Let $\bar{\rho}$ denote the unique solution to the diffusion equation \eqref{diffusion-eq} with initial data $\rho_0$. Furthermore, for each $N \in \mathbb{N}_+$, we recall that $\pi^N$ and $\bar{X}^{1,N}$ denote the empirical measure and the fluctuation field of the SSEP, respectively. 

In this section, we establish an error estimate between the quadratic variation of the fluctuation field and the corresponding mobility matrix. We begin by introducing Duhamel's formula and a characterization of the quadratic variation of the martingale term. 

Recall that $\{S_N(t)\}_{t \in [0,T]}$ is the semigroup defined by \eqref{semi-group-SN}. By Dynkin's formulas \eqref{Dynkin-1}, and upon applying the Duhamel's formula as shown in Lemma \ref{Duhamel-formula}, we obtain that almost surely, for every $\phi \in C^{\infty}(\mathbb{T}^d)$ and every $t\in[0,T]$, 
\begin{align}\label{mild-form-ips}
\langle \pi^N_t, \phi \rangle = \langle  \pi^N_0, S_N(t)\phi \rangle + \int_0^t \langle S_N(t - s) \, \phi,  dM^N_s \rangle,
\end{align}
where $\{M^N_t\}_{t \in [0,T]}$ is a martingale with respect to the filtration $\{\mathcal{F}_t\}_{t \in [0,T]} = \{\sigma(\eta_s : s \leq t)\}_{t \in [0,T]}$.

In the following, we consider the averaged empirical measure 
\begin{align}\label{discrete-heateq}
\rho_N(t) = \mathbb{E} \pi^N(t) = N^{-d} \sum_{x \in \mathbb{T}^d_N} \mathbb{E} \eta_t(x) \, \delta_{x}(dy), \quad t \in [0,T].
\end{align}
Then $\rho_N$ satisfies the discrete diffusion equation
\begin{align}\label{disc-diffusion-eq}
\partial_t \rho_N = \mathcal{L}_N \rho_N, \qquad \rho_N(0) = \mathbb{E} \pi^N_0.
\end{align}
We remark that the generator $\mathcal{L}_N$ is originally defined as an operator acting on functions of configurations. However, with a slight abuse of notation, we extend the definition of $\mathcal{L}_N$ to act on $\rho_N$ by employing the expression \eqref{L-N}. This expression coincides with the discrete Laplace operator and therefore can be applied to any suitably regular function.

To analyze the error estimate, we decompose the fluctuation field into two parts:
\begin{align*}
\bar{X}^{1,N}(t, dy) = \bar{X}^{1,N}_1(t, dy) + \bar{X}^{1,N}_2(t, dy),
\end{align*}
where
\begin{align}
\bar{X}^{1,N}_1(t, dy) &= N^{-d/2} \sum_{x \in \mathbb{T}^d_N} \big( \eta_t(x) - \mathbb{E}\eta_t(x) \big) \, \delta_{x}(dy), \label{X11}\\
\bar{X}^{1,N}_2(t, dy) &= N^{-d/2} \sum_{x \in \mathbb{T}^d_N} \big( \mathbb{E}\eta_t(x) - \bar{\rho}(x, t) \big) \, \delta_{x}(dy).\label{X12}
\end{align}
A straightforward computation shows that, almost surely, for every $\phi \in C^{\infty}(\mathbb{T}^d)$ and $t \in [0,T]$,
\begin{align*}
\bar{X}^{1,N}_1(t, \phi) = \langle N^{d/2}S_N(t)(\pi^N_0 - \mathbb{E} \pi^N_0), \phi \rangle + N^{d/2} \int_0^t \langle \phi, S_N(t - s) \, dM^N_s \rangle.
\end{align*}

We next present a basic characterization of the quadratic variation of the martingale term introduced above.

\begin{lemma}\label{lem-carre-du-champ}
Let $\phi \in C^{\infty}([0,T] \times \mathbb{T}^d)$ be a time-dependent test function. Then
\begin{align}\label{carre-du-champ}
\mathbb{E} \left( N^{d/2} \int_0^t \langle \phi(s), \, dM^N_s \rangle \right)^2 
= \mathbb{E} \left( \int_0^t \mathcal{L}_N \big( \bar{X}^{1,N}_1(s, \phi(s))^2 \big) 
- 2 \, \bar{X}^{1,N}_1(s, \phi(s)) \, \mathcal{L}_N \bar{X}^{1,N}_1(s, \phi(s)) \, ds \right).
\end{align}
\end{lemma}

\begin{proof}
The proof can be found in~\cite[Appendix~1.5, page~330, Lemma~5.1]{KL99}.
\end{proof}

\begin{remark}
The operator appearing on the right-hand side of \eqref{carre-du-champ} is commonly referred to as the \emph{carr\'e du champ} operator. More precisely, for each $N\in\mathbb{N}_+$, the \emph{carr\'e du champ} operator is defined by
\begin{align}\label{def-carre-du-champ}
\Gamma_N f(\eta) := (\mathcal{L}_N f^2)(\eta) - 2 f(\eta) \mathcal{L}_N f(\eta), \quad \text{for any } \eta \in \Omega_N,
\end{align}
where $f\colon \Omega_N \to \mathbb{R}$ denotes a suitable test function. In the following, we provide a characterization of the carr\'e du champ operator that will be instrumental in our analysis.
\end{remark}

\begin{lemma}
For every $N\in\mathbb{N}_+$, let $\Gamma_N$ be the carr\'e du champ operator defined by \eqref{def-carre-du-champ}. Then we have the following representation: 
\begin{align}\label{carre-property}
\Gamma_N f(\eta)=N^2 \sum_{x \in \mathbb{T}^d_N} \sum_{z \in \mathbb{T}^d_N} \eta(x)(1 - \eta(x+z)) p(z)(f(\eta^{x,x+z})-f(\eta))^2, \quad \text{for any } \eta \in \Omega_N,
\end{align}
where $f\colon \Omega_N \to \mathbb{R}$ denotes a suitable test function. 
\end{lemma}
\begin{proof}
Let $f:\Omega_N\rightarrow\mathbb{R}$ be a test function. By the definition of the generator, it follows that 
\begin{align*}
	\Gamma_N f(\eta) =& (\mathcal{L}_N f^2)(\eta) - 2 f(\eta) \mathcal{L}_N f(\eta)\\
	=&N^2 \sum_{x \in \mathbb{T}^d_N} \sum_{z \in \mathbb{T}^d_N} \eta(x)(1 - \eta(x+z)) p(z)(f(\eta^{x,x+z})^2-f(\eta)^2)\\
	&-2f(\eta)N^2 \sum_{x \in \mathbb{T}^d_N} \sum_{z \in \mathbb{T}^d_N} \eta(x)(1 - \eta(x+z)) p(z)(f(\eta^{x,x+z})-f(\eta))\\
	=&N^2 \sum_{x \in \mathbb{T}^d_N} \sum_{z \in \mathbb{T}^d_N} \eta(x)(1 - \eta(x+z)) p(z)(f(\eta^{x,x+z})-f(\eta))^2, \quad \text{for any } \eta \in \Omega_N. 
	\end{align*}
	This completes the proof. 
\end{proof}

As a preparation, we present a numerical error estimate between the discrete and continuous diffusion equations. Let $N \in \mathbb{N}_+$, and recall that $\rho_N$ denotes the solution to the discrete diffusion equation \eqref{disc-diffusion-eq} with initial condition $\rho_N(0) = \mathbb{E}\pi_N(0) = \sum_{x' \in \mathbb{T}^d_N} \mathbb{E}\eta_0(x') \delta_{x'}$. Under the assumption of slowly varying initial data, for every $x \in \mathbb{T}^d_N$ we have 
\begin{align*}
	\rho_N(0,x) = \mathbb{E}\pi_N(0,x) 
	= \sum_{x' \in \mathbb{T}^d_N} \rho_0(x') \delta_{x'}(x)
	= \rho_0(x).
\end{align*}
We now establish the following numerical error estimate.
\begin{lemma}
	For every $N \in \mathbb{N}_+$, it holds that 
	\begin{align}\label{disc-lap-error}
		\sup_{s \in [0,T],\, y \in \mathbb{T}^d_N} 
		|\rho_N(s,y) - \bar{\rho}(s,y)| 
		\leq C(\rho_0) N^{-2},
	\end{align}
	and 
	\begin{align}\label{disc-lap-error-2}
	N^d \sup_{s \in [0,T],\, y \in \mathbb{T}^d_N}\Big|\mathcal{L}_N \rho_N - \frac{1}{2}\Delta \bar{\rho}\Big|^2\leq C(\rho_0)N^{d-4}. 
	\end{align}

\end{lemma}
\begin{proof}
	Let $e_N = \rho_N - \bar{\rho}$. Taking the difference of the discrete and continuous equations yields 
	\begin{align*}
		\partial_t e_N = \mathcal{L}_N e_N + R_N,
	\end{align*}
	where $R_N = \mathcal{L}_N \bar{\rho} - \tfrac{1}{2}\Delta \bar{\rho}$. By the definition of $\mathcal{L}_N$ and applying a fourth order Taylor expansion of $\bar{\rho}$, we obtain 
	\begin{align*}
		\sup_{x \in \mathbb{T}^d_N} |R_N(x)|
		\leq \frac{1}{24 N^{2}}
		\sum_{j=1}^d \|\partial_{x_j}^4 \bar{\rho}\|_{L^{\infty}(\mathbb{T}^d)}.
	\end{align*}
	By applying Duhamel's formula, we complete the proof of \eqref{disc-lap-error}. Applying the same trick to $\mathcal{L}_N\rho_N-\frac{1}{2}\Delta\bar{\rho}$, we conclude the proof of \eqref{disc-lap-error-2} similarly. 
\end{proof}

In the following, we establish an error estimate comparing the jump rates of the SSEP with their corresponding expressions derived from the hydrodynamic limit. 

\begin{lemma}\label{error-rate}
Let $\rho_0 \in C^{\infty}(\mathbb{T}^d;[0,1])$. Consider the SSEP $(\eta_t)_{t \in [0,T]}$ with initial distribution associated with the profile $\rho_0$. For every $N\in\mathbb{N}_+$, let $\bar{\rho}$ denote the solution of \eqref{heat-equation} with initial datum $\rho_0$ and let $\rho_N$ be defined by \eqref{discrete-heateq}. Then for $t_1\in(0,T]$, the following estimate holds:
\begin{align*}
\sup_{s \in [t_1,T],\ x,z \in \mathbb{T}^d_N}\Big| \mathbb{E} \big( \eta_s(x)(1 - \eta_s(x+z)) \big)& - \bar{\rho}(s,x) \big( 1 - \bar{\rho}(s,x+z) \big) \Big|\\ 
\lesssim&\quad  N^{-2}+N^{-2}I_{\{d=1\}}+\frac{\log(1+N^2)}{1+N^2}I_{\{d=2\}}+\frac{1}{1+N^2}I_{\{d\ge3\}}, 
\end{align*}
for every $N \in \mathbb{N}_+$. 	
\end{lemma}
\begin{proof}
By the triangle inequality, we obtain
\begin{align*}
&\Big| \mathbb{E} \big( \eta_s(x)(1 - \eta_s(x+z)) \big) - \bar{\rho}(s,x) \big( 1 - \bar{\rho}(s,x+z) \big) \Big| \\
\leq\ & \Big| \mathbb{E} \eta_s(x) - \mathbb{E} \big( \eta_s(x) \eta_s(x+z) \big) + \rho_N(s,x) \rho_N(s,x+z) - \rho_N(s,x) \rho_N(s,x+z) \\
&\quad + \bar{\rho}(s,x) \bar{\rho}(s,x+z) - \bar{\rho}(s,x) \Big| \\
\leq\ & \big| \rho_N(s,x) - \bar{\rho}(s,x) \big| + \big| \mathbb{E} \big( \eta_s(x) \eta_s(x+z) \big) - \rho_N(s,x) \rho_N(s,x+z) \big| \\
&\quad + \big| \big( \rho_N(s,x) - \bar{\rho}(s,x) \big) \rho_N(s,x+z) \big| + \big| \bar{\rho}(s,x) \big( \rho_N(s,x+z) - \bar{\rho}(s,x+z) \big) \big| \\
\lesssim\ & \big| \rho_N(s,x) - \bar{\rho}(s,x) \big| + \big| \mathbb{E} \big( \eta_s(x) \eta_s(x+z) \big) - \rho_N(s,x) \rho_N(s,x+z) \big|.
\end{align*}
To estimate the first term, we apply \eqref{disc-lap-error}, which yields 
\begin{align*}
\sup_{s \in [0,T],\ x \in \mathbb{T}^d_N} \big| \rho_N(s,x) - \bar{\rho}(s,x) \big| \lesssim N^{-2}.
\end{align*}
To estimate the second term, we apply \cite[Theorem 2.1 and Theorem 2.2]{FPV91} and \cite[Theorem 2.3]{FGL00}, which yields
\begin{align}\label{covariance-es}
\sup_{s \in [t_1,T],\ x, z \in \mathbb{T}^d_N} \Big| \mathbb{E} \big( \eta_s(x) \eta_s(x+z) \big)& - \rho_N(s,x) \rho_N(s,x+z) \Big|\notag\\
 \lesssim& \quad N^{-2}I_{\{d=1\}}+\frac{\log(N^2)}{N^2}I_{\{d=2\}}+\frac{1}{N^2}I_{\{d\ge3\}}.
\end{align}
We remark that \cite[Theorem 2.1 and Theorem 2.2]{FPV91} provides a bound for the SSEP only in the one-dimensional case; however, the authors emphasize that these bounds can be improved to arbitrary higher dimensions, and that the one-dimensional case is the most difficult. Later, \cite[Theorem~2.3]{FGL00} establishes such bounds for arbitrary dimensions. In \cite[Theorem~2.3]{FGL00}, the authors work with the standard (or normal) time scale, and the estimates are expressed in terms of time. In the present context, however, we adopt the diffusive time scale; consequently, the time parameter $t$ in the estimates of \cite[Theorem~2.3]{FGL00} is replaced by $N^2 t$. Applying this rescaling, we obtain the estimate~\eqref{covariance-es}, which concludes the proof.
\end{proof}

Furthermore, we establish the following estimate for the Riemann sum approximation. 

\begin{lemma}
For every $\phi\in C^{\infty}(\mathbb{T}^d)$, the following bound holds:
\begin{align}\label{riemann-error}
	\left| \langle \nabla\phi, \bar{\rho}(1 - \bar{\rho}) \nabla\phi \rangle - N^{2-d} \sum_{x \in \mathbb{T}^d_N} \sum_{z \in \mathbb{T}^d_N} \bar{\rho}(x)(1 - \bar{\rho}(x+z))p(z)[\phi(x+z) - \phi(x)]^2 \right|
	\leq C(\rho_0,\phi) N^{-1}. 
\end{align}	
\end{lemma}

\begin{proof}
Fix $x \in \mathbb{T}^d_N$ and $z = \pm \bar{e}_j$ for $j=1,\dots,d$. By a Taylor expansion, we have
$$
\phi(x+z) - \phi(x) = \frac{1}{N} (\pm \partial_{x_j} \phi(x)) + \frac{1}{2N^2} \partial_{x_j}^2 \phi(\xi_{x,z}),
$$
for some $\xi_{x,z}$ between $x$ and $x+z$. Squaring both sides gives
$$
[\phi(x+z) - \phi(x)]^2 = \frac{1}{N^2} (\partial_{x_j} \phi(x))^2 + Error_1(N),
$$
where $Error_1(N) \leq C(\phi) N^{-4}$. Summing over $z = \pm \bar{e}_j$ with $p(z) = 1/(2d)$ yields
$$
\sum_{z=\pm \bar{e}_j} p(z) [\phi(x+z) - \phi(x)]^2 = \frac{1}{d N^2} (\partial_{x_j} \phi(x))^2 + \frac{1}{d} Error_1(N), \quad j=1,\dots,d.
$$

Summing over $j=1,\dots,d$, we obtain
$$
\sum_{j=1}^d \sum_{z = \pm \bar{e}_j} p(z) [\phi(x+z) - \phi(x)]^2 = \frac{1}{N^2} |\nabla \phi(x)|^2 + \sum_{j=1}^d\frac{1}{d}Error_1(N).
$$

Next, applying a first-order expansion to $\bar{\rho}(x+z)$ gives
$$
\bar{\rho}(x+z) = \bar{\rho}(x) + z \cdot \nabla \bar{\rho}(\xi_{x,z}),
$$
for some $\xi_{x,z}$ between $x$ and $x+z$, which implies
$$
\bar{\rho}(x) (1 - \bar{\rho}(x+z)) = \bar{\rho}(x) (1 - \bar{\rho}(x)) + Error_2(N),
$$
with $Error_2(N) \leq C(\rho_0) N^{-1}$. Combining the above estimates, we have
$$
\sum_{z \in \mathbb{T}^d_N} \bar{\rho}(x) (1 - \bar{\rho}(x+z)) p(z) [\phi(x+z) - \phi(x)]^2
= \frac{1}{N^2} \bar{\rho}(x) (1 - \bar{\rho}(x)) |\nabla \phi(x)|^2 + Error_3(N),
$$
where $Error_3(N) \leq C(\phi,\rho_0) N^{-1}$. Multiplying by $N^{2-d}$ and summing over $x \in \mathbb{T}^d_N$ gives a Riemann sum approximating $\int_{\mathbb{T}^d} \bar{\rho}(1-\bar{\rho}) |\nabla \phi|^2 dx$. Standard Riemann sum error estimates for smooth functions then yield
\begin{align*}
&\left| \langle \nabla\phi, \bar{\rho}(1 - \bar{\rho}) \nabla\phi \rangle - N^{2-d} \sum_{x \in \mathbb{T}^d_N} \sum_{z \in \mathbb{T}^d_N} \bar{\rho}(x)(1 - \bar{\rho}(x+z))p(z)[\phi(x+z) - \phi(x)]^2 \right|\\
&\lesssim \left| \langle \nabla\phi, \bar{\rho}(1 - \bar{\rho}) \nabla\phi \rangle - N^{-d} \sum_{x \in \mathbb{T}^d_N}  \bar{\rho}(x)(1 - \bar{\rho}(x))|\nabla\phi(x)|^2 \right| + \frac{C}{N} \lesssim \frac{C}{N},
\end{align*}
where $C$ depends on finitely many derivatives of $\bar{\rho}$ and $\phi$.
\end{proof}

In the following, we derive an error estimate quantifying the discrepancy between the quadratic variation of the process $\bar{X}^{1,N}_1$, as defined in \eqref{X11}, and the expected limiting expression involving the mobility matrix. 

\begin{proposition}\label{error-1}
Suppose the assumptions in Lemma \ref{error-rate} hold. For every $N \in \mathbb{N}_+$, let $\bar{X}^{1,N}_1$ be defined as in \eqref{X11}. Then, for all $t > h > 0$ and for every test function $\phi \in C^{\infty}(\mathbb{T}^d)$, the following estimate holds:
\begin{align}
\Bigg|\frac{1}{h} \mathbb{E}\big(\bar{X}^{1,N}_1(t+h,\phi)& - \bar{X}^{1,N}_1(t,\phi)\big)^2 - \langle \nabla\phi, \bar{\rho}(1 - \bar{\rho}) \nabla\phi \rangle\Bigg|\leq C(\rho_0,\phi)\left(h +N^{-1}\right). 
\end{align}
\end{proposition}
\begin{proof}
Applying Duhamel's formula to \eqref{X11}, we observe that, almost surely, for every $t > h > 0$, the difference $\bar{X}^{1,N}_1(t+h,\phi) - \bar{X}^{1,N}_1(t,\phi)$ can be decomposed as follows:
\begin{align*}
\bar{X}^{1,N}_1(t+h,\phi)-\bar{X}^{1,N}_1(t,\phi)
=&\ \langle N^{d/2}(\pi^N_0 - \mathbb{E}\pi^N_0), (S_N(t+h)-S_N(t))\phi \rangle \\
&+ N^{d/2} \int_0^{t+h} \langle S_N(t+h-s)\,\phi,  dM^N_s \rangle 
- N^{d/2} \int_0^t \langle S_N(t-s)\,\phi,  dM^N_s \rangle \\
=&\ \langle (\pi^N_0 - \mathbb{E}\pi^N_0), (S_N(t+h)-S_N(t))\phi \rangle \\
&+ N^{d/2} \int_0^t \langle (S_N(t+h-s) - S_N(t-s))\, \phi, dM^N_s \rangle \\
&+ N^{d/2} \int_t^{t+h} \langle S_N(t+h-s)\,\phi,  dM^N_s \rangle.
\end{align*}

Taking the second moment of this expression, expanding the square, and applying It\^o's product rule together with the properties of conditional expectations with respect to the initial sigma-algebra $\mathcal{F}_0$, we find that all covariance terms vanish. As a consequence, we obtain
\begin{align*}
\mathbb{E} \left( \bar{X}^{1,N}_1(t+h,\phi) - \bar{X}^{1,N}_1(t,\phi) \right)^2 
=&\ \mathbb{E} \left\langle N^{d/2}(\pi^N_0 - \mathbb{E}\pi^N_0), (S_N(t+h)-S_N(t))\phi \right\rangle^2 \\
&+ \mathbb{E} \left( N^{d/2} \int_0^t \langle (S_N(t+h-s) - S_N(t-s))\, \phi, dM^N_s \rangle \right)^2 \\
&+ \mathbb{E} \left( N^{d/2} \int_t^{t+h} \langle S_N(t+h-s)\,\phi,  dM^N_s \rangle \right)^2.
\end{align*}
Thanks to the characterization of the martingale term provided in Lemma \ref{lem-carre-du-champ}, the stochastic integral terms appearing in the second moment expansion can be expressed using the carr\'e du champ operator associated with the generator $\mathcal{L}_N$. More precisely, we have
\begin{align*}
&\mathbb{E}\left( N^{d/2} \int_0^t \langle (S_N(t+h-s) - S_N(t-s))\,\phi,  dM^N_s \rangle \right)^2\\
=& \mathbb{E} \Bigg( \int_0^t \mathcal{L}_N\left( \bar{X}^{1,N}_1(s, (S_N(t+h-s) - S_N(t-s))\phi )^2 \right) \\
&\quad - 2\, \bar{X}^{1,N}_1(s, (S_N(t+h-s) - S_N(t-s))\phi )\, \mathcal{L}_N \bar{X}^{1,N}_1(s, (S_N(t+h-s) - S_N(t-s))\phi ) \, ds \Bigg) \\
=:& I_1,
\end{align*}
and similarly,
\begin{align*}
\mathbb{E}\left( N^{d/2} \int_t^{t+h} \langle \phi, S_N(t+h-s)\, dM^N_s \rangle \right)^2
&= \mathbb{E} \Bigg( \int_t^{t+h} \mathcal{L}_N\left( \bar{X}^{1,N}_1(s, S_N(t+h-s)\phi )^2 \right) \\
&\quad - 2\, \bar{X}^{1,N}_1(s, S_N(t+h-s)\phi )\, \mathcal{L}_N \bar{X}^{1,N}_1(s, S_N(t+h-s)\phi ) \, ds \Bigg) \\
&=: I_2.
\end{align*}
We now estimate the two terms $I_1$ and $I_2$ separately. For the first term $I_1$, by applying the characterization of the carr\'e du champ operator given in \eqref{carre-property}, and using the definition \eqref{X11} of $\bar{X}^{1,N}_1$, we obtain
\begin{align*}
I_1 
=&\ \mathbb{E} \int_0^t N^2 \sum_{x \in \mathbb{T}^d_N} \sum_{z \in \mathbb{T}^d_N} \eta_s(x)(1 - \eta_s(x+z))\, p(z) \\
&\quad \cdot \left[ \bar{X}^{1,N}_1(s, (S_N(t+h-s) - S_N(t-s))\phi)(\eta^{x,x+z}) - \bar{X}^{1,N}_1(s, (S_N(t+h-s) - S_N(t-s))\phi) \right]^2 ds \\
=&\ \mathbb{E} \int_0^t N^{2-d} \sum_{x \in \mathbb{T}^d_N} \sum_{z \in \mathbb{T}^d_N} \eta_s(x)(1 - \eta_s(x+z))\, p(z) \\
&\quad \cdot \left[ \sum_{y \in \mathbb{T}^d_N} \left( \eta^{x,x+z}_s(y) - \eta_s(y)-\mathbb{E}(\eta^{x,x+z}_s(y) - \eta_s(y)) \right) \left( (S_N(t+h-s) - S_N(t-s))\phi \right)(y) \right]^2 ds.
\end{align*}
Similarly, for the second term $I_2$, we have
\begin{align*}
I_2 
=&\ \mathbb{E} \int_t^{t+h} N^{2-d} \sum_{x \in \mathbb{T}^d_N} \sum_{z \in \mathbb{T}^d_N} \eta_s(x)(1 - \eta_s(x+z))\, p(z) \\
&\quad \cdot \left[ \sum_{y \in \mathbb{T}^d_N} \left( \eta^{x,x+z}_s(y) - \eta_s(y)-\mathbb{E}(\eta^{x,x+z}_s(y) - \eta_s(y)) \right) (S_N(t+h-s)\phi)(y) \right]^2 ds.
\end{align*}
Concerning the term $I_1$, we analyze the structure of the configuration difference under the particle exchange. By definition, for any $x, z \in \mathbb{T}^d_N$ such that $\eta_s(x) = 1$ and $\eta_s(x+z) = 0$, it holds that 
\begin{align}\label{configuration}
\eta^{x,x+z}_s(y) - \eta_s(y) = 
\begin{cases}
-1 & \text{if } y = x, \\
1  & \text{if } y = x+z, \\
0  & \text{otherwise}.
\end{cases}
\end{align}
Similarly, it follows that 
\begin{align}\label{configuration-2}
\mathbb{E}(\eta^{x,x+z}_s(y) - \eta_s(y)) = 
\begin{cases}
-\mathbb{P}(\eta_s(x) = 1,\ \eta_s(x+z) = 0) & \text{if } y = x, \\
\mathbb{P}(\eta_s(x) = 1,\ \eta_s(x+z) = 0)  & \text{if } y = x+z, \\
0  & \text{otherwise}.
\end{cases}
\end{align}
This observation implies that the difference $\eta^{x,x+z}_s(y) - \eta_s(y)$ is supported only on the sites $x$ and $x+z$, and hence the inner sum in the expression for $I_1$ simplifies. Therefore, we obtain
\begin{align*}
I_1 
=&\ \mathbb{E} \int_0^t N^{2-d} \sum_{x \in \mathbb{T}^d_N} \sum_{z \in \mathbb{T}^d_N} \eta_s(x)(1 - \eta_s(x+z))\, p(z) \Big[(1-\mathbb{P}(\eta_s(x) = 1,\ \eta_s(x+z) = 0))\\
&\quad\cdot ((S_N(t+h-s) - S_N(t-s))\phi)\left(x+z\right) - ((S_N(t+h-s) - S_N(t-s))\phi)\left(x\right) \Big]^2 ds.
\end{align*}

Applying Lagrange's intermediate value theorem and using the smoothness of $\phi$, together with the definition \eqref{semi-group-SN} of the discrete semigroup $S_N(t)$ and Young's inequality for convolutions, we obtain
\begin{align*}
I_1 \leq& \mathbb{E} \int_0^t N^{2-d} \sum_{x \in \mathbb{T}^d_N} \sum_{z \in \mathbb{T}^d_N} \eta_s(x)\left(1 - \eta_s(x+z)\right) p(z) \left\|(S_N(t+h-s) - S_N(t-s)) \nabla \phi\right\|_{L^\infty(\mathbb{T}^d)}^2|z|^2 \, ds. 
\end{align*}
With the help of the definition of $S_N(t)$, for every $\xi\in\mathbb{T}^d$, 
\begin{align*}
\left| (S_N(t+h-s)\nabla\phi)(\xi) - (S_N(t-s)\nabla \phi)(\xi) \right|=&\left|\int^{t+h-s}_{t-s}S_N(r)\frac{1}{2}\Delta_N\nabla\phi(\xi)dr\right|\\ 
\lesssim& \int^{t+h-s}_{t-s}\| S_N(r) \|_{\mathcal{L}(L^{\infty}(\mathbb{T}^d),L^{\infty}(\mathbb{T}^d))} \cdot \| \Delta_N\nabla \phi \|_{L^\infty(\mathbb{T}^d)}dr,
\end{align*}
where $\|\cdot\|_{\mathcal{L}(L^{\infty}(\mathbb{T}^d),L^{\infty}(\mathbb{T}^d))}$ denotes the operator norm from $L^{\infty}(\mathbb{T}^d)$ to $L^{\infty}(\mathbb{T}^d)$. Noting that by the definition of $S_N(t)$ and the discrete convolutional Young's inequality, we deduce that $$\| S_N(r) \|_{\mathcal{L}(L^{\infty}(\mathbb{T}^d),L^{\infty}(\mathbb{T}^d))}\lesssim1,\ \text{for every }r\in[0,T].$$ Combining this with the fact that $|z|=\frac{1}{N}$, we have  
\begin{align*}
I_1\leq& \mathbb{E} \int_0^t N^{2-d} \sum_{x \in \mathbb{T}^d_N} \sum_{z \in \mathbb{T}^d_N} \eta_s(x)\left(1 - \eta_s(x+z)\right) p(z)\cdot \frac{1}{N^2} h^2\|\Delta_N\nabla \phi\|_{L^\infty(\mathbb{T}^d)}^2 \, ds.
\end{align*}
Using the properties of the function $p$, the boundedness of $\eta$, the finiteness of the discrete torus, and continuity properties of the kernel $G_N$, we deduce that
\begin{align*}
I_1 \leq \left( \int_0^t N^{-d} \sum_{x \in \mathbb{T}^d_N} 1 \, ds \right) C(\phi) h^2 \leq C(\phi, T) h^2.
\end{align*}

Regarding the term $I_2$, we apply the properties \eqref{configuration} and \eqref{configuration-2} of the configuration to obtain
\begin{align*}
I_2 \lesssim& \, \mathbb{E}\int_t^{t+h} N^{2-d} \sum_{x\in\mathbb{T}^d_N} \sum_{z\in\mathbb{T}^d_N} \eta_s(x)(1-\eta_s(x+z))p(z) \\
& \quad \cdot \Big[(S_N(t+h-s)\phi)(x+z) - (S_N(t+h-s)\phi)(x)\Big]^2 \, ds.
\end{align*}
To facilitate the estimation of $I_2$, we decompose it into four terms and analyze each individually. By expanding the integrand, we obtain:
\begin{align*}
I_2 \lesssim& \int_t^{t+h} N^{2-d} \sum_{x\in\mathbb{T}^d_N} \sum_{z\in\mathbb{T}^d_N} \mathbb{E}\big(\eta_s(x)(1-\eta_s(x+z))\big)p(z) \\
& \quad \cdot \Big[(S_N(t+h-s)\phi)(x+z) - \phi(x+z)  - (S_N(t+h-s)\phi)(x) + \phi(x) + \phi(x+z) - \phi(x)\Big]^2 \, ds \\
=:& \, N^{2-d} \sum_{x\in\mathbb{T}^d_N} \sum_{z\in\mathbb{T}^d_N} \bar{\rho}(t,x)(1 - \bar{\rho}(t,x+z))p(z) [\phi(x+z) - \phi(x)]^2 h + I_{2,1} + I_{2,2} + I_{2,3},
\end{align*}
where $I_{2,1}$, $I_{2,2}$ and $I_{2,3}$ are defined by 
\begin{align*}
I_{2,1}=&N^{2-d} \sum_{x\in\mathbb{T}^d_N} \sum_{z\in\mathbb{T}^d_N} \bar{\rho}(t,x)(1 - \bar{\rho}(t,x+z))p(z) \\
& \cdot \int_t^{t+h} \Big[(S_N(t+h-s)\phi)(x+z) - (S_N(t+h-s)\phi)(x)- (\phi(x+z) - \phi(x))\Big]^2 ds,\\
I_{2,2}=&2N^{2-d} \sum_{x\in\mathbb{T}^d_N} \sum_{z\in\mathbb{T}^d_N} \bar{\rho}(t,x)(1 - \bar{\rho}(t,x+z))p(z)\int_t^{t+h} [\phi(x+z) - \phi(x)] \\
& \qquad \cdot \Big[(S_N(t+h-s)\phi)(x+z) - (S_N(t+h-s)\phi)(x)- (\phi(x+z) - \phi(x))\Big] ds,\\
I_{2,3}=&\int_t^{t+h} N^{2-d} \sum_{x\in\mathbb{T}^d_N} \sum_{z\in\mathbb{T}^d_N} \Big[\mathbb{E}(\eta_s(x)(1-\eta_s(x+z))) - \bar{\rho}(t,x)(1 - \bar{\rho}(t,x+z))\Big]p(z) \\
& \quad \cdot \Big[(S_N(t+h-s)\phi)(x+z) - \phi(x+z) - (S_N(t+h-s)\phi)(x) + \phi(x) + \phi(x+z) - \phi(x)\Big]^2 ds.	
\end{align*}

We now estimate $I_{2,1}$. By the Lagrange's intermediate value theorem, for each $N \in \mathbb{N}$ and $x, z \in \mathbb{T}^d_N$, there exist points $\xi_{x,z,N}, \xi_{x,z,N}' \in \mathbb{T}^d$ such that:
\begin{align*}
&I_{2,1} = N^{2-d} \sum_{x\in\mathbb{T}^d_N} \sum_{z\in\mathbb{T}^d_N} \bar{\rho}(t,x)(1 - \bar{\rho}(t,x+z))p(z) \\
&\quad \cdot \int_t^{t+h} \left[ (\nabla S_N(t+h-s)\phi)(\xi_{x,z,N}) \cdot z - (\nabla \phi)(\xi_{x,z,N}') \cdot z \right]^2 ds \\
&\leq N^{2-d} \sum_{x\in\mathbb{T}^d_N} \sum_{z\in\mathbb{T}^d_N} \bar{\rho}(t,x)(1 - \bar{\rho}(t,x+z))p(z) \int_t^{t+h} \left| (\nabla S_N(t+h-s)\phi)(\xi_{x,z,N}) - (\nabla \phi)(\xi_{x,z,N}') \right|^2 |z|^2 ds \\
&\leq N^{-d} \sum_{x\in\mathbb{T}^d_N} \sum_{z\in\mathbb{T}^d_N} \bar{\rho}(t,x)(1 - \bar{\rho}(t,x+z))p(z) \int_t^{t+h} \left| (\nabla S_N(t+h-s)\phi)(\xi_{x,z,N}) - (\nabla \phi)(\xi_{x,z,N}') \right|^2 ds.
\end{align*}

We now split the difference into two terms and estimate separately:
\begin{align*}
I_{2,1} 
&\lesssim N^{-d} \sum_{x\in\mathbb{T}^d_N} \sum_{z\in\mathbb{T}^d_N} \bar{\rho}(t,x)(1 - \bar{\rho}(t,x+z))p(z) \int_t^{t+h} \left| (\nabla S_N(t+h-s)\phi)(\xi_{x,z,N}) - (\nabla \phi)(\xi_{x,z,N}) \right|^2 ds \\
&\quad + N^{-d} \sum_{x\in\mathbb{T}^d_N} \sum_{z\in\mathbb{T}^d_N} \bar{\rho}(t,x)(1 - \bar{\rho}(t,x+z))p(z) \int_t^{t+h} \left| (\nabla \phi)(\xi_{x,z,N}) - (\nabla \phi)(\xi_{x,z,N}') \right|^2 ds.
\end{align*}

For the first term, with the help of the definition of $S_N(t)$, for every $\xi\in\mathbb{T}^d$, 
\begin{align*}
\left| (\nabla S_N(t+h-s)\phi)(\xi) - (\nabla \phi)(\xi) \right|=&\left|\int^{t+h-s}_0S_N(r)\frac{1}{2}\Delta_N\nabla\phi(\xi)dr\right|\\ 
\lesssim& \int^{t+h-s}_0\| S_N(r) \|_{\mathcal{L}(L^{\infty}(\mathbb{T}^d),L^{\infty}(\mathbb{T}^d))} \cdot \| \Delta_N\nabla \phi \|_{L^\infty(\mathbb{T}^d)}dr. 
\end{align*}
This implies that $\left| (\nabla S_N(t+h-s)\phi)(\xi) - (\nabla \phi)(\xi) \right|\lesssim (t+h-s)C(\phi)$, for every $\xi\in\mathbb{T}^d$.

For the second term, we apply the regularity of $\nabla \phi$ and the estimate:
$$
\left| (\nabla \phi)(\xi) - (\nabla \phi)(\xi') \right| \lesssim \| \phi \|_{C^2(\mathbb{T}^d)} |\xi - \xi'|.
$$

Combining all estimates, we obtain:
\begin{align*}
I_{2,1} 
&\lesssim N^{-d} \sum_{x\in\mathbb{T}^d_N} \sum_{z\in\mathbb{T}^d_N} \bar{\rho}(t,x)(1 - \bar{\rho}(t,x+z))p(z) \int_t^{t+h} (t+h-s)^2C(\phi) ds \\
&\quad + N^{-d} \sum_{x\in\mathbb{T}^d_N} \sum_{z\in\mathbb{T}^d_N} \bar{\rho}(t,x)(1 - \bar{\rho}(t,x+z))p(z) \int_t^{t+h} \|\phi \|_{C^2(\mathbb{T}^d)}^2 |\xi_{x,z,N} - \xi_{x,z,N}'|^2 ds \\
&\leq C(\bar{\rho}, \phi) h^3 + C(\bar{\rho}, \phi) h N^{-2}.
\end{align*}

Applying a similar argument to $I_{2,2}$, we deduce the estimate:
\begin{align*}
I_{2,2} \leq C(\bar{\rho}, \phi) h^2N^{-1} + C(\bar{\rho}, \phi) h N^{-2}.
\end{align*}
To estimate the term $I_{2,3}$, we begin by decomposing it into two parts based on the difference between the microscopic correlation and its hydrodynamic approximation, and the temporal fluctuation of the latter. Specifically, we write:
\begin{align*}
I_{2,3} =& \int_t^{t+h} N^{2-d} \sum_{x \in \mathbb{T}^d_N} \sum_{z \in \mathbb{T}^d_N} \big[\mathbb{E}(\eta_s(x)(1 - \eta_s(x+z))) - \bar{\rho}(s,x)(1 - \bar{\rho}(s,x+z))\big] p(z) \\
& \cdot \Big[(S_N(t+h-s)\phi)(x+z) - \phi(x+z) - (S_N(t+h-s)\phi)(x)  + \phi(x) + \phi(x+z) - \phi(x) \Big]^2 \, ds \\
& + \int_t^{t+h} N^{2-d} \sum_{x \in \mathbb{T}^d_N} \sum_{z \in \mathbb{T}^d_N} \big[\bar{\rho}(s,x)(1 - \bar{\rho}(s,x+z)) - \bar{\rho}(t,x)(1 - \bar{\rho}(t,x+z))\big] p(z) \\
& \cdot \Big[(S_N(t+h-s)\phi)(x+z) - \phi(x+z) - (S_N(t+h-s)\phi)(x)  + \phi(x) + \phi(x+z) - \phi(x) \Big]^2 \, ds.
\end{align*}

To control the first term, we use the replacement lemma \ref{error-rate}, which allows us to approximate the local product $\mathbb{E}(\eta_s(x)(1 - \eta_s(x+z)))$ by $\bar{\rho}(s,x)(1 - \bar{\rho}(s,x+z))$ up to an error bounded by $$N^{-2}+N^{-2}I_{\{d=1\}}+\frac{\log(1+N^2)}{1+N^2}I_{\{d=2\}}+\frac{1}{1+N^2}I_{\{d\ge3\}}\lesssim N^{-1}.$$ The additional factor of $h$ comes from the time integration. For the second term, we apply the regularity of the hydrodynamic profile $\bar{\rho}(t,x)$ in time, which implies that the difference $$\bar{\rho}(s,x)(1 - \bar{\rho}(s,x+z)) - \bar{\rho}(t,x)(1 - \bar{\rho}(t,x+z))$$ is of order $h$ uniformly in $x$ and $z$. Thus, we conclude that 
\begin{align*}
I_{2,3}\leq & \int_t^{t+h} N^{2-d} \sum_{x \in \mathbb{T}^d_N} \sum_{z \in \mathbb{T}^d_N}  N^{-1}  \, p(z) \\
& \cdot \Big[(S_N(t+h-s)\phi)(x+z) - \phi(x+z) - (S_N(t+h-s)\phi)(x)  + \phi(x) + \phi(x+z) - \phi(x) \Big]^2 \, ds \\
& + \int_t^{t+h} N^{2-d} \sum_{x \in \mathbb{T}^d_N} \sum_{z \in \mathbb{T}^d_N} C(\rho_0) h \, p(z) \\
& \cdot \Big[(S_N(t+h-s)\phi)(x+z) - \phi(x+z) - (S_N(t+h-s)\phi)(x) + \phi(x) + \phi(x+z) - \phi(x) \Big]^2 \, ds.
\end{align*}
Bounding the square term involving differences of $\phi$ and $S_N$ by Taylor expansions and smoothing estimates, we obtain
\begin{align*}
I_{2,3} \lesssim C(\rho_0, \phi) \Big(  N^{-1} + h \Big)h.
\end{align*}

To control the contribution from the initial data, we estimate
\begin{align*}
&\mathbb{E} \left\langle N^{d/2}(S_N(t+h) - S_N(t)) (\pi^N_0 - \mathbb{E}\pi^N_0), \phi \right\rangle^2\\
\leq&\mathbb{E}\Big(N^{-d/2}\sum_{x\in\mathbb{T}^d_N}(\eta_0(x)-\rho_0(x))[(S_N(t+h) - S_N(t))\phi](x)\Big)^2\\
\leq&N^{-d}\sum_{x\in\mathbb{T}^d_N}\mathbb{E}(\eta_0(x)-\rho_0(x))^2\|(S_N(t+h) - S_N(t))\phi\|_{L^{\infty}(\mathbb{T}^d)}^2\\
\leq& C(\phi,\rho_0) h^2,
\end{align*}
which follows from the contractivity of the semigroup and the assumed bounded variance of the initial empirical measure.

Finally, combining all the above estimates, we arrive at
\begin{align}
\Big| \frac{1}{h} \mathbb{E} \big(\bar{X}^{1,N}_1(t+h, \phi) - &\bar{X}^{1,N}_1(t, \phi)\big)^2 - \langle \nabla \phi, \bar{\rho}(1 - \bar{\rho}) \nabla \phi \rangle \Big|\notag\\
 \leq& C(\rho_0, \phi)(h^2 + h + N^{-2}+N^{-1})     + C(\rho_0, \phi) \Big(  N^{-1} + h \Big)\notag\\
\leq& C(\rho_0,\phi)\left(h +N^{-1}\right). 
\end{align}

\end{proof}

\begin{theorem}
	Suppose the assumptions in Lemma \ref{error-rate} hold. For every $0 < h < t < T$, we have
	\begin{align}
	\Big|\frac{1}{h}\mathbb{E}(\bar{X}^{1,N}(t+h,\phi)-\bar{X}^{1,N}(t,\phi))^2 - \langle \nabla\phi, \bar{\rho}(1-\bar{\rho}) \nabla\phi \rangle \Big|\leq  C(\rho_0,\phi)\left(h + hN^{d-4}+N^{-1}\right). \notag
	\end{align}
\end{theorem}

\begin{proof}
By the definition of $\bar{X}^{1,N}$, we decompose its increment over the time interval $[t, t+h]$ as
\begin{align*}
	&\frac{1}{h} \mathbb{E}(\bar{X}^{1,N}(t+h,\phi) - \bar{X}^{1,N}(t,\phi))^2 \\
	=&\; \frac{1}{h} \mathbb{E} \left( \bar{X}^{1,N}_1(t+h,\phi) - \bar{X}^{1,N}_1(t,\phi) + N^{-d/2} \sum_{x \in \mathbb{T}^d_N} \left[ (\rho_N - \bar{\rho})(x,t+h) - (\rho_N - \bar{\rho})(x,t) \right] \phi(x) \right)^2 \\
	=&\; \frac{1}{h} \mathbb{E}\left( \bar{X}^{1,N}_1(t+h,\phi) - \bar{X}^{1,N}_1(t,\phi) \right)^2 + N^{-d} \frac{1}{h} \left( \sum_{x \in \mathbb{T}^d_N} \left[ (\rho_N - \bar{\rho})(x,t+h) - (\rho_N - \bar{\rho})(x,t) \right] \phi(x) \right)^2,
\end{align*}
where we used the independence of the martingale part $\bar{X}^{1,N}_1$ and the deterministic fluctuation to eliminate cross-terms in the expectation.

We now estimate the second term. By the mean value theorem in time (Lagrange's form), there exists $\theta \in (0,1)$ such that
\begin{align*}
	&(\rho_N - \bar{\rho})(x,t+h) - (\rho_N - \bar{\rho})(x,t) = h\,\partial_t(\rho_N - \bar{\rho})(x,t + \theta h).
\end{align*}
Since both $\rho_N$ and $\bar{\rho}$ satisfy heat equations with respective discrete and continuous Laplacians, we have
$$
\partial_t (\rho_N - \bar{\rho}) = \mathcal{L}_N \rho_N - \frac{1}{2}\Delta \bar{\rho},
$$
and hence
\begin{align*}
	&N^{-d} \frac{1}{h} \left( \sum_{x \in \mathbb{T}^d_N} \left[ (\rho_N - \bar{\rho})(x,t+h) - (\rho_N - \bar{\rho})(x,t) \right] \phi(x) \right)^2 \\
	\leq\;& N^{-d} \frac{1}{h} \left( \sum_{x \in \mathbb{T}^d_N} h\,\|\mathcal{L}_N \rho_N - \frac{1}{2}\Delta \bar{\rho}\|_{L^\infty([0,T] \times \mathbb{T}^d_N)} \phi(x) \right)^2 \\
	\leq\;& h\, \|\mathcal{L}_N \rho_N - \frac{1}{2}\Delta \bar{\rho} \|^2_{L^\infty([0,T] \times \mathbb{T}^d_N)} \cdot N^{-d} \left( \sum_{x \in \mathbb{T}^d_N} \phi(x) \right)^2 \\
	\leq\;& C(\rho_0,\phi)\, h N^d \|\mathcal{L}_N \rho_N -\frac{1}{2}\Delta \bar{\rho} \|^2_{L^\infty([0,T] \times \mathbb{T}^d_N)} \leq C(\rho_0)hN^{d-4}.
\end{align*}

Now, applying Proposition~\ref{error-1}, together with the triangle inequality, \eqref{disc-lap-error-2} and \eqref{riemann-error}, we conclude 
\begin{align*}
	&\left| \frac{1}{h} \mathbb{E}(\bar{X}^{1,N}(t+h,\phi) - \bar{X}^{1,N}(t,\phi))^2 - \langle \nabla\phi, \bar{\rho}(1-\bar{\rho}) \nabla\phi \rangle \right| \\
	\leq\;& \left| \frac{1}{h} \mathbb{E}(\bar{X}^{1,N}_1(t+h,\phi) - \bar{X}^{1,N}_1(t,\phi))^2 - \langle \nabla\phi, \bar{\rho}(1-\bar{\rho}) \nabla\phi \rangle \right| \\
	&+ N^{-d} \frac{1}{h} \left( \sum_{x \in \mathbb{T}^d_N} \left[ (\rho_N - \bar{\rho})(x,t+h) - (\rho_N - \bar{\rho})(x,t) \right] \phi(x) \right)^2 \\
	\leq\;&  C(\rho_0,\phi)\left(h + hN^{d-4}+N^{-1}\right), 
\end{align*}
which completes the proof.
\end{proof}

\section{Fluctuating hydrodynamics with regular coefficients: error estimates}\label{sec-5}

In this section, our primary objective is to derive quantitative estimates comparing the quadratic variation of the fluctuation field to the corresponding mobility matrix. We first consider the case $\sigma(\zeta)=\sqrt{\zeta(1-\zeta)}$, for every $\zeta\in[0,1]$. Let $(\sigma_n(\cdot))_{n\geq1}$ be a sequence of approximation of $\sigma(\cdot)$, satisfying Assumption \ref{Assump-sigman}, recall that $\rho_0\in C^{\infty}(\mathbb{T}^d;[0,1])$, we investigate the equation \eqref{SPDE-1}, and recall that \eqref{SPDE-1} can be formulated by the following It\^o form: 
\begin{align}\label{SPDE-1-ito-sec-5}
d\rho_{\varepsilon} = \frac{1}{2}\Delta \rho_{\varepsilon}\,dt - \varepsilon^{1/2} \nabla\cdot\left(\sigma_{n(\varepsilon)}(\rho_{\varepsilon})\, dW_{\delta(\varepsilon)}\right) + \frac{\varepsilon}{2} \nabla\cdot\left(\sigma_{n(\varepsilon)}'(\rho_{\varepsilon})^2\, \nabla \rho_{\varepsilon} \right) F_{1,\delta(\varepsilon)},\quad \rho_{\varepsilon}(0)=\rho_0. 
\end{align}

To facilitate comparison, we also consider the deterministic hydrodynamic limit given by the heat equation with the same initial profile:
\begin{equation}\label{PDE}
d\bar{\rho} = \frac{1}{2}\Delta\bar{\rho}\, dt, \quad \bar{\rho}(0) = \rho_0.
\end{equation}

Define the fluctuation field by
$$
\bar{\rho}^1_{\varepsilon} := \varepsilon^{-1/2}(\rho_{\varepsilon} - \bar{\rho}),
$$
which captures the deviation of the stochastic density profile $\rho_{\varepsilon}$ from its deterministic limit $\bar{\rho}$. Applying Duhamel's formula and performing a straightforward computation, one obtains that $\bar{\rho}^1_{\varepsilon}$ satisfies the following mild representation:
\begin{align*}
\bar{\rho}^1_{\varepsilon}(t) 
= - \int_0^t S(t-s) \nabla \cdot \left( \sigma_{n(\varepsilon)}(\rho_{\varepsilon}) \, dW_{\delta(\varepsilon)}(s) \right) 
+ \frac{\varepsilon^{1/2}}{2} \int_0^t S(t-s) \nabla \cdot \left( \sigma_{n(\varepsilon)}'(\rho_{\varepsilon})^2 \nabla \rho_{\varepsilon} \right) F_{1,\delta(\varepsilon)} \, ds,
\end{align*}
where $S(t)$ denotes the heat semigroup generated by $\frac{1}{2}\Delta$ on $\mathbb{T}^d$.

Let $\bar{\rho}^1$ denote the solution to the limiting Gaussian stochastic PDE. It also admits a mild formulation given by
\begin{align*}
\bar{\rho}^1(t) = -\int_0^t S(t-s)\nabla\cdot\left(\sqrt{\bar{\rho}(1 - \bar{\rho})} \, dW(s)\right).
\end{align*}

We begin our analysis by providing a quantitative estimate corresponding to a law of large numbers result.

\begin{lemma}\label{lem-LLN-regular}
Assume that the initial condition $\rho_0$ and the regularization sequence $(\sigma_n(\cdot))_{n \geq 1}$ satisfy Assumptions~\ref{Assump-initial-1} and~\ref{Assump-sigman}, respectively. For each $\varepsilon > 0$, let $\rho_{\varepsilon}$ denote the mild solution to \eqref{SPDE-1-ito-sec-5} with initial condition $\rho_0$, and let $\bar{\rho}$ denote the mild solution to \eqref{PDE} with the same initial data. Then, under the scaling condition
\begin{align*}
\lim_{\varepsilon \to 0} \varepsilon \delta(\varepsilon)^{-d-2} = 0,
\end{align*}
it follows that
\begin{align*}
\sup_{t\in[0,T]}\mathbb{E} \left\| \rho_{\varepsilon}(t) - \bar{\rho}(t) \right\|_{L^2(\mathbb{T}^d)}^2 
\leq \varepsilon \delta(\varepsilon)^{-d-2} 
\longrightarrow 0 \quad \text{as } \varepsilon \to 0.
\end{align*}
\end{lemma}
\begin{proof}
Applying It\^o's formula, we observe that 
\begin{align*}
\sup_{t\in[0,T]}\mathbb{E}\|\rho_{\varepsilon}(t)\|_{L^2(\mathbb{T}^d)}^2+\mathbb{E}\int^T_0\|\nabla\rho_{\varepsilon}\|_{L^2(\mathbb{T}^d)}^2ds\leq& \|\rho_0\|_{L^2(\mathbb{T}^d)}^2+C(\rho_0,T)\varepsilon F_{3,\delta(\varepsilon)}\\
\leq&\|\rho_0\|_{L^2(\mathbb{T}^d)}^2+C(\rho_0,T)\varepsilon \delta(\varepsilon)^{-d-2}. 	
\end{align*}
Similarly, applying It\^o's formula to $\rho_{\varepsilon}-\bar{\rho}$, we conclude that 
\begin{align*}
	\sup_{t\in[0,T]}\mathbb{E}\|\rho_{\varepsilon}(t)-\bar{\rho}(t)\|_{L^2(\mathbb{T}^d)}^2+\mathbb{E}\int^T_0\|\nabla(\rho_{\varepsilon}-\bar{\rho})\|_{L^2(\mathbb{T}^d)}^2ds\leq&C(\rho_0,T)\varepsilon \delta(\varepsilon)^{-d-2}.
\end{align*}

\end{proof}

\begin{theorem}
Let $\sigma(\zeta)=\sqrt{\zeta(1-\zeta)}$. Assume that the initial condition $\rho_0$ satisfies Assumption \ref{Assump-initial-1} and the regularization sequence $(\sigma_n(\cdot))_{n \geq 1}$ satisfies Assumption \ref{Assump-sigman} with $\sigma_n\rightarrow\sigma$ in $C^1_{loc}((0,1))$. Let $c$ denote the lower bound of the initial data, as specified in Assumption \ref{Assump-initial-1}. Let $\phi \in C^{\infty}(\mathbb{T}^d)$, and define $\rho_{\varepsilon}^{\phi} := \langle \rho_{\varepsilon}, \phi \rangle$. The corresponding fluctuation field is given by
\begin{align*}
\bar{\rho}_{\varepsilon}^{1,\phi} := \varepsilon^{-1/2}(\rho_{\varepsilon}^{\phi} - \bar{\rho}^{\phi}).
\end{align*}
Then, for every $t > h > 0$, the following estimate holds:
\begin{align*}
\left| \frac{1}{h} \mathbb{E} \left( \bar{\rho}_{\varepsilon}^{1,\phi}(t+h) - \bar{\rho}_{\varepsilon}^{1,\phi}(t) \right)^2 
- \left\langle \nabla \phi, \bar{\rho}(t)(1 - \bar{\rho}(t)) \nabla \phi \right\rangle \right|
\leq C(\phi, \rho_0, t) \cdot \mathrm{Error}(h, \varepsilon, \delta(\varepsilon), n(\varepsilon)),
\end{align*}
where the error term is given by
\begin{align*}
\mathrm{Error}(h, \varepsilon, \delta(\varepsilon), n(\varepsilon)) =\ & h 
+ \varepsilon \delta(\varepsilon)^{-2d} \|\sigma_{n(\varepsilon)}'\|_{L^{\infty}([0,1])}^4 h + \varepsilon \delta(\varepsilon)^{-d-2}. 
\end{align*}
\end{theorem}
\begin{proof}
We expand the difference $\bar{\rho}^{1,\phi}_{\varepsilon}(t+h) - \bar{\rho}^{1,\phi}_{\varepsilon}(t)$ using Duhamel's formula. The integral is naturally split into two parts: one over the interval $[0, t]$ and the other over $[t, t+h]$. A straightforward computation yields the following expression:
\begin{align*}
&\bar{\rho}^{1,\phi}_{\varepsilon}(t+h) - \bar{\rho}^{1,\phi}_{\varepsilon}(t)\\ 
=& - \int_0^{t+h} \Big\langle \nabla\phi, S(t+h - s) \sigma_{n(\varepsilon)}(\rho_{\varepsilon}) \, dW_{\delta(\varepsilon)}(s) \Big\rangle 
+ \int_0^t \Big\langle \nabla\phi, S(t - s) \sigma_{n(\varepsilon)}(\rho_{\varepsilon}) \, dW_{\delta(\varepsilon)}(s) \Big\rangle \\
& + \frac{\varepsilon^{1/2}}{2} F_{1,\delta(\varepsilon)} \int_0^{t+h} \left\langle \nabla\phi, S(t+h - s) \left( \sigma_{n(\varepsilon)}'(\rho_{\varepsilon})^2 \nabla \rho_{\varepsilon} \right) \right\rangle ds \\
& - \frac{\varepsilon^{1/2}}{2} F_{1,\delta(\varepsilon)} \int_0^t \left\langle \nabla\phi, S(t - s) \left( \sigma_{n(\varepsilon)}'(\rho_{\varepsilon})^2 \nabla \rho_{\varepsilon} \right) \right\rangle ds \\
=& - \int_0^t \Big\langle S(h)\nabla\phi - \nabla\phi, S(t - s) \sigma_{n(\varepsilon)}(\rho_{\varepsilon}) \, dW_{\delta(\varepsilon)}(s) \Big\rangle 
- \int_t^{t+h} \Big\langle \nabla\phi, S(t+h - s) \sigma_{n(\varepsilon)}(\rho_{\varepsilon}) \, dW_{\delta(\varepsilon)}(s) \Big\rangle \\
& + \frac{\varepsilon^{1/2}}{2} F_{1,\delta(\varepsilon)} \int_0^{t+h} \left\langle \nabla\phi, S(t+h - s) \left( \sigma_{n(\varepsilon)}'(\rho_{\varepsilon})^2 \nabla \rho_{\varepsilon} \right) \right\rangle ds \\
& - \frac{\varepsilon^{1/2}}{2} F_{1,\delta(\varepsilon)} \int_0^t \left\langle \nabla\phi, S(t - s) \left( \sigma_{n(\varepsilon)}'(\rho_{\varepsilon})^2 \nabla \rho_{\varepsilon} \right) \right\rangle ds.
\end{align*}
Taking the second moment, expanding the square, and applying It\^o's product rule, we observe that all cross-covariance terms vanish. Using It\^o's isometry along with the covariance structure of the Brownian motions, we obtain
\begin{align*}
\mathbb{E} \left( \bar{\rho}^{1,\phi}_{\varepsilon}(t+h) - \bar{\rho}^{1,\phi}_{\varepsilon}(t) \right)^2 
=&\ \mathbb{E} \int_0^t \| (S(h)\nabla\phi - \nabla\phi)\, S(t - s) \sigma_{n(\varepsilon)}(\rho_{\varepsilon}) \|_{L^2(\mathbb{T}^d)}^2\, ds \\
&+ \mathbb{E} \int_t^{t+h} \| \nabla\phi\, S(t+h - s)\, \sigma_{n(\varepsilon)}(\rho_{\varepsilon}) \|_{L^2(\mathbb{T}^d)}^2\, ds \\
&+ \frac{\varepsilon}{4} F_{1,\delta(\varepsilon)}^2\, \mathbb{E} \left( \int_0^{t+h} \left\langle \nabla\phi, S(t+h - s) \left( \sigma_{n(\varepsilon)}'(\rho_{\varepsilon})^2 \nabla \rho_{\varepsilon} \right) \right\rangle ds \right. \\
&\qquad \left. - \int_0^t \left\langle \nabla\phi, S(t - s) \left( \sigma_{n(\varepsilon)}'(\rho_{\varepsilon})^2 \nabla \rho_{\varepsilon} \right) \right\rangle ds \right)^2 \\
=:&\ I_1(h,\varepsilon) + I_2(h,\varepsilon) + I_3(h,\varepsilon).
\end{align*}

We now estimate each of the terms $I_1(h,\varepsilon)$, $I_2(h,\varepsilon)$, and $I_3(h,\varepsilon)$ separately. For the first term $I_1(h,\varepsilon)$, applying the convolutional Young's inequality yields
\begin{align*}
I_1(h,\varepsilon) 
&\leq \mathbb{E} \int_0^t \| S(h)\nabla\phi - \nabla\phi \|_{L^\infty(\mathbb{T}^d)}^2 
\cdot \| S(t - s)\sigma_{n(\varepsilon)}(\rho_{\varepsilon}) \|_{L^{2}(\mathbb{T}^d)}^2\, ds \\
&\leq C(\rho_0)\int_0^t \Big(\int^h_0\| S(r)\Delta\nabla\phi \|_{L^\infty(\mathbb{T}^d)}dr\Big)^2 
\cdot \| p(t - s) \|_{L^1(\mathbb{T}^d)}^2\, ds \\
&\leq C(\phi,\rho_0)\, h^2 t,
\end{align*}
where $p(t - s)$ denotes the heat kernel on $\mathbb{T}^d$ and the final inequality follows from the smoothing property of the heat semigroup.

To estimate $I_2(h,\varepsilon)$, we decompose the integrand by introducing an intermediate term involving the pointwise evaluation at time $t$. Specifically, we write:
\begin{align*}
I_2(h,\varepsilon) 
=&\ \mathbb{E} \int_t^{t+h} \left( \| \nabla\phi\, S(t+h-s)\, \sigma_{n(\varepsilon)}(\rho_{\varepsilon}(s)) \|_{L^2(\mathbb{T}^d)}^2 
- \| \nabla\phi\, \sigma_{n(\varepsilon)}(\rho_{\varepsilon}(t)) \|_{L^2(\mathbb{T}^d)}^2 \right) ds \\
&+ \mathbb{E} \int_t^{t+h} \left( \| \nabla\phi\, \sigma_{n(\varepsilon)}(\rho_{\varepsilon}(t)) \|_{L^2(\mathbb{T}^d)}^2 
- \| \nabla\phi\, \sqrt{\bar{\rho}(t)(1 - \bar{\rho}(t))} \|_{L^2(\mathbb{T}^d)}^2 \right) ds \\
&+ h \| \nabla\phi\, \sqrt{\bar{\rho}(t)(1 - \bar{\rho}(t))} \|_{L^2(\mathbb{T}^d)}^2 \\
=:&\ I_{2,1}(h,\varepsilon) + I_{2,2}(h,\varepsilon) + h \| \nabla\phi\, \sqrt{\bar{\rho}(t)(1 - \bar{\rho}(t))} \|_{L^2(\mathbb{T}^d)}^2.
\end{align*}

We first consider $I_{2,1}(h,\varepsilon)$. Using the non-negativity of $\rho_{\varepsilon}$ and the preservation of mass, and expanding the square via the polarization identity, we obtain:
\begin{align*}
I_{2,1}(h,\varepsilon) 
=&\ \mathbb{E} \int_t^{t+h} \int_{\mathbb{T}^d} \left( 
|\nabla\phi\, S(t+h - s)\, \sigma_{n(\varepsilon)}(\rho_{\varepsilon}(s))|^2 
- |\nabla\phi\, \sigma_{n(\varepsilon)}(\rho_{\varepsilon}(t))|^2 
\right) dx\, ds \\
=&\ \mathbb{E} \int_t^{t+h} \int_{\mathbb{T}^d} \left( 
\nabla\phi\, S(t+h - s)\, \sigma_{n(\varepsilon)}(\rho_{\varepsilon}(s)) 
+ \nabla\phi\, \sigma_{n(\varepsilon)}(\rho_{\varepsilon}(t)) 
\right) \cdot \\
&\hspace{3cm} \left( 
\nabla\phi\, S(t+h - s)\, \sigma_{n(\varepsilon)}(\rho_{\varepsilon}(s)) 
- \nabla\phi\, \sigma_{n(\varepsilon)}(\rho_{\varepsilon}(t)) 
\right) dx\, ds \\
\leq&\ C(\phi) \mathbb{E} \int_t^{t+h} \| S(t+h - s)\, \sigma_{n(\varepsilon)}(\rho_{\varepsilon}(s)) 
- \sigma_{n(\varepsilon)}(\rho_{\varepsilon}(t)) \|_{L^2(\mathbb{T}^d)} ds.
\end{align*}

We then split the difference inside the $L^1$-norm into two terms:
\begin{align*}
I_{2,1}(h,\varepsilon) 
\leq&\ C(\phi) \mathbb{E} \int_t^{t+h} \| S(t+h - s)\, \sigma_{n(\varepsilon)}(\rho_{\varepsilon}(s)) 
- S(t+h - s)\, \sigma_{n(\varepsilon)}(\rho_{\varepsilon}(t)) \|_{L^2(\mathbb{T}^d)} ds \\
&+ C(\phi) \mathbb{E} \int_t^{t+h} \| S(t+h - s)\, \sigma_{n(\varepsilon)}(\rho_{\varepsilon}(t)) 
- \sigma_{n(\varepsilon)}(\rho_{\varepsilon}(t)) \|_{L^2(\mathbb{T}^d)} ds \\
=:&\ I_{2,1,1}(h,\varepsilon) + I_{2,1,2}(h,\varepsilon).
\end{align*}
We now estimate the term $I_{2,1,1}(h,\varepsilon)$. By applying convolutional Young's inequality to the heat semigroup $S(t)$, we obtain
\begin{align*}
I_{2,1,1}(h,\varepsilon) 
&\leq C(\phi) \mathbb{E} \int_t^{t+h} \| \sigma_{n(\varepsilon)}(\rho_{\varepsilon}(s)) - \sigma_{n(\varepsilon)}(\rho_{\varepsilon}(t)) \|_{L^2(\mathbb{T}^d)} ds.
\end{align*}
Using the Lipschitz continuity of $\sigma_{n(\varepsilon)}$ and the triangle inequality, we decompose the integrand into microscopic and macroscopic fluctuations:
\begin{align*}
I_{2,1,1}(h,\varepsilon) 
\leq\ & C(\phi) \| \sigma_{n(\varepsilon)}' \|_{L^{\infty}([0,1])} \mathbb{E} \int_t^{t+h} \| \rho_{\varepsilon}(s) - \bar{\rho}(s) \|_{L^2(\mathbb{T}^d)} ds \\
& + C(\phi) \| \sigma_{n(\varepsilon)}' \|_{L^{\infty}([0,1])} \mathbb{E} \| \rho_{\varepsilon}(t) - \bar{\rho}(t) \|_{L^2(\mathbb{T}^d)}\, h \\
& + C(\phi) \| \sigma_{n(\varepsilon)}' \|_{L^{\infty}([c,1-c])} \mathbb{E} \int_t^{t+h} \| \bar{\rho}(s) - \bar{\rho}(t) \|_{L^2(\mathbb{T}^d)} ds \\
=:&\ I_{2,1,1,1}(h,\varepsilon) + I_{2,1,1,2}(h,\varepsilon) + I_{2,1,1,3}(h,\varepsilon).
\end{align*}

To bound $I_{2,1,1,1}(h,\varepsilon)$, we apply Lemma \ref{lem-LLN-regular} to see that 
\begin{align*}
I_{2,1,1,1}(h,\varepsilon) 
&\leq C(\phi) \| \sigma_{n(\varepsilon)}' \|_{L^{\infty}([0,1])} 
\left( \varepsilon \delta(\varepsilon)^{-d-2} \right) h.
\end{align*}

For the second term, $I_{2,1,1,2}(h,\varepsilon)$, we use the same type of fluctuation estimate at a fixed time $t$:
\begin{align*}
I_{2,1,1,2}(h,\varepsilon) 
&\leq C(\phi) \left( \varepsilon \delta(\varepsilon)^{-d-2} \right) h.
\end{align*}

Finally, for the macroscopic contribution $I_{2,1,1,3}(h,\varepsilon)$, we use the temporal regularity of the macroscopic density $\bar{\rho}$, which is governed by the heat equation. In particular, the smoothing effect of the Laplacian implies that $\bar{\rho} \in W^{1,\infty}([0,T]; L^1(\mathbb{T}^d))$. Therefore, we obtain
\begin{align*}
I_{2,1,1,3}(h,\varepsilon) 
&\leq C(\phi) \| \sigma_{n(\varepsilon)}' \|_{L^{\infty}([c,1-c])} 
\| \bar{\rho} \|_{W^{1,\infty}([0,T]; L^2(\mathbb{T}^d))} h^2 \\
&\leq C(\phi, c) \| \bar{\rho} \|_{L^{\infty}([0,T]; H^2(\mathbb{T}^d))} h^2 
\leq C(\phi, c, \rho_0) h^2.
\end{align*}

We next estimate the term $I_{2,1,2}(h,\varepsilon)$. We have the following decomposition: 
\begin{align*}
I_{2,1,2}(h,\varepsilon)
\leq\ & C(\phi) \mathbb{E} \int_t^{t+h} \| S(t+h-s)\sigma_{n(\varepsilon)}(\rho_{\varepsilon}(t)) - S(t+h-s)\sigma_{n(\varepsilon)}(\bar{\rho}(t)) \|_{L^2(\mathbb{T}^d)} ds \\
&+ C(\phi) \mathbb{E} \int_t^{t+h} \| S(t+h-s)\sigma_{n(\varepsilon)}(\bar{\rho}(t)) - \sigma_{n(\varepsilon)}(\bar{\rho}(t)) \|_{L^2(\mathbb{T}^d)} ds\\
&+C(\phi)\mathbb{E}\int^{t+h}_t\|\sigma_{n(\varepsilon)}(\bar{\rho}(t))-\sigma_{n(\varepsilon)}(\rho_{\varepsilon}(t))\|_{L^2(\mathbb{T}^d)}ds. 
\end{align*}
Each term is estimated as follows, the first term and the third term can be controlled via the law of large numbers and Lipschitz continuity of $ \sigma_{n(\varepsilon)} $. The second term is controlled by the strong continuity of the semigroup $ S(t) $. Putting these together, applying H\"older's inequality and Lemma \ref{lem-LLN-regular}, we obtain
\begin{align*}
I_{2,1,2}(h,\varepsilon)\leq&C(\phi)\mathbb{E}\int^{t+h}_t\| \sigma_{n(\varepsilon)}(\bar{\rho}(t)) - \sigma_{n(\varepsilon)}(\rho_{\varepsilon}(t)) \|_{L^1(\mathbb{T}^d)} ds+C(\phi) h^2\\
\leq&C(\phi) \|\sigma_{n(\varepsilon)}'\|_{L^{\infty}([0,1])}(\sup_{t\in[0,T]}\mathbb{E}\|\rho_{\varepsilon}(t)-\bar{\rho}(t)\|_{L^2(\mathbb{T}^d)})h+C(\phi) h^2\\
\leq& C(\phi) \|\sigma_{n(\varepsilon)}'\|_{L^{\infty}([0,1])} 
\left( \varepsilon \delta(\varepsilon)^{-d-2} \right) h 
+ C(\phi) h^2.
\end{align*}

\medskip

We now consider $ I_{2,2}(h,\varepsilon) $. By the triangle inequality, we write
\begin{align*}
I_{2,2}(h,\varepsilon)
&\leq C(\phi) \mathbb{E} \int_t^{t+h} \int_{\mathbb{T}^d} \left| \sigma_{n(\varepsilon)}(\rho_{\varepsilon}(s))^2 - \bar{\rho}(t)(1 - \bar{\rho}(t)) \right| dx ds \\
&\leq C(\phi) \mathbb{E} \int_t^{t+h} \int_{\mathbb{T}^d} \left| \sigma_{n(\varepsilon)}(\rho_{\varepsilon}(s))^2 - \sigma_{n(\varepsilon)}(\bar{\rho}(s))^2 \right| dx ds \\
&\quad + \int_t^{t+h} \int_{\mathbb{T}^d} \left| \sigma_{n(\varepsilon)}(\bar{\rho}(s))^2 - \sigma_{n(\varepsilon)}(\bar{\rho}(t))^2 \right| dx ds \\
&\quad + h \int_{\mathbb{T}^d} \left| \sigma_{n(\varepsilon)}(\bar{\rho}(t))^2 - \bar{\rho}(t)(1 - \bar{\rho}(t)) \right| dx \\
&=: I_{2,2,1}(h,\varepsilon) + I_{2,2,2}(h,\varepsilon) + I_{2,2,3}(h,\varepsilon).
\end{align*}

For $ I_{2,2,1}(h,\varepsilon) $, we apply the Lipschitz bound on $ \sigma_{n(\varepsilon)}^2 $ and use the fluctuation estimate:
\begin{align*}
I_{2,2,1}(h,\varepsilon)
&\leq C(\phi) \| \sigma_{n(\varepsilon)} \sigma_{n(\varepsilon)}' \|_{L^{\infty}([0,1])} 
\mathbb{E} \int_t^{t+h} \| \rho_{\varepsilon}(s) - \bar{\rho}(s) \|_{L^1(\mathbb{T}^d)} ds \\
&\leq C(\phi) \| \sigma_{n(\varepsilon)} \sigma_{n(\varepsilon)}' \|_{L^{\infty}([0,1])} 
\left( \sup_{s \in [0,T]} \mathbb{E} \| \rho_{\varepsilon}(s) - \bar{\rho}(s) \|_{L^2(\mathbb{T}^d)}^2 \right)^{1/2} h \\
&\leq C(\phi) \left( \varepsilon \delta(\varepsilon)^{-d-2} \right) h.
\end{align*}

For $ I_{2,2,2}(h,\varepsilon) $, we exploit the temporal regularity of the macroscopic profile $ \bar{\rho} $:
\begin{align*}
I_{2,2,2}(h,\varepsilon)
&\leq C(\phi) \| \sigma_{n(\varepsilon)}\sigma_{n(\varepsilon)}' \|_{L^{\infty}([c,1-c])} 
\| \bar{\rho} \|_{W^{1,\infty}([0,T]; L^1(\mathbb{T}^d))} h^2 \\
&\leq C(\phi, c) \| \bar{\rho} \|_{L^{\infty}([0,T]; W^{2,1}(\mathbb{T}^d))} h^2 
\leq C(\phi, c, \rho_0) h^2.
\end{align*}

Lastly, for $ I_{2,2,3}(h,\varepsilon) $, the convergence of the approximation sequence $ \sigma_{n(\varepsilon)}^2 \to \zeta(1 - \zeta) $ implies
\begin{align*}
I_{2,2,3}(h,\varepsilon)
&\leq C(\phi) h \sup_{\zeta \in [c, 1 - c]} \left| \sigma_{n(\varepsilon)}(\zeta)^2 - \zeta(1 - \zeta) \right|.
\end{align*}
We observe that, the mollified noise coefficient $ \sigma_{n(\varepsilon)}(\cdot) $ is chosen such that, for sufficiently small $\varepsilon>0$, 
$$
\sigma_{n(\varepsilon)}(\zeta) = \sqrt{\zeta(1 - \zeta)} \quad \text{for all } \zeta \in [c, 1 - c],
$$
then for sufficiently small $ \varepsilon $, the approximation error satisfies
$$
\sup_{\zeta \in [c, 1 - c]} \left| \sigma_{n(\varepsilon)}(\zeta)^2 - \zeta(1 - \zeta) \right| = 0.
$$

Finally, we estimate the term $ I_3(h,\varepsilon) $. By applying integration by parts, we rewrite the expression as
\begin{align*}
I_3(h,\varepsilon)
= \frac{\varepsilon}{4} F_{1,\delta(\varepsilon)}^2 
\mathbb{E} \left( 
\int_0^{t+h} \langle \Delta \phi, S(t+h - s) \Sigma_{n(\varepsilon)}(\rho_{\varepsilon}(s)) \rangle ds 
- \int_0^t \langle \Delta \phi, S(t - s) \Sigma_{n(\varepsilon)}(\rho_{\varepsilon}(s)) \rangle ds 
\right)^2,
\end{align*}
where the function $ \Sigma_{n(\varepsilon)} \colon [0,1] \to \mathbb{R} $ is defined by
$$
\Sigma_{n(\varepsilon)}(\zeta) := \int_0^{\zeta} \sigma_{n(\varepsilon)}'(\zeta')^2 d\zeta'.
$$

Applying the triangle inequality and splitting the difference of integrals, we estimate:
\begin{align*}
I_3(h,\varepsilon)
&\lesssim \varepsilon \delta(\varepsilon)^{-2d} 
\mathbb{E} \left( \int_0^t \langle \Delta \phi, \left( S(t+h - s) - S(t - s) \right) \Sigma_{n(\varepsilon)}(\rho_{\varepsilon}(s)) \rangle ds \right)^2 \\
&\quad + \varepsilon \delta(\varepsilon)^{-2d} 
\mathbb{E} \left( \int_t^{t+h} \langle \Delta \phi, S(t+h - s) \Sigma_{n(\varepsilon)}(\rho_{\varepsilon}(s)) \rangle ds \right)^2.
\end{align*}

Since $ \Sigma_{n(\varepsilon)} $ is uniformly Lipschitz, and the heat semigroup is strongly continuous, both terms are controlled similarly. Using boundedness of $ \Delta \phi $ and standard smoothing properties of $ S(t) $, we obtain:
\begin{align*}
I_3(h,\varepsilon) 
\lesssim C(\phi) \varepsilon \delta(\varepsilon)^{-2d} 
\| \sigma_{n(\varepsilon)}' \|_{L^{\infty}([0,1])}^4 h^2.
\end{align*}

Combining the bounds on $ I_1(h,\varepsilon) $, $ I_2(h,\varepsilon) $, and $ I_3(h,\varepsilon) $, we conclude the proof.

\end{proof}

For the case of Dean-Kawasaki equation, the identity result can be obtained by using the same approach. 
\begin{theorem}
Let $\sigma(\zeta)=\sqrt{\zeta}$. Assume that the initial condition $\rho_0$ satisfies Assumption \ref{Assump-initial-2} and the regularization sequence $(\sigma_n(\cdot))_{n \geq 1}$ satisfies Assumption \ref{Assump-sigman-sqrt} with $\sigma_n\rightarrow\sigma$ in $C^1_{loc}((0,\infty))$. Let $c$ denote the lower bound of the initial data, as specified in Assumption \ref{Assump-initial-2}. Let $\phi \in C^{\infty}(\mathbb{T}^d)$, and define $\rho_{\varepsilon}^{\phi} := \langle \rho_{\varepsilon}, \phi \rangle$. The corresponding fluctuation field is given by
\begin{align*}
\bar{\rho}_{\varepsilon}^{1,\phi} := \varepsilon^{-1/2}(\rho_{\varepsilon}^{\phi} - \bar{\rho}^{\phi}).
\end{align*}
Then, for every $t > h > 0$, the following estimate holds:
\begin{align*}
\left| \frac{1}{h} \mathbb{E} \left( \bar{\rho}_{\varepsilon}^{1,\phi}(t+h) - \bar{\rho}_{\varepsilon}^{1,\phi}(t) \right)^2 
- \left\langle \nabla \phi, \bar{\rho}(t) \nabla \phi \right\rangle \right|
\leq C(\phi, \rho_0, t) \cdot \mathrm{Error}(h, \varepsilon, \delta(\varepsilon), n(\varepsilon)),
\end{align*}
where the error term is given by
\begin{align*}
\mathrm{Error}(h, \varepsilon, \delta(\varepsilon), n(\varepsilon)) =\ & h 
+ \varepsilon \delta(\varepsilon)^{-2d} \|\sigma_{n(\varepsilon)}'\|_{L^{\infty}([0,\infty))}^4 h +\varepsilon \delta(\varepsilon)^{-d-2}. 
\end{align*}
\end{theorem}

\section{Fluctuating hydrodynamics with irregular coefficients: asymptotic behaviors}\label{sec-6}
In this section, we study the fluctuating hydrodynamic SPDEs with an irregular coefficient:
\begin{equation}
d\rho_{\varepsilon} = \frac{1}{2}\Delta\rho_{\varepsilon}\,dt - \varepsilon^{1/2}\nabla\cdot(\sigma(\rho_{\varepsilon})\circ dW_{\delta(\varepsilon)}).
\end{equation}
As a concrete example, we first focus on the case $\sigma(\zeta) = \sqrt{\zeta}$. By computing the Stratonovich-to-It\^o corrections \cite{FG24}, we can rewrite the equation in It\^o form: 
\begin{equation}\label{SPDE-irregular-sec-6}
d\rho_{\varepsilon} = \frac{1}{2}\Delta\rho_{\varepsilon}\,dt - \varepsilon^{1/2}\nabla\cdot(\sigma(\rho_{\varepsilon})\,dW_{\delta(\varepsilon)}) + \frac{\varepsilon}{2}F_{1,\delta(\varepsilon)}\nabla\cdot(\sigma'(\rho_{\varepsilon})^2\nabla\rho_{\varepsilon})\,dt.
\end{equation}

We begin by establishing a law of large numbers result for the family $\{\rho_\varepsilon\}_{\varepsilon>0}$.

\begin{lemma}\label{LLN-irregular}
Let $\sigma(\zeta)=\sqrt{\zeta}$. Assume that $\rho_0$ satisfies Assumption \ref{Assump-initial-2}. Then, under the scaling regime $$\lim_{\varepsilon\rightarrow0}\varepsilon\delta(\varepsilon)^{-d-2}=0,$$ for every $\delta' > 0$ we have
\begin{align*}
\lim_{\varepsilon \rightarrow 0} \mathbb{P} \left( \|\rho_{\varepsilon} - \bar{\rho}\|_{L^{\infty}([0,T]; L^1(\mathbb{T}^d))} > \delta' \right) = 0.
\end{align*}
\end{lemma}

\begin{proof}
The proof proceeds by adapting the arguments developed in \cite[Proposition~3.5 and Theorems~3.9, 3.10]{DFG}, where the law of large numbers and the central limit theorem for the fluctuating SSEP are established with convergence in probability, based on the Moser iteration technique. For more general coefficients, including the Dean-Kawasaki case, we refer the reader to \cite{GWZ24} and \cite{CF23}. 
\end{proof}
Let $\chi_{\varepsilon} = I_{\{0 < \zeta < \rho_{\varepsilon}\}}$ denote the kinetic function associated with $\rho_\varepsilon$, and let $p_{\varepsilon}$ be the corresponding kinetic measure.

\begin{lemma}\label{boundedness-kineticmeasure}
Let $\sigma(\zeta)=\sqrt{\zeta}$. Assume that $\rho_0$ satisfies Assumption \ref{Assump-initial-2}. For every $\varepsilon > 0$, let $p_{\varepsilon}$ be the kinetic measure associated with $\rho_\varepsilon$. Then we have the estimate
\begin{align*}
\mathbb{E}\left(h^{-1 + \varepsilon'} p_{\varepsilon}\left([t, t + h] \times \mathbb{T}^d \times \left[\frac{h^{1 - \varepsilon'}}{2}, h^{1 - \varepsilon'}\right]\right)\right)^2 \leq C(\rho_0)\varepsilon F_{1,\delta(\varepsilon)},
\end{align*}
for all $t \in [0, T]$, $\varepsilon' \in (0, 1/4)$, $\varepsilon \in (0, 1)$, and $h \in (0, 1)$. 
\end{lemma}

\begin{proof}
Let us consider the test function
$$
\varphi_h(\zeta) = \int_0^{\zeta} 2 h^{-1 + \varepsilon'} I_{\left[\frac{h^{1 - \varepsilon'}}{2}, h^{1 - \varepsilon'}\right]}(\zeta') \, d\zeta'. 
$$
This choice localizes the kinetic formulation in a narrow band of small densities. Following the approach in \cite[Proposition 4.5]{FG24}, we estimate
\begin{align*}
&\mathbb{E}\left(2 h^{-1 + \varepsilon'} p_{\varepsilon}\left([t, t + h] \times \mathbb{T}^d \times \left[\frac{h^{1 - \varepsilon'}}{2}, h^{1 - \varepsilon'}\right]\right)\right)^2 \\
&\quad\lesssim \mathbb{E}\left( \int_{\mathbb{R}_+ \times \mathbb{T}^d} \chi_{\varepsilon} \varphi_h \, dx d\zeta \Big|_{s = t}^{s = t + h} \right)^2 \\
&\qquad+ \varepsilon^2 h^{-2 + 2\varepsilon'} \, \mathbb{E}\left( \int_t^{t + h} \int_{\mathbb{T}^d} I_{\left\{ \frac{h^{1 - \varepsilon'}}{2} < \rho_{\varepsilon} < h^{1 - \varepsilon'} \right\}} F_{3,\delta(\varepsilon)} \sigma(\rho_{\varepsilon})^2 \, dx ds \right)^2 \\
&\qquad+ \varepsilon \, \mathbb{E} \left( \int_t^{t + h} \int_{\mathbb{T}^d} \varphi_h(\rho_{\varepsilon}) \nabla \cdot (\sigma(\rho_{\varepsilon}) \, dW_{\delta(\varepsilon)}) \right)^2.
\end{align*}

The first two terms on the right-hand side are controlled by \cite[Proposition 4.5]{FG24}, yielding
\begin{align*}
\mathbb{E}\left( \int_{\mathbb{R}_+ \times \mathbb{T}^d} \chi_{\varepsilon} \varphi_h \, dx d\zeta \Big|_{s = t}^{s = t + h} \right)^2 
+ \varepsilon^2 h^{-2 + 2\varepsilon'} \, \mathbb{E}\left( \int_t^{t + h} \int_{\mathbb{T}^d} I_{\left\{ \frac{h^{1 - \varepsilon'}}{2} < \rho_{\varepsilon} < h^{1 - \varepsilon'} \right\}} F_3 \sigma(\rho_{\varepsilon})^2 \, dx ds \right)^2 
\leq C.
\end{align*}

To estimate the stochastic integral, we apply It\^o's isometry. Observe that the derivative of $\varphi_h(\rho_{\varepsilon})$ is given by
$$
\varphi_h'(\rho_{\varepsilon}) = 2 h^{-1 + \varepsilon'} I_{\left[\frac{h^{1 - \varepsilon'}}{2}, h^{1 - \varepsilon'}\right]}(\rho_{\varepsilon}),
$$
and hence,
\begin{align*}
&\varepsilon \, \mathbb{E} \left( \int_t^{t + h} \int_{\mathbb{T}^d} \varphi_h(\rho_{\varepsilon}) \nabla \cdot (\sigma(\rho_{\varepsilon}) \, dW_{\delta(\varepsilon)}) \right)^2 \\
&\quad= \varepsilon F_{1,\delta(\varepsilon)} \, \mathbb{E} \int_t^{t + h} \int_{\mathbb{T}^d} \left| \varphi_h'(\rho_{\varepsilon}) \nabla \rho_{\varepsilon} \sigma(\rho_{\varepsilon}) \right|^2 dx ds \\
&\quad\leq 4 \varepsilon F_{1,\delta(\varepsilon)} h^{-2 + 2\varepsilon'} \, \mathbb{E} \int_t^{t + h} \int_{\left\{ \frac{h^{1 - \varepsilon'}}{2} < \rho_{\varepsilon} < h^{1 - \varepsilon'} \right\}} |\nabla \rho_{\varepsilon}|^2 \sigma(\rho_{\varepsilon})^2 \, dx ds.
\end{align*}
Using the bound $\sigma(\zeta) = \sqrt{\zeta}$ and the inequality $|\nabla \rho_{\varepsilon}|^2 \leq 4 |\nabla \sqrt{\rho_{\varepsilon}}|^2 \rho_{\varepsilon}$, we deduce
$$
I_{\{\frac{h^{1-\varepsilon'}}{2}<\rho_{\varepsilon}<h^{1-\varepsilon'}\}}|\nabla \rho_{\varepsilon}|^2 \sigma(\rho_{\varepsilon})^2 \lesssim I_{\{\frac{h^{1-\varepsilon'}}{2}<\rho_{\varepsilon}<h^{1-\varepsilon'}\}}\rho_{\varepsilon}^2 |\nabla \sqrt{\rho_{\varepsilon}}|^2 \lesssim h^{2(1 - \varepsilon')} |\nabla \sqrt{\rho_{\varepsilon}}|^2.
$$
Thus,
\begin{align*}
\varepsilon \, \mathbb{E} \left( \int_t^{t + h} \int_{\mathbb{T}^d} \varphi_h(\rho_{\varepsilon}) \nabla \cdot (\sigma(\rho_{\varepsilon}) \, dW_{\delta(\varepsilon)}) \right)^2 
\lesssim \varepsilon F_{1,\delta(\varepsilon)}\mathbb{E} \int_t^{t + h} \|\nabla \sqrt{\rho_{\varepsilon}}\|_{L^2(\mathbb{T}^d)}^2 \, ds \leq C(\rho_0)\varepsilon F_{1,\delta(\varepsilon)}.
\end{align*}

Combining the estimates for all three terms yields the desired bound, and the proof is complete.
\end{proof}

\begin{theorem}\label{asymptotic-behavior}
Let $\sigma(\zeta)=\sqrt{\zeta}$. Assume that $\rho_0$ satisfies Assumption \ref{Assump-initial-2}. Let $\rho_{\varepsilon}$ be the renormalized kinetic solution of \eqref{SPDE-irregular}. Then, under the scaling regime $$\lim_{\varepsilon\rightarrow0}\varepsilon\delta(\varepsilon)^{-d-2}=0,$$ for every nonnegative test function $\phi \in C^{\infty}(\mathbb{T}^d)$ and every $t \in [0,T]$, the following asymptotic fluctuation identity holds:
\begin{align*}
	\lim_{\varepsilon \to 0} \liminf_{h \to 0} \frac{1}{h} \, \mathbb{E} \left( \left\langle \varepsilon^{-1/2}(\rho_{\varepsilon} - \bar{\rho})(t + h) - \varepsilon^{-1/2}(\rho_{\varepsilon} - \bar{\rho})(t), \phi \right\rangle \right)^2 = \langle \sigma(\bar{\rho}(t))^2 \nabla \phi, \nabla \phi \rangle.
\end{align*}
\end{theorem}
\begin{proof}
We begin by estimating the fluctuation of the centered and rescaled process. Observe that
\begin{align*}
&\frac{1}{h} \, \mathbb{E} \left( \left\langle \varepsilon^{-1/2}(\rho_{\varepsilon} - \bar{\rho})(t + h) - \varepsilon^{-1/2}(\rho_{\varepsilon} - \bar{\rho})(t), \phi \right\rangle \right)^2\\ 
\leq& \frac{1}{h} \, \mathbb{E} \left( \left\langle \varepsilon^{-1/2} \rho_{\varepsilon}(t + h) - \varepsilon^{-1/2} \rho_{\varepsilon}(t), \phi \right\rangle \right)^2 + \frac{1}{h} \, \mathbb{E} \left( \left\langle \varepsilon^{-1/2} \bar{\rho}(t + h) - \varepsilon^{-1/2} \bar{\rho}(t), \phi \right\rangle \right)^2 \\
&+\frac{2}{h}\mathbb{E}\left( \left\langle \varepsilon^{-1/2} \rho_{\varepsilon}(t + h) - \varepsilon^{-1/2} \rho_{\varepsilon}(t), \phi \right\rangle\left\langle \varepsilon^{-1/2} \bar{\rho}(t + h) - \varepsilon^{-1/2} \bar{\rho}(t), \phi \right\rangle \right)\\
\leq& \frac{1}{h} \, \mathbb{E} \left( \left\langle \varepsilon^{-1/2} \rho_{\varepsilon}(t + h) - \varepsilon^{-1/2} \rho_{\varepsilon}(t), \phi \right\rangle \right)^2 
+ C(\rho_0) \varepsilon^{-1} (h+h^{1/2}),
\end{align*}
where the last inequality follows from the smoothness of $\bar{\rho}$ in time and the preservation of mass for $\rho_{\varepsilon}$. 

Next, fix $\beta \in (0, 1/2)$ (to be specified later), and let $\eta_{\beta}$ be a smooth truncation function such that $\eta_{\beta}(\zeta) = 1$ for $\zeta \geq \beta$, $\eta_{\beta}(\zeta) = 0$ for $\zeta \leq \beta/2$, and satisfying the derivative bound $\eta'_{\beta} \lesssim \beta^{-1} I_{\left[\frac{\beta}{2}, \beta\right]}$. Let $\chi_{\varepsilon} = I_{\{0 < \zeta < \rho_{\varepsilon}\}}$ denote the kinetic function. Then we decompose
\begin{align}\label{expansion}
&\varepsilon^{-1/2} \left\langle \rho_{\varepsilon}(t + h) - \rho_{\varepsilon}(t), \phi \right\rangle\notag\\ 
=& \varepsilon^{-1/2} \left\langle \chi_{\varepsilon}(t + h) - \chi_{\varepsilon}(t), \eta_{\beta} \phi \right\rangle_{L^2(\mathbb{T}^d\times\mathbb{R}_+)}
+ \varepsilon^{-1/2} \left\langle \chi_{\varepsilon}(t + h) - \chi_{\varepsilon}(t), (1 - \eta_{\beta}) \phi \right\rangle_{L^2(\mathbb{T}^d\times\mathbb{R}_+)}.
\end{align}

We estimate the contribution from the region where $\rho_{\varepsilon} \leq \beta$. Note that
\begin{align*}
\frac{1}{h} \, \mathbb{E} \left( \varepsilon^{-1/2} \left\langle \chi_{\varepsilon}(t + h) - \chi_{\varepsilon}(t), (1 - \eta_{\beta}) \phi \right\rangle_{L^2(\mathbb{T}^d\times\mathbb{R}_+)} \right)^2 &\leq \frac{1}{h} \varepsilon^{-1} \, \mathbb{E} \left( \left\langle \rho_{\varepsilon}(t + h) \wedge \beta - \rho_{\varepsilon}(t) \wedge \beta, \phi \right\rangle \right)^2 \\
&\leq C(\phi) \varepsilon^{-1} \beta^2 h^{-1}.
\end{align*}

We now apply the kinetic formulation to the regular part. Using the kinetic equation, we write
\begin{align*}
&\varepsilon^{-1/2} \left\langle \chi_{\varepsilon}(t + h) - \chi_{\varepsilon}(t), \eta_{\beta} \phi \right\rangle_{L^2(\mathbb{T}^d\times\mathbb{R}_+)}\\ 
&= -\varepsilon^{-1/2} \frac{1}{2}\int_t^{t + h} \int_{\mathbb{T}^d} \nabla \phi \, \eta_{\beta}(\rho_{\varepsilon}) \nabla \rho_{\varepsilon} \, dx ds 
- \varepsilon^{-1/2} \int_t^{t + h} \int_{\mathbb{T}^d \times \mathbb{R}_+} \phi \, \eta_{\beta}' \, dp_{\varepsilon} \\
&\quad+ \frac{1}{2} \varepsilon^{1/2} \int_t^{t + h} \int_{\mathbb{T}^d} \phi \, \eta_{\beta}'(\rho_{\varepsilon}) \, \sigma(\rho_{\varepsilon})^2 F_{3,\delta(\varepsilon)} \, dx ds \\
&\quad- \int_t^{t + h} \int_{\mathbb{T}^d} \phi \, \eta_{\beta}(\rho_{\varepsilon}) \, \nabla \cdot (\sigma(\rho_{\varepsilon}) \, dW_{\delta(\varepsilon)}(s)) \, dx \\
&\quad- \frac{\varepsilon^{1/2}}{2} F_{1,\delta(\varepsilon)} \int_t^{t + h} \int_{\mathbb{T}^d} \nabla\phi \, \eta_{\beta}(\rho_{\varepsilon})  \cdot  \sigma'(\rho_{\varepsilon})^2 \nabla \rho_{\varepsilon}  \, dx ds \\
&=: I_1 + I_2 + I_3 + I_4 + I_5.
\end{align*}
Each of the terms $ I_1 $ through $ I_5 $ will now be estimated separately. 

\medskip
\noindent\textbf{Estimate of $ I_1 $.} By integration by parts, we write:
\begin{align*}
\frac{1}{h} \, \mathbb{E} I_1^2 
&= \frac{1}{h} \, \varepsilon^{-1} \, \mathbb{E} \left( \frac{1}{2}\int_t^{t+h} \int_{\mathbb{T}^d} \nabla \cdot (\nabla \phi \, \eta_{\beta}(\rho_{\varepsilon})) \, \rho_{\varepsilon} \, dx ds \right)^2 \\
&= \frac{1}{h} \, \varepsilon^{-1} \, \mathbb{E} \left( \frac{1}{2}\int_t^{t+h} \int_{\mathbb{T}^d} \left[ \Delta \phi \, \eta_{\beta}(\rho_{\varepsilon}) + \nabla \phi \, \eta_{\beta}'(\rho_{\varepsilon}) \, \nabla \rho_{\varepsilon} \right] \rho_{\varepsilon} \, dx ds \right)^2.
\end{align*}
We split the resulting expression into the smooth term and the one involving $\eta_\beta'$:
\begin{align*}
\frac{1}{h} \, \mathbb{E} I_1^2 
&\lesssim C(\phi)\|\rho_0\|_{L^1(\mathbb{T}^d)}^2 h \varepsilon^{-1} \\
&\quad+ C(\phi) \frac{1}{h} \varepsilon^{-1} \, \mathbb{E} \left( \int_t^{t+h} \int_{\mathbb{T}^d} \nabla \phi \, \eta_{\beta}'(\rho_{\varepsilon}) \, 2 \sqrt{\rho_{\varepsilon}} \, \nabla \sqrt{\rho_{\varepsilon}} \, dx ds \right)^2.
\end{align*}

For the last term, we use the fact that $ \eta_{\beta}' \lesssim \beta^{-1} $ and $ \sqrt{\rho_{\varepsilon}} \leq \sqrt{\beta} $ in the support of $ \eta_{\beta}' $, yielding
\begin{align*}
\frac{1}{h} \, \mathbb{E} I_1^2 
&\lesssim C(\phi, \rho_0) h \varepsilon^{-1} 
+ C(\phi) \frac{1}{h} \varepsilon^{-1} \beta \, \mathbb{E} \left( \int_t^{t+h} \int_{\mathbb{T}^d} \left| \nabla \sqrt{\rho_{\varepsilon}} \right| dx ds \right)^2 \\
&\lesssim C(\phi, \rho_0) h \varepsilon^{-1} + C(\phi, \rho_0) \varepsilon^{-1} \beta.
\end{align*}

\medskip
\noindent\textbf{Estimate of $ I_2 $.} For the term involving the kinetic measure, we apply Lemma~\ref{boundedness-kineticmeasure}. Fix $ \varepsilon' \in (0, 1/4) $, and choose $ \beta = \beta(h) = h^{1 - \varepsilon'} $. Then:
\begin{align*}
\frac{1}{h} \, \mathbb{E} I_2^2 
&\leq-\varepsilon^{-1/2}\frac{1}{h}\mathbb{E}\Big(\int^{t+h}_t\int_{\mathbb{T}^d}\phi\frac{1}{\beta}I_{\{[\beta/2,\beta]\}}(\rho_{\varepsilon})2|\nabla\sqrt{\rho_{\varepsilon}}|^2\rho_{\varepsilon}dxds\Big)^2 \\
&\leq-\varepsilon^{-1/2}\frac{1}{h}\Big(\mathbb{E}\int^{t+h}_t\int_{\mathbb{T}^d}\phi\frac{1}{\beta}I_{\{[\beta/2,\beta]\}}(\rho_{\varepsilon})2|\nabla\sqrt{\rho_{\varepsilon}}|^2\rho_{\varepsilon}dxds\Big)^2 \\
&\leq C(\phi) \varepsilon^{-1}  \Big(\frac{1}{h^{1/2}}\mathbb{E}\int^{t+h}_t\int_{\mathbb{T}^d}\int_{[h^{1-\varepsilon'}/2,h^{1-\varepsilon'}]}|\nabla\sqrt{\rho_{\varepsilon}}|^2dm_{\varepsilon}\Big)^2,
\end{align*}
where $dm_{\varepsilon}:=\delta_{0}(\rho_{\varepsilon}-\zeta)d\zeta dxds$. We also note that by the distributional inequality $p_{\varepsilon}\geq\frac{1}{2}\delta_0(\rho_{\varepsilon}-\zeta)|\nabla\rho_{\varepsilon}|^2$ and Lemma \ref{boundedness-kineticmeasure}, the second order moment $\mathbb{E}\Big(\int^{t+h}_t\int_{\mathbb{T}^d}\phi\frac{1}{\beta}I_{\{[\beta/2,\beta]\}}(\rho_{\varepsilon})2|\nabla\sqrt{\rho_{\varepsilon}}|^2\rho_{\varepsilon}dxds\Big)^2$ in the first line is finite. Furthermore, with the help of \cite[Lemma 7]{FG23}, we have that 
\begin{align*}
\liminf_{h\rightarrow0}\frac{1}{h}\mathbb{E}I_2^2\leq C(\phi) \varepsilon^{-1}  \liminf_{h\rightarrow0}\Big(\frac{1}{h^{1/2}}\mathbb{E}\int^{t+h}_t\int_{\mathbb{T}^d}\int_{[h^{1-\varepsilon'}/2,h^{1-\varepsilon'}]}|\nabla\sqrt{\rho_{\varepsilon}}|^2dm_{\varepsilon}\Big)^2=0. 	
\end{align*}

\medskip
\noindent\textbf{Estimate of $ I_3 $.} For the It\^o correction term, we directly estimate:
\begin{align*}
\frac{1}{h} \, \mathbb{E} I_3^2 
&= \frac{1}{4} \varepsilon \frac{1}{h} \, \mathbb{E} \left( \int_t^{t+h} \int_{\mathbb{T}^d} \phi \, \eta_{\beta}'(\rho_{\varepsilon}) \, \sigma(\rho_{\varepsilon})^2 F_{3,\delta(\varepsilon)} \, dx ds \right)^2 \\
&\lesssim C(\phi) \varepsilon \|F_{3,\delta(\varepsilon)}\|_{L^1(\mathbb{T}^d)}^2 h.
\end{align*}

\noindent\textbf{Estimate of $ I_4 $.} We apply It\^o's isometry to obtain the identity
\begin{align*}
	\frac{1}{h}\mathbb{E}I_4^2=&\frac{1}{h}\mathbb{E}\int_{t}^{t+h}\|\eta_{\delta(\varepsilon)} \ast (\nabla(\phi \eta_{\beta}(\rho_{\varepsilon})) \sigma(\rho_{\varepsilon}))\|_{L^2(\mathbb{T}^d)}^2 \, ds \\
	=&\frac{1}{h}\mathbb{E}\int_{t}^{t+h} \Big[ \|\eta_{\delta(\varepsilon)} \ast (\nabla(\phi \eta_{\beta}(\rho_{\varepsilon})) \sigma(\rho_{\varepsilon}))-\eta_{\delta(\varepsilon)} \ast (\nabla(\phi \eta_{\beta}(\bar{\rho})) \sigma(\bar{\rho}))\|_{L^2(\mathbb{T}^d)}^2 \Big] ds \\
	&+ \frac{1}{h} \int_{t}^{t+h} \Big[ \|\eta_{\delta(\varepsilon)} \ast (\nabla(\phi \eta_{\beta}(\bar{\rho})) \sigma(\bar{\rho}))-\nabla(\phi \eta_{\beta}(\bar{\rho})) \sigma(\bar{\rho})\|_{L^2(\mathbb{T}^d)}^2 \Big] ds \\
	&+ \frac{1}{h} \int_{t}^{t+h} \Big[ \|\nabla(\phi \eta_{\beta}(\bar{\rho})) \sigma(\bar{\rho})-\nabla\phi \sigma(\bar{\rho})\|_{L^2(\mathbb{T}^d)}^2 \Big] ds \\
	&+ \frac{1}{h} \int_{t}^{t+h} \|\nabla\phi \sigma(\bar{\rho})-\nabla\phi \sigma(\bar{\rho}(t))\|_{L^2(\mathbb{T}^d)}^2ds + \|\nabla\phi \sigma(\bar{\rho}(t))\|_{L^2(\mathbb{T}^d)}^2+\text{correlated terms} \\
	=:& I_{4,1} + I_{4,2} + I_{4,3} + I_{4,4} + C(\phi,\bar{\rho})(I_{4,1}^{1/2} + I_{4,2}^{1/2} + I_{4,3}^{1/2} + I_{4,4}^{1/2})+ \|\nabla\phi \sigma(\bar{\rho}(t))\|_{L^2(\mathbb{T}^d)}^2.
\end{align*}

Here we use `correlated terms' to denote the non-quadratic term in the quadratic expansion, precisely, $(a_1+...+a_n)^2=a_1^2+...+a_n^2+{\text{correlated terms}}$. These terms produce lower order terms compare to $I_{4,1} + I_{4,2} + I_{4,3} + I_{4,4}$, therefore we do not formulate the precise expressions of them, and focus on the estimates of $I_{4,1} + I_{4,2} + I_{4,3} + I_{4,4}$. 

We estimate each term in turn. For $I_{4,1}$, using Young's convolution inequality, we obtain
\begin{align*}
I_{4,1} \lesssim& \frac{1}{h} \mathbb{E} \int_t^{t+h} \|\eta_{\delta} \ast \big( \nabla(\phi \eta_{\beta}(\rho_{\varepsilon})) \sigma(\rho_{\varepsilon}) - \nabla(\phi \eta_{\beta}(\bar{\rho})) \sigma(\bar{\rho}) \big) \|_{L^2(\mathbb{T}^d)}^2 ds \\
\leq& \frac{1}{h} \mathbb{E} \int_t^{t+h} \| \nabla(\phi \eta_{\beta}(\rho_{\varepsilon})) \sigma(\rho_{\varepsilon}) - \nabla(\phi \eta_{\beta}(\bar{\rho})) \sigma(\bar{\rho}) \|_{L^2(\mathbb{T}^d)}^2 ds \\
\leq& C(\phi) \frac{1}{h} \mathbb{E} \int_t^{t+h} \| \eta_{\beta}(\rho_{\varepsilon}) \sigma(\rho_{\varepsilon}) - \eta_{\beta}(\bar{\rho}) \sigma(\bar{\rho}) \|_{L^2(\mathbb{T}^d)}^2 ds \\
&+ C(\phi) \frac{1}{h} \mathbb{E} \int_t^{t+h} \| \eta_{\beta}'(\rho_{\varepsilon}) \nabla\rho_{\varepsilon} \sigma(\rho_{\varepsilon}) - \eta_{\beta}'(\bar{\rho}) \nabla\bar{\rho} \sigma(\bar{\rho}) \|_{L^2(\mathbb{T}^d)}^2 ds \\
=:& I_{4,1,1} + I_{4,1,2}.
\end{align*}

To bound $I_{4,1,1}$, we apply the triangle inequality:
\begin{align*}
I_{4,1,1} \leq& C(\phi) \frac{1}{h} \mathbb{E} \int_t^{t+h} \| I_{\{[\beta/2,\beta]\}}(\rho_{\varepsilon}) \sigma(\rho_{\varepsilon}) \|_{L^2(\mathbb{T}^d)}^2 ds + C(\phi) \frac{1}{h} \mathbb{E} \int_t^{t+h} \| \sigma(\rho_{\varepsilon}) - \sigma(\bar{\rho}) \|_{L^2(\mathbb{T}^d)}^2 ds \\
&+ C(\phi) \frac{1}{h} \mathbb{E} \int_t^{t+h} \| I_{\{[\beta/2,\beta]\}}(\bar{\rho}) \sigma(\bar{\rho}) \|_{L^2(\mathbb{T}^d)}^2 ds \\
\leq& C(\phi)\beta + C(\phi) \frac{1}{h} \mathbb{E} \int_t^{t+h} \| \rho_{\varepsilon} - \bar{\rho} \|_{L^1(\mathbb{T}^d)} ds.
\end{align*}

By Lemma~\ref{LLN-irregular}, for every $\delta'>0$,
\begin{align*}
\frac{1}{h} \mathbb{E} \int_t^{t+h} \|\rho_{\varepsilon} - \bar{\rho}\|_{L^1(\mathbb{T}^d)} ds \leq& \int_{\{\|\rho_{\varepsilon} - \bar{\rho}\|_{L^{\infty}([0,T];L^1(\mathbb{T}^d))} > \delta'\}} \|\rho_{\varepsilon} - \bar{\rho}\|_{L^{\infty}([0,T];L^1(\mathbb{T}^d))} d\mathbb{P} \\
&+ \int_{\{\|\rho_{\varepsilon} - \bar{\rho}\|_{L^{\infty}([0,T];L^1(\mathbb{T}^d))} \leq \delta'\}} \|\rho_{\varepsilon} - \bar{\rho}\|_{L^{\infty}([0,T];L^1(\mathbb{T}^d))} d\mathbb{P} \\
\leq& 2\|\rho_0\|_{L^1(\mathbb{T}^d)} \mathbb{P}\left( \|\rho_{\varepsilon} - \bar{\rho}\|_{L^{\infty}([0,T];L^1(\mathbb{T}^d))} > \delta' \right) + \delta' \to \delta',
\end{align*}
as $\varepsilon \to 0$. Hence,
$$
\lim_{\varepsilon \to 0} \liminf_{h \to 0} \frac{1}{h} \mathbb{E} \int_t^{t+h} \|\rho_{\varepsilon} - \bar{\rho}\|_{L^1(\mathbb{T}^d)} ds = 0.
$$

For $I_{4,1,2}$, by the structure of $\eta_{\beta}'$ and the regularity of $\rho_{\varepsilon}$ and $\bar{\rho}$, we deduce
\begin{align*}
I_{4,1,2} \leq& C(\phi) \frac{1}{h} \mathbb{E} \int_t^{t+h} \left\| I_{\{[\beta/2,\beta]\}}(\rho_{\varepsilon}) \nabla\sqrt{\rho_{\varepsilon}} \right\|_{L^2(\mathbb{T}^d)}^2 ds + C(\phi) \frac{1}{h} \mathbb{E} \int_t^{t+h} \left\| I_{\{[\beta/2,\beta]\}}(\bar{\rho}) \nabla\sqrt{\bar{\rho}} \right\|_{L^2(\mathbb{T}^d)}^2 ds.
\end{align*}
Using the choice $\beta = h^{1-\varepsilon'}$ and similarly by \cite[Lemma 7]{FG23}, we find
\begin{align*}
\liminf_{h \to 0} \Big[ C(\phi) \frac{1}{h} \mathbb{E} \int_t^{t+h} &\left\| I_{\{[h^{1-\varepsilon'}/2, h^{1-\varepsilon'}]\}}(\rho_{\varepsilon}) \nabla\sqrt{\rho_{\varepsilon}} \right\|_{L^2(\mathbb{T}^d)}^2 ds\\
 &+ C(\phi) \frac{1}{h} \mathbb{E} \int_t^{t+h} \left\| I_{\{[h^{1-\varepsilon'}/2, h^{1-\varepsilon'}]\}}(\bar{\rho}) \nabla\sqrt{\bar{\rho}} \right\|_{L^2(\mathbb{T}^d)}^2 ds \Big] = 0.
\end{align*}

For $I_{4,2}$, using the chain rule, we observe that
\begin{align*}
\nabla[\nabla(\phi\eta_{\beta}(\bar{\rho}))\sigma(\bar{\rho})]
=& \nabla[\nabla\phi\,\eta_{\beta}(\bar{\rho})\,\sigma(\bar{\rho}) + \phi\,\eta_{\beta}'(\bar{\rho})\,\nabla\bar{\rho}\,\sigma(\bar{\rho})] \\
\leq& \nabla^{\otimes2}\phi\,\eta_{\beta}(\bar{\rho})\,\sigma(\bar{\rho}) + \nabla\phi\,\eta_{\beta}'(\bar{\rho})\,\otimes\nabla\bar{\rho}\,\sigma(\bar{\rho}) + \nabla\phi\,\eta_{\beta}(\bar{\rho})\,\otimes\nabla\sigma(\bar{\rho}) \\
&+ \nabla\phi\,\eta_{\beta}'(\bar{\rho})\,\otimes\nabla\bar{\rho}\,\sigma(\bar{\rho}) + \phi\,\eta_{\beta}''(\bar{\rho})\,\nabla\bar{\rho}\otimes\nabla\bar{\rho}\,\sigma(\bar{\rho}) + \phi\,\eta_{\beta}'(\bar{\rho})\,\nabla[\nabla\bar{\rho}\sigma(\bar{\rho})].
\end{align*}
Under the assumptions $\rho_0 \geq c > 0$ and $\rho_0$ smooth, the above quantity belongs to $L^{\infty}([0,T];L^{\infty}(\mathbb{T}^d))$. Consequently,
\begin{align*}
I_{4,2} \leq& \operatorname*{ess\,sup}_{t \in [0,T]} \int_{\mathbb{T}^d} \left| \int_{\mathbb{T}^d} \left| \delta(\varepsilon)^{-d} \eta\left(\frac{x-y}{\delta(\varepsilon)}\right) \nabla(\phi\,\eta_{\beta}(\bar{\rho}))\,\sigma(\bar{\rho})(x) - \nabla(\phi\,\eta_{\beta}(\bar{\rho}))\,\sigma(\bar{\rho})(y) \right| dy \right|^2 dx \\
\lesssim& \left\| \nabla[\nabla(\phi\,\eta_{\beta}(\bar{\rho}))\,\sigma(\bar{\rho})] \right\|_{L^{\infty}([0,T];L^{\infty}(\mathbb{T}^d))}^2 \int_{\mathbb{T}^d} \left| \int_{\mathbb{T}^d} \eta_{\delta(\varepsilon)}(x-y)\,|x-y|\,dy \right|^2 dx \\
\leq& C(\rho_0)\,\delta(\varepsilon)^2.
\end{align*}

For $I_{4,3}$, using the choice $\beta = h^{1 - \varepsilon'}$, we clearly have
$$
\lim_{h \to 0} I_{4,3} = 0.
$$

For $I_{4,4}$, by the differentiability of the function $ t \mapsto \int_0^t \|\nabla\phi\,\sigma(\bar{\rho}(s))\|_{L^2(\mathbb{T}^d)}^2 ds $, we conclude that
$$
\lim_{h \to 0} I_{4,4} = 0.
$$

\bigskip

\noindent\textbf{Estimate of $ I_5 $.} By integration by parts, we obtain
\begin{align*}
I_5=& -\frac{\varepsilon^{1/2}}{2}F_{1,\delta(\varepsilon)} \int_t^{t+h} \int_{\mathbb{T}^d} \nabla\phi\,\eta_{\beta}(\rho_{\varepsilon}) \cdot \sigma'(\rho_{\varepsilon})^2\,\nabla\rho_{\varepsilon}\,dx\,ds \\
=& \frac{\varepsilon^{1/2}}{2}F_{1,\delta(\varepsilon)} \int_t^{t+h} \int_{\mathbb{T}^d} \nabla\cdot(\nabla\phi\,\eta_{\beta}(\rho_{\varepsilon})) \left( \int_0^{\rho_{\varepsilon}} \sigma'(\zeta)^2\,d\zeta \right) dx\,ds. 
\end{align*}

Hence, using the compact support of $\eta_{\beta}$ and the property $\sigma'(\zeta) \lesssim \zeta^{-1/2}$, we deduce
\begin{align*}
\frac{1}{h}\mathbb{E}(I_5)^2\lesssim& \frac{1}{h} \mathbb{E} \left( \frac{\varepsilon^{1/2}}{2}F_{1,\delta(\varepsilon)} \int_t^{t+h} \int_{\mathbb{T}^d} \nabla\cdot(\nabla\phi\,\eta_{\beta}(\rho_{\varepsilon})) \left( \int_{\beta/2}^{\rho_{\varepsilon}} \sigma'(\zeta)^2\,d\zeta \right) dx\,ds \right)^2 \\
\lesssim& \varepsilon\,F_{1,\delta(\varepsilon)}^2\,C(\phi)\left[ h\beta^{-1/2} (\log\beta - \log\frac{\beta}{2}) \mathbb{E}\|\nabla\sqrt{\rho_{\varepsilon}}\|_{L^2(\mathbb{T}^d)}^2 \right] \\
\leq& \varepsilon\,F_{1,\delta(\varepsilon)}^2\,C(\phi,\rho_0)\left[ h\beta^{-1/2} (\log\beta - \log\frac{\beta}{2}) \right].
\end{align*}

Using the choice $\beta = \beta(h) = h^{1 - \varepsilon'}$, we see that both
$$
h\beta(h)^{-1/2}\Big(\log\beta(h) - \log\frac{\beta(h)}{2}\Big) \to 0 \quad \text{and} \quad \beta(h)^2 h^{-1} \to 0, \quad \text{as } h \to 0.
$$
Return back to \eqref{expansion}, we get
\begin{align*}
&\lim_{\varepsilon\rightarrow0}\limsup_{h\rightarrow0}\frac{1}{h}\mathbb{E}\Big(\varepsilon^{-1/2}\langle\rho_{\varepsilon}(t+h)-\rho_{\varepsilon}(t),\phi\rangle\Big)^2\\
\leq&\limsup_{\varepsilon\rightarrow0}	\Big(C(\rho_0)\varepsilon F_{1,\delta(\varepsilon)}+C(\rho_0)\delta(\varepsilon)^2\Big)+\langle \sigma(\bar{\rho}(t))^2 \nabla \phi, \nabla \phi \rangle\\
=&\langle \sigma(\bar{\rho}(t))^2 \nabla \phi, \nabla \phi \rangle.
\end{align*}
Consequently, we conclude that 
\begin{align*}
&\lim_{\varepsilon\rightarrow0}\limsup_{h\rightarrow0}\Big|\frac{1}{h} \, \mathbb{E} \left( \left\langle \varepsilon^{-1/2}(\rho_{\varepsilon} - \bar{\rho})(t + h) - \varepsilon^{-1/2}(\rho_{\varepsilon} - \bar{\rho})(t), \phi \right\rangle \right)^2-\langle \sigma(\bar{\rho}(t))^2 \nabla \phi, \nabla \phi \rangle\Big|\\ 
\lesssim& \lim_{\varepsilon\rightarrow0}\limsup_{h\rightarrow0}\Big|\frac{1}{h} \, \mathbb{E} \left( \left\langle \varepsilon^{-1/2} \rho_{\varepsilon}(t + h) - \varepsilon^{-1/2} \rho_{\varepsilon}(t), \phi \right\rangle \right)^2 
+ C(\rho_0) \varepsilon^{-1} h-\langle \sigma(\bar{\rho}(t))^2 \nabla \phi, \nabla \phi \rangle\Big|\\
=&0.
\end{align*}
This completes the proof. 
\end{proof}

Similarly, for the case of Dean-Kawasaki equation, the identity result can be obtained by using the same approach.
\begin{theorem}
Let $\sigma(\zeta)=\sqrt{\zeta(1-\zeta)}$. Assume that $\rho_0$ satisfies Assumption \ref{Assump-initial-1}. Let $\rho_{\varepsilon}$ be the renormalized kinetic solution of \eqref{SPDE-irregular}. Then, under the scaling regime $$\lim_{\varepsilon\rightarrow0}\varepsilon\delta(\varepsilon)^{-d-2}=0,$$ for every test function $\phi \in C^{\infty}(\mathbb{T}^d)$ and every $t \in [0,T]$, the following asymptotic fluctuation identity holds:
\begin{align*}
	\lim_{\varepsilon \to 0} \liminf_{h \to 0} \frac{1}{h} \, \mathbb{E} \left( \left\langle \varepsilon^{-1/2}(\rho_{\varepsilon} - \bar{\rho})(t + h) - \varepsilon^{-1/2}(\rho_{\varepsilon} - \bar{\rho})(t), \phi \right\rangle \right)^2 = \langle \sigma(\bar{\rho}(t))^2 \nabla \phi, \nabla \phi \rangle.
\end{align*}
\end{theorem}
\begin{proof}
	The proof essentially follows the same strategy as that of Theorem~\ref{asymptotic-behavior}. However, since the coefficient $\sqrt{\zeta(1-\zeta)}$ now possesses two singular points, we need to introduce an additional truncation near $\zeta = 1$. More precisely, by applying the same argument as in the proof of Lemma~\ref{boundedness-kineticmeasure}, we obtain
\begin{align*}
\mathbb{E}\left(h^{-1 + \varepsilon'} 
p_{\varepsilon}\!\left([t, t + h] \times \mathbb{T}^d \times 
\left[\tfrac{h^{1 - \varepsilon'}}{2}, h^{1 - \varepsilon'}\right]
\cup
\left[1-h^{1 - \varepsilon'}, 1-\tfrac{h^{1 - \varepsilon'}}{2}\right]
\right)\right)^2 
\leq C(\rho_0),
\end{align*}
for all $t \in [0, T]$, $\varepsilon' \in (0, 1/4)$, $\varepsilon \in (0, 1)$, and $h \in (0, 1/4)$. Moreover, we introduce a truncation function to localize both singularities at $0$ and $1$. Fix $\beta \in (0, 1/4)$, and let $\eta_{\beta}$ be a smooth truncation function satisfying $\eta_{\beta}(\zeta) = 1$ for $\zeta \in [\beta,1-\beta]$, $\eta_{\beta}(\zeta) = 0$ for $\zeta \in [0,\beta/2] \cup [1-\beta/2,1]$, and 
$$
\eta'_{\beta} \lesssim \beta^{-1} I_{\left\{[\tfrac{\beta}{2}, \beta]\cup[1-\beta,1-\tfrac{\beta}{2}]\right\}}.
$$
Let $\chi_{\varepsilon} = I_{\{0 < \zeta < \rho_{\varepsilon}\}}$ denote the kinetic function. We then decompose
\begin{align*}
&\varepsilon^{-1/2} \left\langle \rho_{\varepsilon}(t + h) - \rho_{\varepsilon}(t), \phi \right\rangle \notag\\
=&\; \varepsilon^{-1/2} \left\langle \chi_{\varepsilon}(t + h) - \chi_{\varepsilon}(t), \eta_{\beta} \phi \right\rangle_{L^2(\mathbb{T}^d\times\mathbb{R}_+)}
+ \varepsilon^{-1/2} \left\langle \chi_{\varepsilon}(t + h) - \chi_{\varepsilon}(t), (1 - \eta_{\beta}) \phi \right\rangle_{L^2(\mathbb{T}^d\times\mathbb{R}_+)}.
\end{align*}
With this decomposition at hand, the argument of Theorem~\ref{asymptotic-behavior} can be adapted to yield the desired result.

\end{proof}

\noindent{\bf  Acknowledgements}\quad Research was sponsored by the Army Research Office and was accomplished under Grant Number W91INF-23-1-0230. The views and conclusions contained in this document are those of the authors and should not be interpreted as representing the official policies, either expressed or implied, of the Army Research Office or the U.S. Government. The U.S. Govemment is authorized to reproduce and distribute reprints for Government purposes notwithstanding any copyright notation herein. ND acknowledges the kind hospitality of TU Munich. 

We would like to thank Mario Ayala for helpful discussions.

\bibliographystyle{alpha}
\bibliography{errorestimates}

\end{document}